\documentclass[12pt]{article}

\usepackage{mathtools}
\usepackage{amssymb}
\usepackage{amsthm}
\usepackage{tikz}
\usepackage{fullpage}
\usepackage{url}

\DeclareSymbolFont{bbold}{U}{bbold}{m}{n}
\DeclareSymbolFontAlphabet{\mathbbold}{bbold}

\newcommand{\N}{\mathbb{N}}

\newcommand{\R}{\mathbb{R}}

\newcommand{\1}{\mathbbold{1}}

\renewcommand{\c}{{\rm c}}
\renewcommand{\epsilon}{\varepsilon}

\DeclareMathAccent{\Circ}{\mathalpha}{operators}{"17}
\newcommand{\interior}[1]{\operatorname{\Circ{#1}}}



\DeclareMathOperator{\spt}{spt}

\DeclareMathOperator{\sym}{sym}
\DeclareMathOperator{\trf}{trf}
\DeclareMathOperator{\dive}{div}
\DeclareMathOperator{\curl}{curl}
\DeclareMathOperator{\grad}{grad}
\DeclareMathOperator{\Grad}{Grad}
\DeclareMathOperator{\Dive}{Div}
\DeclareMathOperator{\Curl}{Curl}
\DeclareMathOperator{\He}{He}

\DeclareMathOperator{\kar}{ker}
\DeclareMathOperator{\rge}{ran}
\DeclareMathOperator{\dom}{dom}

\renewcommand{\Re}{\operatorname{Re}}

\DeclareMathOperator{\lin}{lin}

\let\eps\varepsilon
\let\phi\varphi

\let\leq\leqslant
\let\ge\geqslant
\let\geq\geqslant

\makeatletter
\def\@row#1,{#1\@ifnextchar;{\@gobble}{&\@row}}
\def\@matrix{%
    \expandafter\@row\my@arg,;%
    \@ifnextchar({\\ \get@in@paren{\@matrix}}{\after@matrix}%
    }
\def\matrixtype#1#2#3{%
    \ifmmode\def\after@matrix{\end{#2}\right#3}%
    \else\def\after@matrix{\end{#2}\right#3$}$\fi
    \left#1\begin{#2}\get@in@paren{\@matrix}%
    }
\def\@column#1,{#1\@ifnextchar;{\@gobble}{\\ \@column}}
\newcommand\vect{}
\def\svect(#1){\left(\begin{smallmatrix}\@column#1,;\end{smallmatrix}\right)}
\def\vect{\get@in@paren{\@vect}}
\def\@vect{\left(\begin{matrix}\expandafter\@column\my@arg,;\end{matrix}\right)}
\def\get@in@paren#1({\def\my@arg{}\def\my@rest{}\def\after@get{#1}\get@arg}
\let\e@a\expandafter
\def\get@arg#1){\e@a\kl@test\my@rest#1(;}
\def\kl@test#1(#2;{\e@a\def\e@a\my@arg\e@a{\my@arg#1}%
                   \ifx:#2:\let\my@exec\after@get
                   \else\let\my@exec\get@arg
                        \e@a\def\e@a\my@arg\e@a{\my@arg(}%
                        \def@rest#2;%
                   \fi\my@exec}
\def\def@rest#1(;{\def\my@rest{#1\kl@zu}}
\def\kl@zu{)}

\makeatletter
\newcommand\MyPairedDelimiter{%
  \@ifstar{\My@Paired@Delimiter{{}}}
          {\My@Paired@Delimiter{}}%
}
\newcommand\My@Paired@Delimiter[4]{%
  \newcommand#2{%
    \@ifstar{\start@PD{#1}{\delimitershortfall=-1sp}{#3}{#4}}
            {\start@PD{#1}{}{#3}{#4}}%
  }%
}
\newcommand\start@PD[5]{%
  #1\mathopen{\mathpalette\put@delim@helper{\put@delim{#2}{#3}{.}{#5}}}%
  #5%
  \mathclose{\mathpalette\put@delim@helper{\put@delim{#2}{.}{#4}{#5}}}%
}
\newcommand\put@delim@helper[2]{%
  \hbox{$\m@th\nulldelimiterspace=0pt #2#1$}%
}
\newcommand\put@delim[5]{%
  \setbox\z@\hbox{$\m@th#5{#4}$}%
  \setbox\tw@\null
  \ht\tw@\ht\z@ \dp\tw@\dp\z@
  #1#5%
  \left#2\box\tw@\right#3%
}

\makeatother
\MyPairedDelimiter*{\abs}{\lvert}{\rvert}
\MyPairedDelimiter*{\norm}{\lVert}{\rVert}
\MyPairedDelimiter{\set}{\{}{\}}

\theoremstyle{plain} 
\newtheorem{theorem}{Theorem}[section]
\newtheorem{corollary}[theorem]{Corollary}
\newtheorem{lemma}[theorem]{Lemma}
\newtheorem{proposition}[theorem]{Proposition}
\theoremstyle{definition}
\newtheorem{example}[theorem]{Example}

\newtheorem*{definition}{Definition}
\newtheorem{remark}[theorem]{Remark}

\usepackage{enumitem}

\setenumerate[1]{nolistsep} 
\setenumerate[2]{nolistsep} 

\begin{document}

\title{Nonlocal $H$-convergence}

\author{Marcus Waurick}

\date{}

\maketitle

\begin{abstract} We introduce the concept of nonlocal $H$-convergence. For this we employ the theory of abstract closed complexes of operators in Hilbert spaces. We show uniqueness of the nonlocal $H$-limit as well as a corresponding compactness result. Moreover, we provide a characterisation of the introduced concept, which implies that local and nonlocal $H$-convergence coincide for multiplication operators. We provide applications to both nonlocal and nonperiodic fully time-dependent 3D Maxwell's equations on rough domains. The material law for Maxwell's equations may also rapidly oscillate between eddy current type approximations and their hyperbolic non-approximated counter parts. Applications to models in nonlocal response theory used in quantum theory and the description of meta-materials, to fourth order elliptic problems as well as to homogenisation problems on Riemannian manifolds are provided.
\end{abstract}



\medmuskip=4mu plus 2mu minus 3mu
\thickmuskip=5mu plus 3mu minus 1mu
\belowdisplayshortskip=9pt plus 3pt minus 5pt

\section{Introduction}\label{s:int}

The theory of homogenisation studies the asymptotic properties of heterogeneous materials with a macroscopic and a microscopic scale for the fictitious limit of the ratio of microscopic over macroscopic scale tending to $0$. When one is to model this problem mathematically, the mentioned ratio is introduced with a parameter say $\epsilon=1/n$, $n\in\mathbb{N}$. For any $n\in\mathbb{N}$ one is then given a partial differential equation, e.g.,
\begin{equation}\label{eq:stex}
   -\dive a_n \grad u_n = f
\end{equation}
for fixed $f\in H^{-1}(\Omega)$, $\Omega\subseteq \mathbb{R}^d$ open and bounded, $u_n\in H_0^1(\Omega)$ and $a_n \in L(L^2(\Omega)^d)$ (i.e., a bounded linear operator from $L^2(\Omega)^d$ into $L^2(\Omega)^d$) satisfying $\Re \langle a_n \phi,\phi\rangle\geq \alpha\langle \phi,\phi\rangle$ for all $n\in\mathbb{N}$ and $\phi\in L^2(\Omega)^d$ and some $\alpha>0$. Then one addresses the question, whether the (uniquely) determined sequence of solutions $(u_n)_n$ has a (weak) limit. Assuming that $u_n\rightharpoonup u$ weakly in $H_0^1(\Omega)$, one furthermore asks, whether there exists $a\in L(L^2(\Omega)^d)$ (independent of $f$) such that
\begin{equation}\label{eq:stexl}
    -\dive a \grad u = f.
\end{equation}

There is a vast amount of literature concerning this or related subjects. We shall only refer to the standard references \cite{Bensoussan1978,Jikov1994,Cioranescu1999,Tartar2009} for some introductory material. In almost all discussions of the subject, the attention is restricted to local coefficient sequences $(a_n)_n$ (in this sense the approach in \cite{Fish2002} is still considered to be \emph{local}), that is, one focusses on multiplication operators being elements of the set 
\begin{multline}\label{eq:Mab}
   M(\alpha,\beta,\Omega)\coloneqq \{ a \in L^\infty(\Omega)^{d\times d}; \forall \xi \in \mathbb{C}^d: \alpha \|\xi\|^2\leq \Re \langle a(x)\xi,\xi\rangle,\\  \Re \langle a(x)^{-1}\xi,\xi\rangle \geq \tfrac{1}{\beta}\|\xi\|^2\text{ for a.e.~}x\in\Omega\}
\end{multline}
for some $0<\alpha<\beta$.

Particularly focussing on the model problem \eqref{eq:stex}, Tartar and Murat have introduced and studied the notion of $H$-convergence (see also \cite{Murat1997}), which we call local $H$-convergence in order to avoid possible misunderstandings later on. The notion reads as follows.

\begin{definition}[local $H$-convergence, {{\cite[Section 5]{Murat1997}, \cite[Definition 6.4]{Tartar2009}}}]
  A sequence $(a_n)_n$ in $M(\alpha,\beta,\Omega)$ is said to be \emph{locally $H$-convergent} to $a\in M(\alpha,\beta,\Omega)$, if the following conditions hold: For all $f\in H^{-1}(\Omega)=H_0^1(\Omega)^*$ and $(u_n)_n$ in $H_0^1(\Omega)$ given by \eqref{eq:stex}, we obtain
  \begin{itemize}
   \item $(u_n)_n$ weakly converges in $H_0^1(\Omega)$ to some $u\in H_0^1(\Omega)$,
   \item $ a_n\grad u_n \rightharpoonup a\grad u$,
   \item $-\dive a \grad u =f$.
  \end{itemize}
  $a$ is called \emph{local $H$-limit} of $(a_n)_n$. 
\end{definition}

Some by now standard properties of local $H$-convergence have been shown by their inventors. For instance, it is possible to associate a topology $\tau_{\textnormal{loc}H}$ with the above notion of local $H$-convergence (see \cite[p 82]{Tartar2009}). We shall state a remarkable property of this topology:
\begin{theorem}[see e.g.~{{\cite[Theorem 6.5]{Tartar2009}}}]\label{thm:TM} $(M(\alpha,\beta,\Omega),\tau_{\textnormal{loc}H})$ is a metrisable and (sequentially) compact Hausdorff space.
\end{theorem}

As a consequence of the latter theorem, the local $H$-limit is unique and any sequence $(a_n)_n$ in $M(\alpha,\beta,\Omega)$ has a locally $H$-convergent subsequence. The arguments used to show the latter result are based on localisation techniques. Further characterising properties for instance as the one in \cite[p 10]{Tartar1997} and concrete formulas for the limit $a$ in case of periodic coefficients use Tartar's method of oscillating test functions as well as the celebrated div-curl lemma (see \cite{Murat1978}). We shall also refer to the techniques in \cite{Jikov1994} or \cite{Cioranescu1999}, which are in turn local in nature.

In recent years the interest in so-called meta-materials has emerged. Although it is generally rather difficult to find a precise definition for meta-materials physicists have been dealing with this notion for quite a while for coining materials with properties that are not known for so-called `classical' materials. In fact, meta-materials do not occur in nature and have to be manufactured artificially. A subclass of these meta-materials are best described by non-local constitutive relations, where integral operators rather than multiplication operators are used as coefficients, see e.g.~\cite{Gorlach2016,Ciattoni2015,Mendez2017}. 

Other nonlocal constitutive relations can be found in nonlocal response theory related to quantum theory, see \cite[Chapter 10]{Keller2011}. Furthermore, we shall refer to the so-called McKean--Vlasov equations, see \cite{Carillo18} and Example \ref{ex:dp1nl} below. For an account on nonlocal elasticity we refer to the recent preprint \cite{Evgrafov2018}.

Also, if the oscillations of the coefficients are `perpendicular' to the differential operators occurring in the differential equation nonlocal effects result after a homogenisation process. For this we refer to \cite{Tartar1989,Waurick2014,W12_HO} as paradigmatic examples where ordinary differential equations with infinite-dimensional state space have been considered. We also refer to \cite{Wellander2001,W14_FE,W16_HPDE} where memory effects have been derived due to a homogenisation process. 

Nonlocal material models also occur, when homogenising materials with `soft' and `stiff' components, which in turn is modelled by non-uniform coercivity estimates in the coefficients with respect to $n$. A prominent example are equations with high-contrast or singular coefficients, see e.g.~\cite[eq. (4.3)]{Cherednichenko2006}. 

In certain cases nonlocal homogenisation procedures have been carried out, see e.g.~\cite{Gorlach2016,Ciattoni2015,Yvonnet2014,Tsukerman2017}. We shall also refer to \cite{Piatnitski2017,Du2016} for non-pde type homogenisation problems.

A general theory, however, describing highly oscillatory \emph{nonlocal} material models has been missing so far. Thus, the aim of the present article is to introduce the notion of \emph{nonlocal $H$-convergence}.  As mentioned above, the notion of nonlocal $H$-convergence will become important, when one analyses \emph{iterated} homogenisation schemes of local models that result in nonlocal limit models or, if one discusses homogenisation problems for certain meta-materials so that nonlocal partial differential equations occur right from the start. 

We shall argue that local $H$-convergence cannot capture nonlocal coefficients. Indeed, assume in \eqref{eq:stex} we allow for general $a_n\in L(L^2(\Omega)^d)$ satisfying (suitable) uniform coercivity and boundedness conditions. In order to be consistent with local $H$-convergence, the nonlocal $H$-convergence needs to coincide, when applied to sequences in $M(\alpha,\beta,\Omega)$. So, assume that $(a_n)_n$ in $L(L^2(\Omega)^d)$ locally $H$-converges to $a$, that is, apply the above definition to general operators in $L^2(\Omega)^d$. Let $b\in L(L^2(\Omega)^d)$ with $a=b$ on $\rge(\interior{\grad})= \{ q\in L^2(\Omega)^d; \exists u\in H_0^1(\Omega): \grad u = q\}$. Then $(a_n)_n$ locally $H$-converges to $b$, as well. Since $\rge(\interior{\grad})^{\bot}=\kar(\dive)=\{q\in L^2(\Omega)^d; \dive q=0\}$ is \emph{infinite-dimensional} as long as $d\geq 2$, we infer that local $H$-convergence is clearly not sufficient to uniquely identify nonlocal limit operators.

When introducing any notion of nonlocal $H$-convergence, we cannot expect properties of local $H$-convergence like independence of the attached boundary conditions (\cite[Lemma 10.3]{Tartar2009}) to carry over to nonlocal $H$-convergence. On the contrary, for the proper functional analytic setting the attached boundary conditions are of prime importance. We refer to Example \ref{ex:dpbc} below for the precise argument showing that nonlocal $H$-convergence \emph{depends} on the boundary conditions attached.

However, we shall obtain a result analogous to Theorem~\ref{thm:TM} for the newly introduced notion (see Proposition \ref{prop:T2} and Theorem \ref{thm:Hcom}), which is one of the main results of the present exposition. Furthermore, we shall show that on $M(\alpha,\beta,\Omega)$ nonlocal $H$-convergence and local $H$-convergence coincide (see Theorem \ref{thm:lHnlH}).

We provide an overview of the contents of this article, next.

In Section \ref{sec:nonH}, we will introduce nonlocal $H$-convergence. For the definition of nonlocal $H$-convergence, one observes that a certain elliptic problem with $\dive$ and $\grad$ both replaced by $\curl$ with appropriate boundary conditions leads to the same homogenised limit as for the original divergence form type equation (see \eqref{eq:stex}). Thus, quite naturally, for nonlocal $H$-convergence, we shall use the theory of closed complexes of operators in Hilbert spaces, which is a generalisation of the operators $\grad$, $\curl$, and $\dive$ and will be specified in Section \ref{sec:com}. In this section, we will also recall a more detailed version of the Lax--Milgram lemma (see Theorem \ref{thm:varwp2} and \cite{TW14_FE}), which is crucial for our later analysis. Note that the core observation that it is possible to formulate kernels of differential operators via the application of other differential operators has been employed already in the context of Picard's extended Maxwell system in order to discuss low-frequency asymptotics for the time-harmonic Maxwell's equations, see \cite{P84}.

The emergence of nonlocal or memory effects during the homogenisation process is rooted in the lack of continuity of the inversion mapping for linear operators in the weak operator topology, see \cite{W12_HO} and Proposition \ref{prop:invdis}. It is easy to see that also multiplication is not jointly continuous in the weak operator topology either. However, a suitable combination of projection, multiplication and inversion of the operator sequence $(a_n)_n$ does \emph{characterise} nonlocal $H$-convergence. This is the subject of Section \ref{sec:block} with its main result Theorem \ref{thm:chH}. 

The results of Section \ref{sec:block} will be used in order to obtain the announced variant of Theorem~\ref{thm:TM} in the context of nonlocal $H$-convergence. From the compactness statement for nonlocal $H$-convergence, we may then deduce Theorem~\ref{thm:lHnlH} -- the relationship of local and nonlocal $H$-convergence. This in turn yields a homogenisation result for static Maxwell type equations under the hypothesis of $H$-convergence for local coefficients, see Corollary~\ref{cor:lHnlH}, which is interesting on its own.

Using the global div-curl lemma obtained in \cite{W17_DCL}, we provide a characterisation of nonlocal $H$-convergence in terms of (abstract) `div-curl quantities' in Section \ref{sec:dct}. This characterisation is an abstract variant of \cite[Lemma 4.5]{Jikov1994} and should be remindful of \cite[p. 10]{Tartar1997}. Note that the main result of Section \ref{sec:dct}, Theorem \ref{thm:chH3}, provides a nice way of practically computing the nonlocal $H$-limit in applications. We will use Theorem \ref{thm:chH3} for the computation of the nonlocal $H$-limit for a linear variant of the McKean--Vlasov equation, see Example \ref{ex:mcv3}.

The range of applicability of the main theoretical results is further touched upon in the two concluding Sections \ref{sec:max} and \ref{sec:mex}. In Section \ref{sec:max} we shall revisit some ideas from \cite{W16_HPDE} and discuss a homogenisation problem for the fully time-dependent, 3D Maxwell's equations. In fact, the main result of Section \ref{sec:max} generalises the main results in \cite{Barbatis2003,Wellander2001} to both non-periodic and nonlocal (in both space and time) settings. We note that non-uniformly dielectric media as occurring for eddy current type approximations are admitted in the general homogenisation scheme. In fact, the underlying media may even rapidly oscillate between strictly positive and vanishing dielectricity on different spatial domains. This oscillatory behaviour between hyperbolic and parabolic type problems has only recently been accessible for $1+1$-dimensional periodic model problems, see \cite{CW17_1D,FW17_1D,W16_SH}.

In Section \ref{sec:mex} we will provide applications to a homogenisation problem of fourth order and an adapted perspective to nonlocal homogenisation on Riemannian manifolds. The latter provides the nonlocal counterpart of \cite{Hoppe2017}.
\section{On closed operator complexes and abstract elliptic pdes}\label{sec:com}

Throughout this section, we let $H_0$, $H_1$, $H_2$ be Hilbert spaces. Furthermore, we let \begin{align*}
A_0&\colon \dom(A_0)\subseteq H_0\to H_1, \\
A_1&\colon \dom(A_1)\subseteq H_1\to H_2\end{align*} be densely defined and closed linear operators.

\begin{definition}
  We say that $(A_0,A_1)$ is a \emph{complex} or \emph{sequence}, if $\rge(A_0)\subseteq \kar(A_1)$. We call a complex $(A_0,A_1)$ \emph{closed}, if both $\rge(A_0)\subseteq H_1$ and $\rge(A_1)\subseteq H_2$ are closed. A complex $(A_0,A_1)$ is \emph{exact}, if $\rge(A_0)=\kar(A_1)$. A complex $(A_0,A_1)$ is called \emph{compact}, if $\dom(A_0^*)\cap \dom(A_1)\hookrightarrow H_1$ compactly.
\end{definition}

For short reference, we shall often address `exact' for complexes, just by saying `$(A_0,A_1)$ is exact' and imply the meaning `$(A_0,A_1)$ is an exact complex' (similarly for `compact' and `closed').

We recall some elementary properties of the theory of complexes of operators in Hilbert spaces, which we state without proof. We refer to \cite[Section 2]{Pauly2016} for the proofs. The assertions, however, follow from the closed range theorem (see e.g.~\cite[Corollary 2.5]{TW14_FE}) and the orthogonal decomposition $H_0=\kar(C)\oplus \overline{\rge}(C^*)$ for $C\colon \dom(C)\subseteq H_0\to H_1$ densely defined, closed. The assertion relating compactness follows from the fact that compact operators are compact if and only if their adjoints are. Moreover, the last statement follows from a contradiction argument and the fact that compact unit balls characterise finite-dimensionality.   

\begin{proposition}\label{prop:comelm} \begin{enumerate}
                                        \item $(A_0,A_1)$ is a complex if and only if $(A_1^*,A_0^*)$ is a complex;
                                        \item $(A_0,A_1)$ is closed if and only if $(A_1^*,A_0^*)$ is closed;
                                        \item Assume $(A_0,A_1)$ is closed. Then $(A_0,A_1)$ is exact if and only if $(A_1^*,A_0^*)$ is exact;
                                        \item $(A_0,A_1)$ is compact if and only if $(A_1^*,A_0^*)$ is compact;
                                        \item Let $(A_0,A_1)$ be compact. Then $(A_0,A_1)$ is closed and $\kar(A_0^*)\cap \kar(A_1)$ is finite-dimensional.
                                       \end{enumerate}
\end{proposition}

Before we treat differential and, thus, particularly, unbounded operators, we shall state a rather trivial example of an exact complex.

\begin{example}\label{ex:tr} Denote by $H$ a Hilbert space and let $\{0\}=\lin\emptyset$ be the trivial Hilbert space consisting of $0$, only. Let $\iota_{0} \colon \lin\emptyset \to H, 0\mapsto 0$ and $1 \colon H \to H, \phi\mapsto \phi$. With the setting $H_0=\{0\}$, $H_1=H$, $H_2=H$ together with $A_0=\iota_{0}$ and $A_1=1$, we are in the situation of the beginning of this section. Indeed, since $A_0$ and $A_1$ are bounded linear operators, they are densely defined and closed. Moreover, their ranges are closed and $\rge(A_0)=\{0\}=\kar(A_1)$, so that $(A_0,A_1)$ is exact. $A_1$ is obviously self-adjoint and $A_0^*=\iota_{0}^*$ is the (orthogonal) projection onto $\{0\}$. By Proposition \ref{prop:comelm} (or direct verification), $(A_1^*,A_0^*)$ is closed and exact, as well. 
\end{example}

For the time being, we focus on the 3-dimensional model case. Note that, however, the theory carries over to the higher-dimensional setting. For this, we refer for instance to \cite[Theorem 3.5]{W17_DCL} for an account on higher-dimensional situations. Other examples are treated in Section \ref{sec:mex}. Note that exactness of the considered complexes is an incarnation of Poincar\'e's lemma (see also Section \ref{sec:mex} below).
\begin{example}\label{ex:diffo} Let $\Omega\subseteq\mathbb{R}^3$ open. We define
\begin{align*}
   \grad_\text{c} & \colon C_c^\infty(\Omega) \subseteq L^2(\Omega) \to L^2(\Omega)^3, \phi\mapsto (\partial_j\phi_j)_{j\in\{1,2,3\}},
   \\ \dive_\text{c} & \colon C_c^\infty(\Omega)^3 \subseteq L^2(\Omega)^3 \to L^2(\Omega), (\phi_j)_{j\in\{1,2,3\}}\mapsto \sum_{j=1}^d\partial_j\phi_j,
   \\ \curl_{\text{c}} & \colon C_c^\infty(\Omega)^{3} \subseteq L^2(\Omega)^{3} \to L^2(\Omega)^{3}, (\phi_{j})_{j\in\{1,2,3\}}\mapsto \left(\begin{smallmatrix} \partial_2\phi_3 - \partial_3\phi_2 \\ \partial_3\phi_1-\partial_1\phi_3 \\ \partial_1\phi_2-\partial_2\phi_1   
                                                                                                                     \end{smallmatrix}\right).
\end{align*}
We set $\interior{\grad}\coloneqq \overline{\grad}_\text{c}$ and, similarly, $\interior{\dive},\interior{\curl}$. Furthermore, we put $\dive\coloneqq -\interior{\grad}^*$, $\grad\coloneqq -\interior{\dive}^*$, and $\curl\coloneqq \interior{\curl}^*$.

Before we state several examples of complexes, we shall highlight the domains of the operators introduced and the differences between them. It is almost immediate from the definition of the adjoint and the distributional gradient that we have
\[
   \dom(\grad) = H^1(\Omega).
\] Since the domain of $\interior{\grad}$ is the closure of $C_c^\infty(\Omega)$ with respect to the $H^1(\Omega)$-scalar product, we obtain that
\[
   \dom(\interior{\grad}) = H_0^1(\Omega),
\]
which, in turn, for $\Omega$ with Lipschitz continuous boundary (so that the boundary trace $\gamma \colon H^1(\Omega) \to H^{1/2}(\partial \Omega), u\mapsto u|_{\partial\Omega}$ is a well-defined, continuous operator) reads
\[
  \dom(\interior{\grad}) = \{ u\in H^1(\Omega); \gamma(u)= 0\}.
\]
Similarly, we obtain
\[
   \dom(\curl) = \{ u\in L^2(\Omega)^3; \curl u \in L^2(\Omega)^3\} \eqqcolon H(\curl,\Omega).
\]
Again, if we restrict ourselves to the setting of $\Omega$ with Lipschitz boundary, we may define the tangential trace by (continuous extension of) $\gamma_{\times} : H^1(\Omega)\subseteq H(\curl,\Omega) \to H^{-1/2}(\partial\Omega), u\mapsto \gamma(u)\times \mathrm{n}$, where $\mathrm{n}$ denotes the unit outward normal of $\partial\Omega$, which exists almost everywhere, see also \cite{Buffa2002}. In particular, we obtain
\[
   \dom(\interior{\curl}) = \{ u \in H(\curl,\Omega); \gamma_{\times}(u)=0\} \eqqcolon H_0(\curl,\Omega).
\]
Finally, by definition, we obtain
\[
   \dom(\dive) = \{ u\in L^2(\Omega)^3; \dive u \in L^2(\Omega)\} \eqqcolon H(\dive,\Omega).
\]
Similarly, we obtain for $\Omega$ admitting a strong Lipschitz boundary $\partial\Omega$ with unit outward normal $\mathrm{n}$ that using the normal trace operator $\gamma_\textrm{n} \colon H^1(\Omega)\subseteq H(\dive,\Omega)\to H^{-1/2}(\partial\Omega), q\mapsto \textrm{n}\cdot q$ again obtained by continuous extension. With this we may also write
\[
   \dom (\interior{\dive})= \{ u\in H(\dive,\Omega); \gamma_\textrm{n}(u)=0\}\eqqcolon H_0(\dive,\Omega).
\]
\begin{enumerate}
                \item[(a1)] If $\Omega$ is bounded in one direction, then, by Poincar\`e's inequality, $(\iota_0,\interior{\grad})$, where $\iota_0 \colon \{0\}\hookrightarrow L^2(\Omega)$, is closed and exact (here $H_0=\{0\}$, $H_1=L^2(\Omega)$ and $H_2=L^2(\Omega)^3$). Consequently, by Proposition \ref{prop:comelm}, so is $(\dive,\iota_0^*)$. In particular, this implies that $\dive$ maps onto $L^2(\Omega)$.
                \item[(b1)] If $\Omega$ is bounded with continuous boundary, by the Rellich--Kondrachov theorem, the complex $(\iota_0,\grad )$ is compact (here $H_0=\{0\}$, $H_1=L^2(\Omega)$ and $H_2=L^2(\Omega)^3$). The same applies to $(\interior{\dive},\iota_0^*)$, by Proposition \ref{prop:comelm}.
\end{enumerate}

For the next examples we refer to \cite{Bauer2016} for the asserted compactness properties as a general reference.  We shall also refer to the references therein for a guide to the literature.
\begin{enumerate} 
                \item[(a2)] If $\Omega$ is a bounded weak Lipschitz domain, that is, if $\Omega$ is a Lipschitz manifold, then $(\interior{\grad},\interior{\curl})$ is compact (here $H_0=L^2(\Omega)$, $H_1=H_2=L^2(\Omega)^3$). In particular, so is $({\curl},\dive)$ (Weck's selection theorem, see also \cite{Weck1974} or \cite{Picard1984}).
                \item[(b2)] If $\Omega$ is a bounded weak Lipschitz domain, then $({\grad},{\curl})$ is compact (here $H_0=L^2(\Omega)$, $H_1=H_2=L^2(\Omega)^3$). In particular, so is $(\interior{\curl},\interior\dive)$.
\end{enumerate}
We refer to \cite{Bauer2016} also for mixed boundary conditions and the respective complex and/or compactness properties.
\end{example}

\begin{example}\label{ex:dfnf} In the situation of the previous example, let $\Omega$ be a bounded weak Lipschitz domain. The exactness of the considered complexes $(\grad,\curl)$ and $(\interior{\grad},\interior{\curl})$ can be guaranteed by topological properties of the domain and its complement. In fact, using the complex property ($\rge(\interior{\curl})\subseteq \kar(\interior{\dive})$ and $\rge(\curl)\subseteq \kar(\dive)$) we can decompose $L^2(\Omega)^3$ as follows
\[
   L^2(\Omega)^3 = \rge(\grad)\oplus \kar(\interior{\dive}) = \rge(\grad)\oplus \big( \kar(\interior{\dive}) \cap \kar(\curl)\big)\oplus \rge(\interior{\curl})
\]
and 
\[
   L^2(\Omega)^3 =  \rge(\interior{\grad})\oplus \kar({\dive}) = \rge(\interior{\grad})\oplus \big(\kar({\dive}) \cap \kar(\interior{\curl})\big)\oplus \rge({\curl}).
\]
Next, by Proposition \ref{prop:comelm}, $(\grad,\curl)$ is exact, if and only if $(\interior{\curl},\interior{\dive})$ is exact, if and only if $\kar(\interior{\dive})=\rge(\interior{\curl})$, if and only if $\dim(\kar(\interior{\dive})\cap \kar(\curl))=\{0\}$. Thus,
\[
    (\grad, \curl)\text{ exact }\iff \mathcal{H}_N\coloneqq  \{ q\in H(\interior{\dive},\Omega)\cap H(\curl,\Omega); \dive q = 0, \curl q =0\}=\{0\}.
\]
Similarly,
\[
   (\interior{\grad}, \interior{\curl})\text{ exact }\iff \mathcal{H}_D \coloneqq \{ q\in H(\dive,\Omega)\cap H(\interior{\curl},\Omega); \dive q = 0, \curl q =0\}=\{0\}.
\]
The space $\mathcal{H}_N$ describes the space of harmonic Neumann fields and $\mathcal{H}_D$ are the harmonic Dirichlet fields. Next, \cite[Remark 3(a)]{Picard1984} in conjunction with \cite[Theorem 1]{Picard82} leads to the following characterisations:
\[
   (\grad,\curl)\text{ exact }\iff \Omega \text{ simply connected}
\]
and
\[
   (\interior{\grad},\interior{\curl})\text{ exact }\iff \mathbb{R}^3\setminus {\Omega} \text{ connected}.
\]
\end{example}

Next, we recall a result on the well-posedness of abstract divergence form equations. This result is the Lax--Milgram lemma with a slight twist. We shall, however, emphasise this twist in the argument and the result. Due to the particular variational form of the considered problem class, one can identify elliptic problems in divergence form as the \emph{composition of three continuously invertible mappings}. This observation is the key for the derivations to come. For this reason we present the full proof.

For the statement of the next result, we introduce for a densely defined, closed linear operator $C\colon \dom(C)\subseteq H_0\to H_1$  the canonical embedding
\[\  \iota_{\textnormal{r},C} \colon \overline{\rge}(C)\hookrightarrow H_1.
\]
We note that $\iota_{\textnormal{r},C}^*$ is the orthogonal projection onto $\overline{\rge}(C)$, see \cite[Lemma 3.2]{PTW15_FI} for the elementary argument. 

\begin{theorem}[{{\cite[Theorem 3.1]{TW14_FE}}}]\label{thm:varwp} Let $B\colon \dom(B)\subseteq H_0\to H_1$ be densely defined and closed. Assume that 
\[
   B\text{ is one-to-one, } \rge(B)\subseteq H_1 \text{ closed}.
\]
Let $a\in L(H_1)$ be such that 
\[\Re \big(\iota_{\textnormal{r},B}^*a\iota_{\textnormal{r},B}\big)=(1/2)\iota_{\textnormal{r},B}^*(a+a^*)\iota_{\textnormal{r},B}\geq \alpha\iota_{\textnormal{r},B}^*\iota_{\textnormal{r},B}\]
for some $\alpha>0$.
Then for all $f\in \dom(B)^*$ there exists a unique $u\in \dom(B)$ such that
\[
    \langle a Bu,Bv\rangle = f(v)\quad(v\in \dom(B)).
\]
More precisely, we have 
\[
   u =\mathcal{B}^{-1} (\iota_{\textnormal{r},B}^*a\iota_{\textnormal{r},B})^{-1} (\mathcal{B}^\diamond)^{-1} f,
\]
where $\mathcal{B}\colon \dom(B)\to \rge(B), \phi\mapsto B\phi$ and $\mathcal{B}^\diamond \colon \rge(B)\to \dom(B)^*$ is given by
\[
   \phi \mapsto \big(\dom(B)\ni  v\mapsto \langle \phi,Bv\rangle_{\rge(B)}\big).
\]
\end{theorem}

\begin{example}\label{ex:dp1} Let $\Omega\subseteq \mathbb{R}^3$ be open and bounded. Then, by Example \ref{ex:diffo}(a1), the operator $\interior{\grad}$ has closed range (Poincar\'e`s inequality). Moreover, since $\dom(\interior{\grad})=H_0^1(\Omega)$, we have that $\interior{\grad}$ is one-to-one. Moreover, let $b\in L(L^2(\Omega)^3)$ satisfy
\[
   \Re \langle b q, q\rangle_{L^2(\Omega)^3}\geq \alpha\langle q,q\rangle_{L^2(\Omega)^3}
\]
for all $q\in \rge(\interior{\grad})$ and some $\alpha>0$. In this setting we may apply Theorem \ref{thm:varwp} to $H_0=L^2(\Omega)$, $H_1=L^2(\Omega)^3$, $B=\interior{\grad}$ and $a=b$. Then, by Theorem \ref{thm:varwp}, for all $f\in \dom(\interior{\grad})^*=H^{-1}(\Omega)$ there exists a unique $u\in H_0^1(\Omega)$ such that
\[
     \langle a \interior{\grad} u,\interior{\grad} v\rangle = f(v)\quad(v\in H_0^1(\Omega)).
\]
Using the notation from Theorem \ref{thm:varwp}, we realise that $\mathcal{B}\colon H_0^1(\Omega)\to \rge(\interior{\grad}),u\mapsto \grad u$. It is furthermore easy to see that $\mathcal{B}^{\diamond} = -\dive \colon \rge(\interior{\grad})\to H^{-1}(\Omega)$.  
\end{example}

We shortly elaborate on a nonlocal differential equation of the form of the previous example. It is a linear, static variant of the so-called McKean--Vlasov equation, see e.g.~\cite{Carillo18}
\begin{example}\label{ex:dp1nl} Let $\Omega\subseteq \mathbb{R}^3$ be open and bounded. Let $k \colon \Omega \times \Omega \to \mathbb{C}$ measurable and bounded, $\alpha>0$. Furthermore let $c\in L(L^2(\Omega)^3)$ satisfy 
\[
    \Re \langle c q, q\rangle_{L^2(\Omega)^3}\geq 2\alpha\langle q,q\rangle_{L^2(\Omega)^3}
\]
for all $q\in \rge(\interior{\grad})$. For $q\in L^2(\Omega)^3$ define
\[
    k*q \coloneqq \big(x\mapsto \int_{\Omega} k(x,y)q(y) dy\big).
\]
Assume that $\|k*\|_{L(L^2(\Omega)^3)}\leq \alpha$. Then $b\coloneqq c+k*$ satisfies the conditions imposed on $b$ in Example \ref{ex:dp1}. Indeed, for all $q\in \rge(\interior{\grad})$ we have
\begin{align*}
  \Re \langle b q,q\rangle & = \Re \langle cq , q\rangle + \Re \langle k*q,q\rangle 
  \\ & \ge 2\alpha \langle q,q\rangle - \alpha \|q\|^2 = \alpha \langle q,q\rangle.
\end{align*}
It is worth noting that using the expressions for $\mathcal{B}$ and $\mathcal{B}^\diamond$, we obtain as a resulting differential equation for any $f\in H^{-1}(\Omega)$, where we assume that $k$ is such that $k*$ commutes with the distributional gradient (we refer the reader also to Remark \ref{rem:varpde}):
\[
  f= - \dive b \interior{\grad} u = -\dive c\interior{\grad} u -\dive k* \interior{\grad} u = -\dive c\interior{\grad} u - \dive \grad k*u,
\]
which is of a form similar to \cite{Carillo18} in a static, linear case. 
\end{example}

Before we turn to the proof of Theorem \ref{thm:varwp}, we provide some particular insight for the case $a=1$.

\begin{proposition}\label{prop:varwp} Let $B\colon \dom(B)\subseteq H_0\to H_1$ be densely defined and closed. Assume that $B$ is one-to-one and has closed range. Then $\mathcal{B}$ and $\mathcal{B}^\diamond$ are Banach space isomorphisms. More precisely, we have \[\mathcal{B}^\diamond = \mathcal{B}'R_{\rge(B)},\] where $\mathcal{B}'\colon \rge(B)^*\to \dom(B)^*$ is the dual operator of $\mathcal{B}$ and $R_{\rge(B)}$ is given by
\[
    R_{\rge(B)} \colon \rge(B)\to \rge(B)^*, \phi\mapsto ( v\mapsto \langle \phi,v\rangle_{\rge(B)}).
\]
\end{proposition}
\begin{proof}
  $\mathcal{B}$ is one-to-one and onto. Hence, an isomorphism. Thus, so is $\mathcal{B}^*$. By unitary equivalence using the Riesz map, we obtain that $\mathcal{B}'$ is an isomorphism, as well. Finally, $R_{\rge(B)}$ is the inverse of the Riesz isomorphism. Thus, we are left showing that $\mathcal{B}^\diamond = \mathcal{B}'R_{\rge(B)}$ holds. For this, let $\phi\in \rge(B)$. Then we have for all $v\in\dom(B)$
  \begin{align*}
  \big(   \mathcal{B}^\diamond (\phi)\big)(v) & =  \langle \phi,Bv\rangle_{\rge(B)} \\
    & = \big(R_{\rge(B)} (\phi)\big)(Bv)\\
    & = \big(\mathcal{B}'R_{\rge(B)} (\phi)\big)(v). \qedhere 
  \end{align*}
\end{proof}
\begin{proof}[Proof of Theorem~\ref{thm:varwp}] We shall reformulate the left-hand side of the equation to be solved, first. For this, let $\pi_{\textnormal{r},B}$ be the orthogonal projection on $\rge(B)$. Note that $\pi_{\textnormal{r},B}=\iota_{\textnormal{r},B}\iota_{\textnormal{r},B}^*$.  Let $u,v\in\dom(B)$. Then we have 
\begin{align*}
   \langle aBu,Bv\rangle &= \langle a \pi_{\textnormal{r},B} B u, \pi_{\textnormal{r},B} B v\rangle \\
    & = \langle a \iota_{\textnormal{r},B}\iota_{\textnormal{r},B}^* B u, \iota_{\textnormal{r},B}\iota_{\textnormal{r},B}^* B v\rangle \\  & = \langle  \iota_{\textnormal{r},B}^* a \iota_{\textnormal{r},B}\mathcal{B} u,\mathcal{B} v\rangle \\
    & = \big(\mathcal{B}^\diamond \iota_{\textnormal{r},B}^* a \iota_{\textnormal{r},B}\mathcal{B} u\big)( v),
\end{align*}
where we used that $\iota_{\textnormal{r},B}^* B(u)=\mathcal{B}(u)$ for all $u\in \dom(B)$. Thus, the equation to be solved reads
\[
    \mathcal{B}^\diamond  \iota_{\textnormal{r},B}^* a \iota_{\textnormal{r},B} \mathcal{B} u =f.
\]
Under the hypotheses on $a$ using Proposition \ref{prop:varwp} and Lemma \ref{lem:conpd}(a) below, we infer both the uniqueness and the existence result as well as the solution formula.
\end{proof}

The next result deals with the case when $B$ is not one-to-one.

\begin{theorem}[{{\cite[Theorem 3.1]{TW14_FE}}}]\label{thm:varwp2} Let $C\colon \dom(C)\subseteq H_0\to H_1$ be densely defined and closed. Assume that
\[
    \rge(C)\subseteq H_1\text{ closed}.
\]
Let $a\in L(H_1)$ be such that for some $\alpha>0$
\[\Re \big(\iota_{\textnormal{r},C}^*a\iota_{\textnormal{r},C}\big)\geq \alpha\iota_{\textnormal{r},C}^*\iota_{\textnormal{r},C}.
\]
Then for all $f\in \dom(C\iota_{\textnormal{r},C^*})^*$ there exists a unique $u\in \dom(C\iota_{\textnormal{r},C^*})$ with the property
\[
   \langle a Cu,Cv\rangle = f(v) \quad(v\in \dom(C\iota_{\textnormal{r},C^*})).
\]
More precisely, we have
\[
     u = \mathcal{C}^{-1} (\iota_{\textnormal{r},C}^*a\iota_{\textnormal{r},C})^{-1} (\mathcal{C}^\diamond)^{-1} f,
\]
where $\mathcal{C} \colon \dom(C\iota_{\textnormal{r},C^*})\to \rge(C),\phi\mapsto C\phi$.
\end{theorem}
\begin{proof}
 The assertion follows from Theorem \ref{thm:varwp} applied to $B=C\iota_{\textnormal{r},C^*}$. 
\end{proof}

\begin{example}\label{ex:np1} Let $\Omega\subseteq \mathbb{R}^3$ be open, bounded with continuous boundary. Recall the differential operators from Example \ref{ex:diffo}.

(a) A standard application of Theorem \ref{thm:varwp2} would be the homogeneous Neumann boundary value problem, that is,
\[
    \langle a \grad u,\grad \phi\rangle = f(\phi)\quad (\phi\in H^1_{\bot}(\Omega)),
\]
for $a\in L(L^2(\Omega)^3)$ with $\Re a \geq \alpha$ for some $\alpha>0$ in the sense of positive definiteness and $H^1_{\bot}(\Omega)\coloneqq \{ u\in H^1(\Omega); \int_\Omega u = 0\}$ and $f\in \big(H^1_{\bot}(\Omega)\big)^*$. In fact, in order to solve this equation for $u\in H^1_\bot(\Omega)$ one applies Theorem \ref{thm:varwp2} to $H_0=L^2(\Omega)$, $H_1=L^2(\Omega)^3$ and $C=\grad$ with $\dom(C)=H^1(\Omega)$. Note that the positive definiteness condition for $a$ is trivially satisfied. Moreover, since $\Omega$ has continuous boundary, by Example \ref{ex:diffo}(b1) in conjunction with Proposition \ref{prop:comelm}(e), we obtain that $\rge(\grad)\subseteq L^2(\Omega)^3$ is closed. It remains to show that $\dom(C\iota_{r,C^*})=H^1_\bot(\Omega)$. For this we observe that \begin{align*}\rge(C^*) & =\rge({\grad}^*)=\rge(-\interior{\dive})\\ &= \kar(\grad)^\bot = \{ u \in L^2(\Omega); \langle u, \1\rangle = 0 \} = \{ u\in L^2(\Omega); \int_\Omega u = 0\}\eqqcolon L^2_\bot(\Omega).\end{align*} Hence, $\dom(C\iota_{r,C^*})= H^1(\Omega)\cap L^2_\bot(\Omega)= H^1_\bot(\Omega)$. We emphasise, using the notation from Theorem \ref{thm:varwp2}, that $\mathcal{C}\colon H^1_\bot(\Omega) \to \rge(\grad), u\mapsto \grad u$ is a topological isomorphism.

(b) Assume in addition that $\Omega$ is simply connected, and that $\Omega$ is a bounded, weak Lipschitz domain. Let again $a\in L(L^2(\Omega)^3)$ be strictly positive definite. By Example \ref{ex:diffo}(b1), we may apply Theorem~\ref{thm:varwp2} to $H_0=L^2(\Omega)^3$, $H_1=L^2(\Omega)^3$ and $C=\curl$ with $\dom(C)=H(\curl,\Omega)$. We define \[H_{\textrm{sol}}(\curl,\Omega)\coloneqq \{ q \in H(\curl,\Omega); \interior{\dive} q=0\}\] endowed with the norm from $H(\curl,\Omega)$. 
Now, by Theorem~\ref{thm:varwp2},  for all $g\in H_{\textrm{sol}}(\curl,\Omega)^*$ there exists a unique $v\in H_{\textrm{sol}}(\curl,\Omega)$ such that
\[
   \langle a \curl v,\curl \psi \rangle = g(\psi)\quad(\psi \in H_{\textrm{sol}}(\curl,\Omega)). 
\]
Indeed, the only thing to prove is that $\dom(C\iota_{r,C^*})=H_{\textrm{sol}}(\curl,\Omega)$. For this, we compute
\[
   \rge(C^*)=\rge(\interior{\curl}) = \kar(\interior{\dive}),
\]where in the last equality we have used that $\Omega$ is simply connected in order that $(\grad,\curl)$ and, hence, $(\interior{\curl},\interior{\dive})$ is exact, see also Examples~\ref{ex:diffo} and \ref{ex:dfnf}.
So, $\dom(C\iota_{r,C^*})=H(\curl,\Omega)\cap \kar(\interior{\dive})=H_{\textrm{sol}}(\curl,\Omega)$. In the situation discussed here, we have with the notation from Theorem~\ref{thm:varwp2}, $\mathcal{C}\colon H_{\textrm{sol}}(\curl,\Omega) \to \rge(\curl), q\mapsto \curl q$, which again yields a topological isomorphism. By construction, it follows that $\mathcal{C}^\diamond \colon \rge(\curl) \to H_{\textrm{sol}}(\curl,\Omega)^*$ is an extension of $\interior{\curl}|_{\dom(\interior{\curl})\cap \rge(\curl)}$ to the whole of $\rge(\curl)$. Moreover, recall that also $\mathcal{C}^\diamond$ is a topological isomorphism. 

(c) The example in (b) applies verbatim also to $C=\interior{\curl}$. In this case, however, one has to assume that $\mathbb{R}^3\setminus\Omega$ is connected in order to obtain \[
\dom(C\iota_{r,C^*})=H_{0,\textrm{sol}}(\curl,\Omega)\coloneqq \{q \in \dom(\interior{\curl}); \dive q =0\}.\]

(d) The connectedness assumptions in (b) and (c) can be dispensed with to the effect that the respective expressions of $\dom(C\iota_{r,C^*})$ are less explicit. In fact, in case of (b), one has
\[
    \dom(C\iota_{r,C^*}) = \{ q \in H(\curl,\Omega); q \in \rge(\interior{\curl})\}
\]
and in case of (c), one has
\[
     \dom(C\iota_{r,C^*}) = \{ q \in H_0(\curl,\Omega); q \in \rge({\curl})\}.
\]
\end{example}

\begin{remark}\label{rem:varpde}
 We note that the variational formulation in Theorem~\ref{thm:varwp2} is (trivially) equivalent to 
 \[
    \langle a Cu,Cv\rangle = f(v) \quad(v\in \dom(\mathcal{C})).
 \]
 Moreover, we see that due to the solution formula and Proposition \ref{prop:varwp}, we obtain a third formulation of the latter variational equation:
 \[
    \mathcal{C}^\diamond \iota_{\textnormal{r},C}^*a\iota_{\textnormal{r},C}\mathcal{C} u =   f.
 \]
\end{remark}

We conclude this section with some additional elementary results needed for the analysis to come.

\begin{lemma}\label{lem:conpd} (a) Let $a\in L(H_1)$ with $\Re a\geq c$ for some $c>0$ in the sense of positive definiteness. Then $a^{-1}\in L(H_1)$, $\|a^{-1}\|\leq 1/c$ and $\Re \big(a^{-1}\big)\geq c/{\|a\|^2}$.

(b) Let $(a_n)_n$ in $L(H_1)$ with $\Re a_n\geq c$ for all $n\in\mathbb{N}$ and some $c>0$. Assume that $a_n\to a$ converges in the weak operator topology to some $a\in L(H_1)$. Then $\Re a\geq c$.

(c) Let $(a_n)_n$ in $L(H_1)$ bounded with $\Re a_n\geq c$ for all $n\in\mathbb{N}$ and some $c>0$. Assume that $a_n^{-1}\to b$ converges in the weak operator topology to some $b\in L(H_1)$. Then $b^{-1}\in L(H_1)$ with $\|b^{-1}\|\leq {\sup_{n\in\mathbb{N}}\|a_n\|^2}/c$, and $\Re \big(b^{-1}\big)\geq c$.

(d) Let $(a_n)_n$ in $L(H_1)$ with $\Re a_n\geq \alpha$ and $\Re \big(a_n^{-1}\big)\geq 1/ \beta$ for all $n\in\mathbb{N}$ and some $\alpha,\beta>0$. Assume that $a_n^{-1}\to b$ in the weak operator topology. Then we have for $a\coloneqq b^{-1}$ that $\Re a\geq \alpha$ and $\Re \big(a^{-1}\big)\geq 1/ \beta$. 
\end{lemma}
\begin{proof}
 (a) is a straightforward consequence of the Cauchy--Schwarz inequality. (b) is easy. The assertion in (d) is a straightforward consequence of (c). Thus, we are left with showing (c). By part (a), we deduce that $\Re a_n^{-1}\geq c/{\sup_{n\in\mathbb{N}}\|a_n\|^2}$. By (b), we deduce that $\Re b\geq c/ {\sup_{n\in\mathbb{N}}\|a_n\|^2}$. This ensures $b^{-1}\in L(H_1)$ and that $\|b^{-1}\|\leq {\sup_{n\in\mathbb{N}}\|a_n\|^2}/c$. Finally, let $\phi\in H_1$ and put $\psi\coloneqq b^{-1}\phi$ as well as $\phi_n\coloneqq a_n^{-1}\psi$. Then we compute
\begin{align*}
  \Re \langle b\phi,\phi\rangle & = \Re \langle \psi, b^{-1}\psi\rangle \\ 
  & = \lim_{n\to\infty} \Re \langle \psi,a_n^{-1}\psi\rangle \\
  & = \liminf_{n\to\infty} \Re \langle a_n\phi_n,\phi_n\rangle \\
  & \geq \liminf_{n\to\infty}  c\langle \phi_n,\phi_n\rangle \\
    & \geq  c\langle \phi,\phi\rangle,
\end{align*}
where in the last step we used that $\phi_n\rightharpoonup \phi$ and so $\|\phi\|\leq \liminf_{n\to\infty} \|\phi_n\|$.
\end{proof}

The following result will be of importance later on, when we compare local $H$-convergence to nonlocal $H$-convergence.

\begin{proposition}\label{prop:invdis} Let $H$ be a Hilbert space. Then the following conditions are equivalent:
\begin{enumerate}
  \item[(i)] $H$ is finite-dimensional.
  \item[(ii)] For any bounded sequence $(a_n)_n$ in $L(H)$ such that $\Re a_n\geq \alpha$ for all $n\in\mathbb{N}$ and some $\alpha>0$, $a\in L(H)$ invertible, we have the following equivalence:
  \begin{multline*}
     (a_n)_n \to a \text{ in the weak operator topology} \\ \iff (a_n^{-1})_n \to a^{-1} \text{ in the weak operator topology}.
  \end{multline*}
\end{enumerate}
\end{proposition}
\begin{proof}
 The statement (i) implies (ii) since then the weak operator topology on $L(H)$ coincides with the norm topology.
 
 For the statement (ii) implies (i), it suffices to provide a counterexample for $H=L^2(0,1)$. Furthermore, it is elementary to see that if we identify any $a\in L^\infty(0,1)$ with the corresponding multiplication operator on $L^2(0,1)$, that the weak operator topology on $L^\infty(0,1)$ coincides with $\sigma(L^\infty,L^1)$, that is, the weak*-topology on $L^\infty(0,1)$. With these preparations let $b\coloneqq \1_{(0,1/2)}+\frac{1}{2}\1_{(1/2,1)}$, where $\1_K$ denotes the characteristic function of a set $K$. Periodically extending $b$ to the whole of $\R$, we put $a_n\coloneqq (x\mapsto b(n\cdot x))$. By \cite[Theorem 2.6]{Cioranescu1999}, we deduce that $a_n \to (\int b)\1_{(0,1)}=\frac{3}{4}\1_{(0,1)}\eqqcolon a$ in $\sigma(L^\infty,L^1)$ as $n\to \infty$. On the other hand, $a_n^{-1} \to \big(\int_{(0,1)} \frac{1}b\big)\1_{(0,1)}=\frac{3}{2}\1_{(0,1)}\neq a^{-1}$, which shows that (ii) is false.
\end{proof}
\section{Nonlocal $H$-convergence for exact sequences}\label{sec:nonH}

As in the previous section, we assume that $A_0$ and $A_1$ are densely defined, closed linear operators from $H_0$ to $H_1$ and $H_1$ to $H_2$, respectively. We shall assume that $(A_0,A_1)$ is closed and exact. Note that then by Proposition \ref{prop:comelm} $(A_1^*,A_0^*)$ is closed and exact, as well. Furthermore, we have the  following orthogonal decompositions for $H_1$:
\begin{equation}\label{eq:od}
   H_1 = \rge(A_0)\oplus\kar(A_0^*) = \rge(A_0)\oplus\rge(A_1^*)= \kar(A_1)\oplus\rge(A_1^*)= \kar(A_1)\oplus\kar(A_0^*).
\end{equation}
For the example cases treated in Example \ref{ex:diffo}, the decompositions expressed in \eqref{eq:od} are abstract variants of Helmholtz decompositions, as it will be exemplified next.
\begin{example}\label{ex:Helm}
 Let $\Omega\subseteq \mathbb{R}^3$ be a bounded weak Lipschitz domain.
 \begin{enumerate}
 \item[(a)] Assume in addition that ${\Omega}$ is simply connected. Then by Example \ref{ex:diffo}, we obtain that $(A_0,A_1)=(\grad,\curl)$ is exact and closed. Thus, we obtain from \eqref{eq:od}
 \[
     L^2(\Omega)^3 =\rge(\grad)\oplus \kar(\interior{\dive}) = \rge(\grad)\oplus \rge(\curl^*) = \rge(\grad)\oplus \rge(\interior{\curl}).
 \]
 As a consequence, any $q\in L^2(\Omega)^3$ decomposes into $q = \grad \phi+\interior{\curl}\psi$ for some $\phi\in H_\bot^1(\Omega)$ and $\psi \in H_{0,\textrm{sol}}(\curl,\Omega)$. Note that $\phi$ and $\psi$ are uniquely determined and depend continuously on $q$.
 \item[(b)] If the complement of $\Omega$ is connected, we infer by Example \ref{ex:diffo} that $(A_0,A_1)=(\interior{\grad},\interior{\curl})$ is closed and exact. As a consequence of equation \eqref{eq:od}, we obtain
 \[
    L^2(\Omega)^3 = \rge(\interior{\grad}) \oplus \kar({\dive}) = \rge(\interior{\grad})\oplus \rge(\curl).
 \] 
 Thus, similar to (a), for any $q\in L^2(\Omega)^3$, we find uniquely determined $\phi\in H_0^1(\Omega)$ and $\psi\in H_{\textrm{sol}}(\curl,\Omega)$ such that $q= \interior{\grad}\phi +\curl\psi$. The `potentials' $\phi$ and $\psi$ depend continuously on $q$.
 \item[(c)] Let $\Omega$ be simply connected with $C^1$-boundary and $\mathbb{R}^3\setminus {\Omega}$ connected. Then, by (a) and (b), we have
 \[
    L^2(\Omega)^3 = \rge(\interior{\grad}) \oplus \rge(\curl) = \rge(\grad)\oplus \rge(\interior{\curl}).
 \]
 Denoting $V\coloneqq \rge(\interior{\grad})^\bot\cap \rge(\grad)$, we obtain
 \[
    L^2(\Omega)^3 = \rge(\grad)\oplus \rge(\interior{\curl}) = \rge(\interior{\grad})\oplus V \oplus \rge(\interior{\curl})=\rge(\interior{\grad}) \oplus \rge(\curl),
 \]
 which implies that $V\oplus \rge(\interior{\curl}) = \rge(\curl)$. For later use, we shall describe $V$ in more detail as follows. Our main aim is to show that $V$ is infinite-dimensional. For this, let $q \in \rge(\grad)$. Then we find a uniquely determined $\psi_q\in H_\bot^1(\Omega)$ such that $q=\grad \psi_q$. In fact, this follows from Example \ref{ex:np1}(a) since $H_\bot^1(\Omega)\to \rge(\grad), u\mapsto \grad u$ is a topological isomorphism. Hence, $q\in V$ if and only if for all $\phi\in H_0^1(\Omega)$
 \[
    \langle \interior{\grad} \phi, \grad \psi_q \rangle = 0,
 \]
 which in turn is equivalent to
 \[
     \dive \grad \psi_q = 0. 
 \]
 Thus,
 \[
     V = \{ \grad \psi; \psi \in H_\bot^1(\Omega), \Delta \psi = 0\}.
 \]
 Recall $\gamma_{\textrm{n}}\colon H(\dive,\Omega)\to H^{-1/2}(\Omega)$ to be the normal trace. Since $\Omega$ has $C^1$-boundary, we find an infinite set $W\subseteq C^2(\overline{\Omega})$ with the property that for any $v,w\in W$ with $v\neq w$ we have $\spt_{\partial\Omega} (\textrm{n}\cdot \grad v),\spt_{\partial\Omega} (\textrm{n}\cdot \grad w)\neq \emptyset$ and $\spt_{\partial\Omega} (\textrm{n}\cdot \grad v) \cap \spt_{\partial\Omega} (\textrm{n}\cdot \grad w)=\emptyset$, where $\spt_{\partial\Omega} f$ denotes the support of a (continuous) function $f$ on $\partial\Omega$. Next, using Example \ref{ex:np1}(a) for any $v \in W$, we may solve for $w\in H^1_\bot(\Omega)$ such that 
 \begin{equation}\label{eq:np2}
    \langle \grad w, \grad \phi\rangle = \langle \grad v, \grad \phi\rangle\quad(\phi\in H_\bot^1(\Omega)).
 \end{equation}
 Then $\grad w \in \dom(\interior{\dive})$ and so $\gamma_{\textnormal{n}}(\grad w)=0$. 
 We put $u_v\coloneqq v-w$. Then $\gamma_\textnormal{n}(\grad u_v) = \gamma_{\textnormal{n}}(\grad v) = \textrm{n}\cdot\grad v|_{\partial\Omega}$. Moreover, \eqref{eq:np2} (trivially) extends to all $\phi \in H^1(\Omega)\supseteq H_0^1(\Omega)$. Hence, for all $\phi\in H_0^1(\Omega)$
 \[
    \langle \grad u_v,\grad \phi\rangle = \langle \grad (v-w),\grad \phi\rangle = \langle \grad v,\grad \phi\rangle -\langle \grad v,\grad \phi\rangle =0.
 \]Thus, $ \grad u_v \in V$. Since $\lin\{ \textnormal{n}\cdot \grad v; v\in W\}$ is infinite-dimensional, we infer the same for $\lin\{\grad u_v; v\in W\}\subseteq V$. Hence, $V$ is infinite-dimensional, which concludes this example.
 \end{enumerate}
\end{example}

Using the notation $\iota_{\textnormal{r},C}$ for densely defined closed linear operators $C\colon \dom(C)\subseteq H_0\to H_1$ from the previous section, we may define $a_{00}\coloneqq \iota_{\textnormal{r},A_0}^*a\iota_{\textnormal{r},A_0} \in L({\rge}(A_0))$ and $a_{11}\coloneqq \iota_{\textnormal{r},A_1^*}^*a\iota_{\textnormal{r},A_1^*} \in L({\rge}(A_1^*))$. The set of admissible (nonlocal) coefficients for which we discuss the notion of nonlocal $H$-convergence is described next. For $\alpha,\beta>0$, we define
\begin{align*}
   \mathcal{M}(\alpha,\beta,(A_0,A_1)) \coloneqq \{ a\in L(H_1);&  \Re a_{00} \geq \alpha 1_{A_0}, \Re a_{00}^{-1} \geq (1/\beta) 1_{A_0},\\ 
    & a \text{ continuously invertible},\\
    & \Re (a^{-1})_{11} \geq (1/\beta) 1_{A_1^*}, \Re (a^{-1})^{-1}_{11} \geq \alpha 1_{A_1^*} \},
\end{align*}
where $1_{A_0}$ and $1_{A_1^*}$ are the identity operators in $\rge(A_0)$ and $\rge(A_1^*)$, respectively.

We shall present a possibly nonlocal coefficient in the following example.

\begin{example}\label{ex:nlc} Let $\Omega\subseteq \mathbb{R}^3$ be a bounded, simply connected, weak Lipschitz domain with connected complement, $0<\alpha \leq 1 \leq \beta $. Then using Example \ref{ex:diffo} and \ref{ex:dfnf} both $(\interior{\grad},\interior{\curl})$ and $(\grad,\curl)$ are exact. Moreover, according to Example \ref{ex:Helm}(c), we have for some $V$ the decomposition
\[
   L^2(\Omega)^3 = \rge(\interior{\grad})\oplus V \oplus \rge(\interior{\curl}).
\]
Denote by $b \in L(V)$ an operator with $\Re b\geq \alpha$ and $\Re b^{-1}\geq 1/\beta$. Then
$a\coloneqq \pi_{V^\bot} + \iota_V b \iota_V^* \in \mathcal{M}(\alpha,\beta,(\interior{\grad},\interior{\curl})) \cap \mathcal{M}(\alpha,\beta,({\grad},{\curl}))$, where $\pi_{V^\bot}$ denotes the orthogonal projection onto $V^\bot$ and $\iota_V \colon V \hookrightarrow L^2(\Omega)^3$

Indeed, for $(A_0,A_1)=(\interior{\grad},\interior{\curl})$ we have $a_{00}=a_{00}^{-1}=1_{\interior{\grad}}$,\[(a^{-1})_{11} = \pi_{\rge(\interior{\curl})} + \iota_{\textnormal{r},\curl}^*\iota_V b^{-1} \iota_V^*\iota_{\textnormal{r},\curl} \geq (1/\beta)1_{\curl},\] and $(a^{-1})_{11}^{-1}\geq \alpha 1_{\curl}$.

For $(A_0,A_1)=(\grad,\curl)$ we have $(a^{-1})_{11} = (a^{-1})_{11}^{-1}= 1_{\interior{\curl}}$, \[a_{00} = \pi_{\rge(\interior{\grad})} + \iota_{\textnormal{r},\grad}^*\iota_V b \iota_V^*\iota_{\textnormal{r},\grad}\geq \alpha 1_{\grad},\] and, similarly, $a_{00}^{-1} \geq  (1/\beta)1_{\grad}$. We shall revisit this example in Example \ref{ex:dpbc} below.
\end{example}

Note that since $(A_0,A_1)$ is closed and exact, both $A_0$ and $A_1$ satisfy the conditions imposed on $C$ in Theorem~\ref{thm:varwp2}. Thus, the equations in the following definitions are uniquely solvable by Theorem~\ref{thm:varwp2}. We will use \[\mathcal{A}_0 \colon \dom(A_0 \iota_{\textnormal{r},A_0^*}) \to \rge(A_0),u\mapsto A_0 u\]and, similarly, \[\mathcal{A}_1^*\colon \dom(A_1^* \iota_{\textnormal{r},A_1})\to \rge(A_1^*),v\mapsto A_1^*v.\]

\begin{example}\label{ex:cala} We shall specify the operators $\mathcal{A}_0$ and $\mathcal{A}_1^*$ in two particular cases. Recall that $\iota_{\textnormal{r},A_0^*}$ is the canonical embedding from $\rge(A_0^*)=\kar(A_0)^\bot$ into $H_0$.

(a) Let $\Omega\subseteq \mathbb{R}^3$ be a bounded, weak Lipschitz domain with connected complement. Then, we recall that $(A_0,A_1)=(\interior{\grad},\interior{\curl})$ with $H_0= L^2(\Omega)$, $H_1=H_2=L^2(\Omega)^3$ is closed and exact; see Examples \ref{ex:diffo} and \ref{ex:dfnf}. In particular, using the results from Example \ref{ex:dp1}, we obtain that
\[
    \mathcal{A}_0 \colon H_0^1(\Omega)\to \rge(\interior{\grad}), u\mapsto \grad u
\]
and with Example \ref{ex:np1}(b) we get
\[
   \mathcal{A}_1^* \colon H_{\textrm{sol}}(\curl,\Omega) \to \rge(\curl), q\mapsto \curl q.
\]
Recall that we also have $\rge(\interior{\curl})=\kar(\interior{\dive})\supseteq H_{\textrm{sol}}(\curl,\Omega)$. Thus, $\mathcal{A}_0$ and $\mathcal{A}_1^*$ act the same way as $A_0$ and $A_1^*$, they are however arranged in the way that the domain and co-domain is restricted in order to make the calligraphic variants of the operators $A_0$ and $A_1^*$ be both topological isomorphisms.

(b) Let $\Omega\subseteq \mathbb{R}^3$ be a bounded, simply connected, weak Lipschitz domain. Then recall $(A_0,A_1)=(\grad,\curl)$ is closed and exact. In this case, we have
\[
     \mathcal{A}_0 \colon H_\bot^1(\Omega)\to \rge({\grad}), u\mapsto \grad u
\]
and
\[
     \mathcal{A}_1^* \colon H_{0,\textrm{sol}}(\curl,\Omega) \to \rge({\curl}), q\mapsto \curl q.
\]
\end{example}

\begin{definition} Let $(A_0,A_1)$ be exact and closed. Let $(a_n)_n$ in $M(\alpha,\beta,(A_0,A_1))$, $a\in L(H_1)$ continuously invertible. Then $(a_n)_n$ is called \emph{nonlocally $H$-convergent to $a$ with respect to $(A_0,A_1)$}, if the following statement holds: For all $f\in \dom(\mathcal{A}_0)^*$ and $g\in \dom(\mathcal{A}^*_1)^*$ let $(u_n)_n$ in $\dom(\mathcal{A}_0)$ and $(v_n)_n$ in $\dom(\mathcal{A}_1^*)$ satisfy
\[
 \langle a_n A_0 u_n, A_0 \phi\rangle  = f(\phi), \quad \langle a_n^{-1} A_1^* v_n,A_1^*\psi\rangle =g(\psi)\quad(n\in \mathbb{N}).
\]
for all $\phi\in\dom(\mathcal{A}_0)$, $\psi\in \dom(\mathcal{A}_1^*)$.
Then
\begin{itemize}
 \item $u_n\rightharpoonup u$ in $\dom(A_0)$ for some $u\in\dom(A_0)$;  
 \item $v_n\rightharpoonup v$ in $\dom(A_1^*)$ for some $v\in\dom(A_1^*)$;
 \item $a_n A_0 u_n \rightharpoonup a A_0 u$; $a_n^{-1} A_1^* v_n \rightharpoonup a^{-1} A_1^* v$;
 \item $u$ and $v$  satisfy
 \[
     \langle a A_0 u, A_0 \phi\rangle  = f(\phi) \text{ and } \langle a^{-1} A_1^* v,A_1^*\psi\rangle =g(\psi)
 \]for all $\phi\in\dom(\mathcal{A}_0)$, $\psi\in \dom(\mathcal{A}_1^*)$.
\end{itemize}
$a$ is called \emph{nonlocal $H$-limit} of $(a_n)_n$. 
If the choice of the exact complex $(A_0,A_1)$ is clear from the context, we shall also say that $(a_n)_n$ nonlocally $H$-converges to $a$, for short.
\end{definition}

Using the Examples \ref{ex:diffo} and \ref{ex:dfnf} together with the descriptions of the Hilbert spaces in Example \ref{ex:np1}, we shall formulate the notion of nonlocal $H$-convergence in the special case of $(\interior{\grad},\interior{\curl})$:

\begin{example} Let $\Omega\subseteq \mathbb{R}^3$ be an open, bounded, weak Lipschitz domain with connected complement. Then with $H_0=L^2(\Omega)$, $H_1=H_2=L^2(\Omega)^3$, we have that $(A_0,A_1)=(\interior{\grad},\interior{\curl})$ is closed and exact. Let $\alpha,\beta>0$ and $(a_n)_n$ in $\mathcal{M}(\alpha,\beta,(\interior{\grad},\interior{\curl}))$. Then $(a_n)_n$ nonlocally $H$-converges to $a\in \mathcal{M}(\alpha,\beta,(\interior{\grad},\interior{\curl}))$ with respect to $(\interior{\grad},\interior{\curl})$, if the following holds:

For all $f\in H^{-1}(\Omega)$ and $g\in H_{\textnormal{sol}}(\curl,\Omega)^*$ let $(u_n)_n$ in $H_0^1(\Omega)$ and $(v_n)_n$ in $H_{\textnormal{sol}}(\curl,\Omega)$ satisfy
\[
 \langle a_n \grad u_n,\grad \phi\rangle  = f(\phi) \quad \langle a_n^{-1} \curl v_n,\curl\psi\rangle =g(\psi)\quad(n\in \mathbb{N}).
\]
for all $\phi\in H_0^1(\Omega)$, $\psi\in H_{\textnormal{sol}}(\curl,\Omega)$.
Then
\begin{itemize}
 \item $u_n\rightharpoonup u$ in $H_0^1(\Omega)$ for some $u\in H_0^1(\Omega)$;  
 \item $v_n\rightharpoonup v$ in $H(\curl,\Omega)$ for some $v\in H(\curl,\Omega)$;
 \item $a_n \grad u_n \rightharpoonup a \grad u$; $a_n^{-1} \curl v_n \rightharpoonup a^{-1} \curl v$ both convergences hold weakly in $L^2(\Omega)^3$;
 \item $u$ and $v$  satisfy
 \[
     \langle a \grad u, \grad \phi\rangle  = f(\phi) \text{ and } \langle a^{-1} \curl v,\curl\psi\rangle =g(\psi)
 \]for all $\phi\in H_0^1(\Omega)$, $\psi\in H_{\textnormal{sol}}(\curl,\Omega)$. 
 \end{itemize}
\end{example}

\begin{remark}
  We have formulated the notion of nonlocal $H$-convergence for exact and closed complexes only. There are two desirable steps of generalisation. A first one is to consider finite-dimensional `cohomology groups' $\kar(A_0^*)\cap \kar(A_1)$. A prime application of this are compact complexes. Thus, the definition of nonlocal $H$-convergence needs to take into account coefficient sequences $(a_n)_n$ acting on or mapping into the finite-dimensional space $\kar(A_0^*)\cap \kar(A_1)$. In applications to concrete complexes, this setting allows for compact complexes and in particular for more general topologies of the underlying domain $\Omega$ in Example \ref{ex:diffo} and Example \ref{ex:dfnf}. A second step is to consider non-closed complexes. In the light of Example \ref{ex:diffo}, this would pave the way to unbounded $\Omega$.
\end{remark}

We shall analyse the relationship to local $H$-convergence of multiplication operators in Section \ref{sec:metco}. This requires further theoretical insight. However, before we discuss more abstract theory for the notion just introduced, we explicitly consider the particular case of (periodic) multiplication operators, which perfectly fits into the scheme above. In the following, we will identify $a\in M(\alpha,\beta,Y)$ (see \eqref{eq:Mab} for the definition) with the corresponding multiplication operator acting on $L^2(Y)^3$. 

\begin{example}\label{ex:1st} Let $Y=[0,1)^3$ and let $a\in L^\infty (\mathbb{R}^3)^{3\times 3}$ be $Y$-periodic. Assume there is $\alpha,\beta>0$ such that $a\in {M}(\alpha,\beta,Y)$. For $n\in\mathbb{N}$ we put $a_n\coloneqq (y
\mapsto a(n\cdot y))$. 

 Let $A_0=\iota_{0}$ and $A_1=1$ as in Example \ref{ex:tr} with $H=L^2(Y)^3$. Then it is easy to see that 
\[
    a_n\in \mathcal{M}(\alpha,\beta,(\iota_0,1))\quad(n\in\mathbb{N}).
\]
Next, let $f\in \dom(\mathcal{A}_0)^*=\{0\}^*=\{0\}$ and $g\in \dom(\mathcal{A}^*_1)^*=L^2(Y)^3$  and let $u_n\in \dom(A_0)\cap \kar(A_0)^\bot=\{0\}$ and $v_n\in \dom(A_1^*)\cap \kar(A_1^*)^\bot = L^2(Y)^3$ satisfy
\[
   A_0^\diamond a_n A_0 u_n = f, \quad (A_1^*)^\diamond a_n^{-1} A_1^* v_n = g.
\]
Note that the first equation is trivially satisfied, as $u_n=0$ and $f=0$. Moreover, note that $a_nA_0 u_n = 0$ for all $n\in\mathbb{N}$. The second equation implies $a_n^{-1}v_n =g$ and so
\[
   v_n = a_n g.
\]
By \cite[Theorem 2.6]{Cioranescu1999}, we deduce that $v_n = a_n g \rightharpoonup (\int_Y a)g\eqqcolon v$. Moreover, 
\[
a_n^{-1}v_n=g\to g=\Big(\int_Y a\Big)^{-1}v \quad(n\to \infty).
\]
Hence, $(a_n)_n$ nonlocally $H$-converges with respect to $(\iota_{0},1)$ to $\int_Y a$. 

\end{example}

A simple modification of Example \ref{ex:1st} shows that nonlocal $H$-convergence with respect to $(\iota_0,1)$ is precisely convergence of $(a_n)_n$ in the weak operator topology.

\begin{proposition}\label{prop:wot} Let $H$ be a Hilbert space. Let $\iota_0\colon \{0\}\hookrightarrow H, 0\mapsto 0$, $1_{H}\colon H \to H, x\mapsto x$ and let $\alpha,\beta>0$. Let $(a_n)_n$ in $\mathcal{M}(\alpha,\beta,(\iota_0,1_{H}))$. Let $a\in L(H)$ invertible. Then the following conditions are equivalent:
\begin{enumerate}
  \item[(i)] $a_n \to a$ in the weak operator topology as $n\to\infty$,
  \item[(ii)] $a_n \to a$ $H$-nonlocally with respect to $(\iota_0,1_{H})$ as $n\to\infty$.
\end{enumerate}
\end{proposition}
\begin{proof}
 The proof of (i)$\Rightarrow$(ii) follows almost literally the arguments outlined in Example \ref{ex:1st}. If (ii) holds, the conditions on nonlocal $H$-convergence imply that $v_n= a_n g$ converges weakly to $v= a g$ for all $g\in H$. This, however, implies (i). 
\end{proof}

The next example is a standard result in homogenisation, see e.g.~\cite[Lemma 4.5]{Jikov1994} and \cite[Theorem 6.1]{Cioranescu1999}.

\begin{example}\label{ex:2nd} Let $Y=[0,1)^3$ and let $a\in L^\infty (\mathbb{R}^3)^{3\times 3}$ be $Y$-periodic. Assume there is $\alpha,\beta>0$ such that $a\in {M}(\alpha,\beta,Y)$ with $a=a^*$. For $n\in\mathbb{N}$ we put $a_n\coloneqq \big(y
\mapsto a(n\cdot y)\big)$. Note that this particularly implies that both $(a_n)_n$ and $(a_n^{-1})_n$ are bounded. Let $H_0 = L^2(Y)$, $H_1=H_2=L^2(Y)^3$ and let $A_0=\interior{\grad}$, $A_1=\interior{\curl}$, and $\dive$ as well as $\curl$ as in Example \ref{ex:diffo} with $\Omega=\interior{Y}$. We may also use the results of the Examples \ref{ex:dp1} and \ref{ex:np1}. Then $(A_0,A_1)$ is exact and closed. (Note that exactness also follows directly with a Fourier series argument.) Moreover, it is plain that $a_n\in \mathcal{M}(\alpha',\beta',(\interior{\grad},\interior{\curl}))$ for some $0<\alpha'\leq \alpha \leq \beta \leq \beta'$. Let $f\in H^{-1}(Y)$, $g\in H_{\textrm{sol}}(\curl,Y)^*$. Let $(u_n)_n$ in $H^1_0(Y)$ and $(v_n)_n$ in $H_{\textrm{sol}}(\curl,Y)$  satisfy the following equations
\[
   -\dive a_n \interior{\grad} u_n = f,\quad \interior{\curl}a_n^{-1} \curl v_n =g,
\]
where in the latter equation, we slightly abused notation, also cf. Remark \ref{rem:varpde} and the concluding comments in Example \ref{ex:np1}(b). By \cite[Theorem 6.1]{Cioranescu1999}, we have that $(u_n)_n$ weakly converges to some $u \in H_0^1(\Omega)$ and there exists a constant coefficient matrix $a_{\textnormal{hom}}$ with the property that
\[
  -\dive a_{\textnormal{hom}} \interior{\grad} u = f. 
\] Moreover, we have $a_n\interior{\grad}u_n\rightharpoonup a_{\textnormal{hom}}\interior{\grad}u$ in $L^2(Y)^3$.  

Next, we set $w_n\coloneqq a_n^{-1} \curl v_n$. Note that by the solution formula in Theorem~\ref{thm:varwp2} applied to $C=\curl$ (see also Example~\ref{ex:np1}(b)), the sequence $(v_n)_n$ is bounded in $H_{\textrm{sol}}(\curl,Y)$ and by the boundedness of $(a_n^{-1})_n$, the sequence $(w_n)_n$ is bounded in $L^2(Y)^3$. Take a weakly convergent subsequence of $(w_n)_n$ (in $L^2(Y)^3$) and $(v_n)_n$ (in $H_{\textrm{sol}}(\curl,Y)$). Denote the corresponding limits by $w$ and $v$. We do not relabel the sequences. From $\interior{\curl}w_n = g$ and $\dive a_nw_n = 0$, it follows with  \cite[Lemma 4.5]{Jikov1994} that $a_nw_n\rightharpoonup a_{\textnormal{hom}}w=\curl v$. Hence, 
\[
   g = \text{w-}\lim_{n\to\infty} \interior{\curl} w_n = \interior{\curl} w=  \interior{\curl} a_{\textnormal{hom}}^{-1}\curl v.
\] 
Uniqueness of $v$ follows from Theorem~\ref{thm:varwp2} and the coercivity of $a_{\textnormal{hom}}$, see \cite[Section 6.3]{Cioranescu1999}.

All in all, we have shown that $a_n \to a_{\textnormal{hom}}$ $H$-nonlocally with respect to $(\interior{\grad},\interior{\curl})$.
\end{example}

We will show in Section \ref{sec:dct} that a result analogous to \cite[Lemma 4.5]{Jikov1994} characterises nonlocal $H$-convergence.

\begin{example}\label{ex:mcv} Let $\Omega\subseteq \mathbb{R}^3$ be an open, bounded weak Lipschitz domain with connected complement. Let $(k_n)_n$ be bounded in $L^\infty(\Omega\times \Omega)$ and assume that there is $0\leq \theta<1$ such that $\|k_n*\|_{L(L^2(\Omega)^3)}\leq\theta$ for all $n\in \mathbb{N}$, where we refer to Example \ref{ex:dp1nl} for a definition. With an argument similar to that in Example \ref{ex:dp1nl} it follows that
\[
   \Re (1 -k_n*)\geq (1-\theta)1_{L^2(\Omega)^3}\quad (n\in\mathbb{N}).
\]
In particular, one has $\Re \big((1 -k_n*)^{-1}\big)\geq (1/\beta') 1_{L^2(\Omega)^3}$ for some $\beta'>0$. Thus $(a_n)_n = (1-k_n*)_n$ is an eligible sequence in $\mathcal{M}(\alpha,\beta,(\interior{\grad},\interior{\curl}))$ for some $\alpha,\beta>0$. So the question is, under which circumstances can we show that
\[
   1-k_n*\to a\text{ $H$-nonlocally with respect to }(\interior{\grad},\interior{\curl}) \text{ as }n\to \infty
\]
for some $a\in \mathcal{M}(\alpha,\beta,(\interior{\grad},\interior{\curl}))$?
 Note that this question particularly deals with the operator $(1-k_n*)^{-1}=\sum_{\ell=0}^\infty (k_n*)^\ell$. For now, however, this operators is too complicated an object to deal with. In Section \ref{sec:dct}, we shall revisit such kind of coefficients using a criterion for nonlocal $H$-convergence detouring the computation of the inverse.
\end{example}

\begin{remark}
A quick comparison of the Examples \ref{ex:1st} and \ref{ex:2nd} shows that the nonlocal $H$-limit is \emph{not} independent of the underlying exact complex. In this line of reasoning it might not be expected that nonlocal $H$-convergence is independent of the considered boundary conditions either. In fact, we will show that nonlocal $H$-convergence indeed depends on the choice of boundary condition. This will be discussed in Example~\ref{ex:dpbc}. We refer also to Remark \ref{rem:neu} below on this matter.
\end{remark}

\section{Block matrix representation of nonlocal \\ $H$-convergence}\label{sec:block}

As in the previous section, we shall assume throughout that $A_0$ and $A_1$ are densely defined, closed linear operators from $H_0$ to $H_1$ and $H_1$ to $H_2$, respectively, with $(A_0,A_1)$ closed and exact.

Our first aim of this section is to characterise nonlocal $H$-convergence by means of convergence of operators in a block matrix representation. For this, we employ the orthogonal decomposition mentioned in \eqref{eq:od}. For $a\in L(H_1)$ we obtain
\begin{multline*}
   a= \begin{pmatrix} \iota_{\textnormal{r},A_0} & \iota_{\textnormal{r},A_1^*} \end{pmatrix}\begin{pmatrix}  \iota_{\textnormal{r},A_0}^* \\ \iota_{\textnormal{r},A_1^*}^* \end{pmatrix} a \begin{pmatrix} \iota_{\textnormal{r},A_0} & \iota_{\textnormal{r},A_1^*} \end{pmatrix}\begin{pmatrix}  \iota_{\textnormal{r},A_0}^* \\\iota_{\textnormal{r},A_1^*}^* \end{pmatrix}\\ =\begin{pmatrix} \iota_{\textnormal{r},A_0} & \iota_{\textnormal{r},A_1^*} \end{pmatrix}
   \begin{pmatrix} \iota_{\textnormal{r},A_0}^*a \iota_{\textnormal{r},A_0} & \iota_{\textnormal{r},A_0}^*a  \iota_{\textnormal{r},A_1^*} \\ \iota_{\textnormal{r},A_1^*}^*a \iota_{\textnormal{r},A_0} & \iota_{\textnormal{r},A_1^*}^*a  \iota_{\textnormal{r},A_1^*} \end{pmatrix}  \begin{pmatrix} \iota_{\textnormal{r},A_0}^* \\ \iota_{\textnormal{r},A_1^*}^* \end{pmatrix}\\ \eqqcolon \begin{pmatrix} \iota_{\textnormal{r},A_0} & \iota_{\textnormal{r},A_1^*} \end{pmatrix}
   \begin{pmatrix} a_{00}  & a_{01} \\ a_{10} & a_{11}  \end{pmatrix} \begin{pmatrix}  \iota_{\textnormal{r},A_0}^* \\\iota_{\textnormal{r},A_1^*}^* \end{pmatrix}.
\end{multline*}
We shall also define the unitary operator
\begin{equation}\label{eq:U}
   U \coloneqq\begin{pmatrix} \iota_{\textnormal{r},A_0} & \iota_{\textnormal{r},A_1^*} \end{pmatrix}\in L(\rge(A_1^*)\oplus \rge(A_0),H_1)
\end{equation}
In particluar, we then obtain
\[
  a = U \begin{pmatrix} a_{00} & a_{01} \\ a_{10} & a_{11}
\end{pmatrix} U^*
\]
With this notation at hand, we can state the first major result of this article. Keep in mind that the block operator matrix representation rests on the generalised Helmholtz decomposition in equation \eqref{eq:od}. 

\begin{theorem}[Characterisation of nonlocal $H$-convergence]\label{thm:chH} Let $\alpha,\beta>0$, $a \in L(H_1)$, and $(a_n)_n$ in $\mathcal{M}(\alpha,\beta,(A_0,A_1))$. Then the following conditions are equivalent:
\begin{enumerate}
  \item[(i)] $a$ is continuously invertible and $(a_n)_n$ nonlocally $H$-converges to $a$;
  \item[(ii)] $(a_{n,00}^{-1})_n$, $\big(a_{n,10}a_{n,00}^{-1}\big)_n$, $\big(a_{n,00}^{-1}a_{n,01}\big)_n$, and $\big(a_{n,11}-a_{n,10}a_{n,00}^{-1}a_{n,01}\big)_n$ converge in the respective weak operator topologies to $a_{00}^{-1}$, $a_{10}a_{00}^{-1}$, $a_{00}^{-1}a_{01}$, and $a_{11}-a_{10}a_{00}^{-1}a_{01}$. Moreover, we have $a\in \mathcal{M}(\alpha,\beta,(A_0,A_1))$.
\end{enumerate}
\end{theorem}

\begin{remark}
We emphasise that due to the lack of continuity of inversion (see also Proposition \ref{prop:invdis}) and the lack of joint continuity of multiplication under the weak operator topology the second item in Theorem \ref{thm:chH} does \emph{neither} imply \emph{nor} is implied by convergence of any of the sequences $(a_{n,00})_n$, $(a_{n,01})_n$, $(a_{n,10})_n$, or $(a_{n,11})_n$ under the weak operator topology.
\end{remark}

\begin{example}\label{ex:dpbc} Let $\Omega\subseteq \mathbb{R}^3$ be a bounded, open, simply connected $C^1$-domain with connected $\mathbb{R}^3\setminus{{\Omega}}$. We shall revisit Example \ref{ex:nlc} and recall that for $b\in L(V)$ with $\Re b\geq \alpha$ and $\Re \big(b^{-1}\big)\geq 1/\beta$ we have that $a(b)=\pi_{V^\bot}+\iota_V b\iota_V^*$ belongs to both $\mathcal{M}(\alpha,\beta,(\grad,\curl))$ and $\mathcal{M}(\alpha,\beta,(\interior{\grad},\interior{\curl}))$. Written as a 3-by-3 block operator matrix according to the decomposition 
\[
    L^2(\Omega)^3 = \rge(\interior{\grad})\oplus V \oplus \rge(\interior{\curl}),
\]
$a(b)$ maybe written as
\[
     \begin{pmatrix} 1 & 0 & 0 \\ 0 & b & 0 \\ 0 & 0 & 1 \end{pmatrix}.
\]
Hence, written as a 2-by-2 block operator matrix in the way introduced at the beginning of this section, we have for $(A_0,A_1)=(\interior{\grad},\interior{\curl})$
\[
    U^*a(b)U = \begin{pmatrix} 1 & \begin{pmatrix} 0 & 0 \end{pmatrix}  \\ \begin{pmatrix} 0 \\ 0 \end{pmatrix} & \begin{pmatrix} b & 0\\ 0 & 1 \end{pmatrix} \end{pmatrix}
\]
and for $(A_0,A_1)=({\grad},{\curl})$
\[
   U^*a(b)U = \begin{pmatrix} \begin{pmatrix} 1 & 0\\ 0 & b \end{pmatrix} & \begin{pmatrix} 0 \\ 0 \end{pmatrix}  \\  \begin{pmatrix} 0 & 0 \end{pmatrix} & 1 \end{pmatrix}.
\] In both cases, we have $a(b)_{01}=0$ and $a(b)_{10}=0$. 

Let now  $(b_n)_n$ be a sequence in $L(V)$ with $\Re b_n\geq \alpha, \Re \big(b_n^{-1}\big)\geq 1/\beta$, and $b\in L(V)$ invertible. Then, by Theorem~\ref{thm:chH}, we obtain that 
\[    a(b_n) \to a(b)\text{ $H$-nonlocally w.r.t.~$(\interior{\grad},\interior{\curl})$}  \iff b_n \to b\text{ in the weak operator topology.}
\]
Again by Theorem~\ref{thm:chH}, we see that
\begin{multline*}
    a(b_n) \to a(b)\text{ $H$-nonlocally w.r.t.~$({\grad},{\curl})$} \\ \iff b_n^{-1} \to b^{-1}\text{ in the weak operator topology.}
\end{multline*}
By Example \ref{ex:Helm}(c), $V$ is infinite-dimensional. Hence, by Proposition \ref{prop:invdis}, the nonlocal $H$-convergence with respect to $(\interior{\grad},\interior{\curl})$ and with respect to $({\grad},{\curl})$ are \emph{not comparable}. In particular, the nonlocal $H$-limit \emph{depends on the attached boundary conditions}.
\end{example}

\begin{corollary}\label{cor:Hu} The nonlocal $H$-limit is unique.
\end{corollary}
\begin{proof}
  Let $(a_n)_n$ nonlocally $H$-converge to invertible $a$ and $b$. By Theorem~\ref{thm:chH}, we obtain $a,b\in \mathcal{M}(\alpha,\beta,(A_0,A_1))$. Moreover, again by Theorem~\ref{thm:chH}, we deduce that $a_{00}^{-1}=b_{00}^{-1}$; thus $a_{00}=b_{00}$. Hence,
  \begin{align*}
     a_{10}a_{00}^{-1}&=\lim_{n\to\infty}a_{n,10}a_{n,00}^{-1} = b_{10}b_{00}^{-1}= b_{10}a_{00}^{-1}\text{ and }\\ a_{00}^{-1}a_{01}& =\lim_{n\to\infty}a_{n,00}^{-1}a_{n,01}= b_{00}^{-1}b_{01}=a_{00}^{-1}b_{01}.
  \end{align*}This implies
  \[
      a_{10}=b_{10}\text{ and }a_{01}=b_{01}.
  \]
  Finally,
  \begin{align*}
     a_{11} & = a_{11}-a_{10}a_{00}^{-1}a_{01} + a_{10}a_{00}^{-1}a_{01} 
     \\ &  = \lim_{n\to\infty} \big(a_{n,11}-a_{n,10}a_{n,00}^{-1}a_{n,01}\big) + a_{10}a_{00}^{-1}a_{01}
     \\ & =b_{11}-b_{10}b_{00}^{-1}b_{01} + a_{10}a_{00}^{-1}a_{01} = b_{11}.\qedhere
  \end{align*}
\end{proof}

By Theorem \ref{thm:chH} and the continuity of computing the adjoint in the weak operator topology as well as the fact that computing the inverse and computing the adjoint are commutative operations, we obtain that computing the adjoint is continuous under nonlocal $H$-convergence. We will provide more details of this line of reasoning in the following. First of all, we shall observe that the set $\mathcal{M}(\alpha,\beta,(A_0,A_1))$ is invariant under computing the adjoint.

\begin{proposition}\label{prop:adj} Let $a\in L(H_1)$, $\alpha,\beta>0$. Then
\[
    a\in \mathcal{M}(\alpha,\beta,(A_0,A_1)) \iff a^* \in \mathcal{M}(\alpha,\beta,(A_0,A_1)).
\]
\end{proposition}
\begin{proof} By $a^{**}=a$, it suffices to just prove one implication. Assume that $a\in \mathcal{M}(\alpha,\beta,(A_0,A_1))$. Since $a$ is continuously invertible, so is $a^*$. We have $U^*aU=\begin{pmatrix} a_{00} & a_{01} \\ a_{10} & a_{11} \end{pmatrix}$ with this we observe 
\[
   U^*a^*U = (U^*a U)^* = \begin{pmatrix} a_{00}^* & a_{10}^* \\ a_{01}^*  & a_{11}^* \end{pmatrix}
\]
and, using $U^*=U^{-1}$ as $U$ is unitary, we obtain
\[
   U^*(a^*)^{-1}U = (U^*a^{-1} U)^* = \begin{pmatrix} (a^{-1})_{00}^* & (a^{-1})_{10}^*   \\ (a^{-1})_{01}^* & (a^{-1})_{11}^* \end{pmatrix}.
\]
Thus, 
\begin{align*}
&   \Re (a^*)_{00}=\Re a_{00}^*=\Re a_{00}\geq \alpha, \\ & \Re \big((a^*)_{00}\big)^{-1}=\Re (a_{00}^*)^{-1}= \Re ((a_{00})^{-1})^*=\Re (a_{00})^{-1}\geq 1/\beta,
\end{align*}
and, similarly,
\begin{align*}
 & \Re \big((a^*)^{-1}\big)_{11}=\Re \big(a^{-1}\big)^*_{11}\geq 1/\beta, \\ & \Re \big((a^*)^{-1}\big)_{11}^{-1} = \Re \big(\big(a^{-1}\big)^*_{11}\big)^{-1}\geq \alpha,
\end{align*}
completing the proof.
\end{proof}

\begin{corollary}\label{cor:chH1} Let $\alpha,\beta>0$, $a \in L(H_1)$,  $(a_n)_n$ in $\mathcal{M}(\alpha,\beta,(A_0,A_1))$. Then the following conditions are equivalent:
\begin{enumerate}
 \item[(i)] $a$ is continuously invertible and $(a_n)_n$ nonlocally $H$-converges to $a$;
 \item[(ii)] $a^*$ is continuously invertible and $(a_n^*)_n$ nonlocally $H$-converges to $a^*$;
\end{enumerate}
\end{corollary}
\begin{proof}
   It suffices to prove `(i)$\Rightarrow$(ii)'. First of all note that by Proposition \ref{prop:adj}, $(b_n)_n\coloneqq (a_n^*)_n$ is a sequence in $\mathcal{M}(\alpha,\beta,(A_0,A_1))$. Moreover, denote $b\coloneqq a^*$.  By Theorem~\ref{thm:chH}, it suffices to show that the sequences
   \[
     (b_{n,00}^{-1})_n, \big(b_{n,10}b_{n,00}^{-1}\big)_n, \big(b_{n,00}^{-1}b_{n,01}\big)_n, \text{ and }\big(b_{n,11}-b_{n,10}b_{n,00}^{-1}b_{n,01}\big)_n
     \] converge in the respective weak operator topologies to $b_{00}^{-1}$, $b_{10}b_{00}^{-1}$, $b_{00}^{-1}b_{01}$, and $b_{11}-b_{10}b_{00}^{-1}b_{01}$.
  This, in turn, is implied by the convergence of the respective adjoints in the weak operator topology. Using $U^*b_nU=(U^*a_nU)^*$ similarly to the proof of Proposition \ref{prop:adj} we have for all $n\in \N$
  \begin{align*}
    (b_{n,00}^{-1})^* & = a_{n,00}^{-1} \\
    \big( b_{n,10}b_{n,00}^{-1}\big)^* & = \big(b_{n,00}^{-1}\big)^*\big( b_{n,10}\big)^* = a_{n,00}^{-1}a_{n,01} \\
    \big(b_{n,00}^{-1}b_{n,01}\big)^* & =a_{n,10} a_{n,00}^{-1}\\
    \big(b_{n,11}-b_{n,10}b_{n,00}^{-1}b_{n,01}\big)^* & =     b_{n,11}^*-b_{n,01}^*(b_{n,00}^*)^{-1}b_{n,10}^*\big)^* = a_{n,11}-a_{n,10}(a_{n,00})^{-1}a_{n,01}
  \end{align*}
  and similarly for $b$ and $a$ replacing $b_n$ and $a_n$. Since $(a_n)_n$ nonlocally $H$-converges to $a$, by Theorem \ref{thm:chH}, we thus obtain that $(a_n^*)_n$ nonlocally $H$-converges to $a^*$.
\end{proof}

The latter result implies the self-adjointness of the nonlocal $H$-limit given the self-adjointness of the sequence converging to it.

\begin{corollary}\label{cor:asa} Let $\alpha,\beta>0$, $a \in L(H_1)$ continuously invertible,  $(a_n)_n$ in $\mathcal{M}(\alpha,\beta,(A_0,A_1))$. Assume that $(a_n)_n$ nonlocally $H$-converges to $a$. If $a_n=a_n^*$ for all $n\in\mathbb{N}$, then $a=a^*$.
\end{corollary}

The proof of Theorem \ref{thm:chH} needs some preparations. 

\begin{lemma}\label{lem:boinv} Let $a\in L(H_1)$ with $a_{00}$ continuously invertible. Then
\begin{enumerate}
\item[(a)]
\[ \begin{pmatrix} a_{00} & a_{01}\\ a_{10} & a_{11}
\end{pmatrix} = \begin{pmatrix} 1 & 0\\a_{10}a_{00}^{-1} & 1
\end{pmatrix}\begin{pmatrix} a_{00} & 0 \\ 0 & a_{11}-a_{10}a_{00}^{-1}a_{01}
\end{pmatrix}\begin{pmatrix} 1 &  a_{00}^{-1} a_{01}\\0 & 1
\end{pmatrix}
\] 
\item[(b)] If, in addition, $a$ is continuously invertible, then $a_{11}-a_{10}a_{00}^{-1}a_{01}$ is and
\begin{align*}
&   \begin{pmatrix}
     \big(a^{-1}\big)_{00} & \big(a^{-1}\big)_{01}\\
     \big(a^{-1}\big)_{10} & \big(a^{-1}\big)_{11}
   \end{pmatrix}\\ & =\begin{pmatrix} 1 & -a_{00}^{-1} a_{01}\\ 0  & 1
\end{pmatrix} \begin{pmatrix} a_{00}^{-1} & 0 \\ 0 & \big(a_{11}-a_{10}a_{00}^{-1}a_{01}\big)^{-1}
\end{pmatrix}\begin{pmatrix} 1 & 0\\ -a_{10}a_{00}^{-1} & 1
\end{pmatrix}\\
& =  \begin{pmatrix}  a_{00}^{-1} + a_{00}^{-1} a_{01}\big(a_{11}-a_{10}a_{00}^{-1}a_{01}\big)^{-1}a_{10}a_{00}^{-1}& -a_{00}^{-1} a_{01}\big(a_{11}-a_{10}a_{00}^{-1}a_{01}\big)^{-1}  \\ -\big(a_{11}-a_{10}a_{00}^{-1}a_{01}\big)^{-1}a_{10}a_{00}^{-1} & \big(a_{11}-a_{10}a_{00}^{-1}a_{01}\big)^{-1}
\end{pmatrix}
\end{align*}In particular, we have $\big(a^{-1}\big)_{11}^{-1}=\big(a_{11}-a_{10}a_{00}^{-1}a_{01}\big)$ and $\big(a^{-1}\big)_{01}=   -a_{00}^{-1} a_{01} \big(a^{-1}\big)_{11}$.
\end{enumerate}
\end{lemma}
\begin{proof}
 The first assertion follows from a direct computation. The statement in (b) is in turn a straightforward consequence of the formula in (a).
\end{proof}

\begin{proof}[Proof of Theorem \ref{thm:chH}] For the proof, we refer to the solution formula for elliptic type problems in Theorem \ref{thm:varwp2}. So, let $n\in\mathbb{N}$ and let $f$ and $g$ be as in the definition of nonlocal $H$-convergence and let $u_n$ and $v_n$ be the corresponding solutions. Then we have 
\begin{align}
\label{eq:chH1}   u_n & = \mathcal{A}_0^{-1} a_{n,00}^{-1}(\mathcal{A}_0^{\diamond})^{-1}f\\
 \notag  v_n & = (\mathcal{A}_1^*)^{-1} (a_n^{-1})_{11}^{-1}((\mathcal{A}_1^*)^{\diamond})^{-1}g\\
  \label{eq:chH2}    & = (\mathcal{A}_1^*)^{-1} \big(a_{n,11}-a_{n,10}a_{n,00}^{-1}a_{n,01}\big)((\mathcal{A}_1^*)^{\diamond})^{-1}g
\end{align}
where the last equation follows from Lemma \ref{lem:boinv}(b).
With this, we infer
\begin{align*}
  a_n A_0 u_n & =   a_n UU^*A_0 u_n
  \\ & = U \begin{pmatrix} a_{n,00}& a_{n,01}
  \\ a_{n,10} & a_{n,11} 
  \end{pmatrix}  \begin{pmatrix} \mathcal{A}_0 u_n  
  \\  0
  \end{pmatrix}  \\ &   = U \begin{pmatrix} a_{n,00}& a_{n,01}
  \\ a_{n,10} & a_{n,11} 
  \end{pmatrix}  \begin{pmatrix} a_{n,00}^{-1}(\mathcal{A}_0^{\diamond})^{-1}f 
  \\ 0
  \end{pmatrix}    = U \begin{pmatrix} (\mathcal{A}_0^{\diamond})^{-1}f \\ a_{n,10} a_{n,00}^{-1}(\mathcal{A}_0^{\diamond})^{-1}f  
  \end{pmatrix}
\end{align*}
and so
\begin{equation}\label{eq:chH3}
 a_n A_0 u_n =U \begin{pmatrix} (\mathcal{A}_0^{\diamond})^{-1}f
  \\a_{n,10} a_{n,00}^{-1}(\mathcal{A}_0^{\diamond})^{-1}f\end{pmatrix}.
\end{equation}
Similarly, we compute
\begin{align*}
   a_n^{-1} A_1^* v_n & =   a_n^{-1} UU^*A_1^* v_n \\
   & =   U\begin{pmatrix} (a_{n}^{-1})_{00}& (a_n^{-1})_{01}
  \\ (a_n^{-1})_{10} & (a_n^{-1})_{11}
  \end{pmatrix} \begin{pmatrix}  0 \\ (a_n^{-1})_{11}^{-1}((\mathcal{A}_1^*)^{\diamond})^{-1}g\end{pmatrix}
  \\ & = U \begin{pmatrix}  (a_n^{-1})_{01}(a_n^{-1})_{11}^{-1}((\mathcal{A}_1^*)^{\diamond})^{-1}g \\ ((\mathcal{A}_1^*)^{\diamond})^{-1}g\end{pmatrix}.
\end{align*}
Next, from Lemma \ref{lem:boinv}(b), we deduce that
\[
  (a_n^{-1})_{01}(a_n^{-1})_{11}^{-1} = a_{n,00}^{-1} a_{n,01}.
\]
Thus,
\begin{equation}\label{eq:chH4}
  a_n^{-1} A_1^* v_n = U \begin{pmatrix}  a_{n,00}^{-1} a_{n,01}((\mathcal{A}_1^*)^{\diamond})^{-1}g \\((\mathcal{A}_1^*)^{\diamond})^{-1}g \end{pmatrix}.
\end{equation}
Next, we observe that $\mathcal{A}_0$, $\mathcal{A}_1^*$, $\mathcal{A}_0^{\diamond}$, $(\mathcal{A}_1^*)^\diamond$ are all isomorphisms by Proposition \ref{prop:varwp}. Hence, the left-hand sides of \eqref{eq:chH1}, \eqref{eq:chH2}, \eqref{eq:chH3}, and \eqref{eq:chH4} converge weakly in $\dom(\mathcal{A}_0)$, $\dom(\mathcal{A}_1^*)$, $H_1$, and $H_1$ for all admissible $f$ and $g$ to the corresponding expression with $a_n$ replaced by $a$, if and only if $a_{n,00}$, $ \big(a_{n,11}-a_{n,10}a_{n,00}^{-1}a_{n,01}\big)$, $a_{n,10} a_{n,00}^{-1}$, and $ a_{n,00}^{-1} a_{n,01}$ converge in the respective weak operator topologies to the corresponding expression without the additional index $n$.
\end{proof}

\begin{remark}
 We shall note here that the restriction to sequences $(a_n)_n$ is \emph{not} necessary. In fact, the corresponding notion of nonlocal $H$-convergence for nets $(a_\iota)_{\iota\in I}$ ($I$ some directed set), is equivalent to the convergence of the corresponding operator nets in (ii) of Theorem \ref{thm:chH}. We will exploit this fact in Section \ref{sec:metco}.
\end{remark}

A closer inspection of the proof of Theorem \ref{thm:chH} reveals the following more detailed version.
\begin{theorem}\label{thm:chHdet} Let $(a_n)_n$ in $\mathcal{M}(\alpha,\beta,(A_0,A_1))$, $a \in L(H_1)$.

Consider the following statements:

\begin{enumerate}
\item[(a)] For all $f\in \dom(\mathcal{A}_0)^*$ let $(u_n)_n$ in $\dom(\mathcal{A}_0)$ be such that
\[
    \langle a_n A_0 u_n, A_0 \phi\rangle =f(\phi)  \quad(n\in\mathbb{N}, \phi\in \dom(A_0)).
\]
Then $u_n\rightharpoonup u \in \dom(A_0)$ and 
\[
     \langle a A_0 u, A_0 \phi\rangle = f(\phi) \quad(\phi\in \dom(\mathcal{A}_0)).
\]
\item[(b)] As (a) with the additional conclusion that $a_n A_0 u_n\rightharpoonup  aA_0 u\in H_1$.
\item[(c)] Let $a$ be invertible. For all $g\in \dom(\mathcal{A}_1^*)^*$ let $(v_n)_n$ in $\dom(\mathcal{A}_1^*)$ be such that
\[
    \langle a_n^{-1} A_1^*v_n, A_1^*\psi\rangle = g(\psi) \quad(n\in\mathbb{N}, \psi\in \dom(\mathcal{A}_1^*)).
\]
Then $v_n\rightharpoonup v \in \dom(A_1^*)$ and 
\[
     \langle a^{-1} A_1^* v, A_1^* \psi\rangle =g(\psi) \quad(\psi\in \dom(\mathcal{A}_1^*)).
\]
\item[(d)] As (c) with the additional conclusion that $a_n^{-1} A_0 v_n\rightharpoonup  a^{-1}A_0 v\in H_1$.
\item[(a')] $\Re a_{00}\geq \alpha$, and $a_{n,00}^{-1}\to a_{00}^{-1}$ in the weak operator topology.
\item[(b')] As in (a') and $a_{n,10}a_{n,00}^{-1}\to a_{10}a_{00}^{-1}$ in the weak operator topology.
\item[(c')] $a_{00}$ is continuously invertible, $\Re \big( \big( a_{11}-a_{10}a_{00}^{-1}a_{01}\big)^{-1}\big)\geq 1/\beta$, and
\[
   a_{n,11}-a_{n,10}a_{n,00}^{-1}a_{n,01}\to    a_{11}-a_{10}a_{00}^{-1}a_{01}
\] 
in the weak operator topology.
\item[(d')] As in (c') with the additional conclusion that $a_{00,n}^{-1}a_{01,n}\to a_{00}^{-1}a_{01}$ in the weak operator topology.
\end{enumerate}

Then (a)$\Leftrightarrow$(a'), (b)$\Leftrightarrow$(b'), (c)$\Leftrightarrow$(c'), and (d)$\Leftrightarrow$(d'). 
\end{theorem}
\begin{proof}
 Most of the things are immediate from the reformulations \eqref{eq:chH1}, \eqref{eq:chH2}, \eqref{eq:chH3}, and \eqref{eq:chH4}. The invertibility statements follow from Lemma \ref{lem:conpd}.
\end{proof}

\begin{remark}\label{rem:GH} Let $\Omega\subseteq \mathbb{R}^3$ be such that $(A_0,A_1)=(\interior{\grad},\interior{\curl})$ is compact and exact; see again Examples \ref{ex:diffo}, \ref{ex:dfnf}. Let $(a_n)_n$ in $M(\alpha,\beta,\Omega)$ and $a\in M(\alpha,\beta,\Omega)$.

\begin{enumerate}
 \item[(a)] If $a_n=a_n^*$, $a=a^*$, and let the statement (a) of Theorem~\ref{thm:chHdet} be satisfied. This is equivalent to $(a_n)_n$ $G$-converging to $a$ (as defined in \cite[Definition 6.1]{Tartar2009}). Thus, we obtain the characterisation of $G$-convergence given in (a') and recover the main result in \cite{W16_Gcon}.
 \item[(b)] Condition (b) in Theorem~\ref{thm:chHdet}  is equivalent to $(a_n)_n$ $H$-converging to $a$ (as defined in \cite[Definition 6.4]{Tartar2009}). Hence, Theorem~\ref{thm:chHdet}(b') is an operator-theoretic description of $H$-convergence. Note that, if in addition $a_n=a_n^*$ and $a=a^*$ and assuming Theorem~\ref{thm:chHdet}(b), we also obtain 
 \[
  a_{n,00}^{-1}a_{n,01} =    \left(a_{n,10}a_{n,00}^{-1}\right)^* \to \left(a_{10}a_{00}^{-1}\right)^*  =  a_{00}^{-1} a_{01}.
 \]
 \item[(c)] Even though assuming both self-adjointness and local $H$-convergence, a suitable characterisation of the convergence of $a_{n,11}-a_{n,10}a_{n,00}^{-1}a_{n,01}$ does not follow from the reformulations outlined in the proof of Theorem \ref{thm:chH}. However, it is possible to show that $a_{n,11}-a_{n,10}a_{n,00}^{-1}a_{n,01}$ does converge to the expected limit. In Theorem~\ref{thm:lHnlH} we shall see that local $H$-convergence and nonlocal $H$-convergence are the same concepts for multiplication operators and will, thus, show the remaining convergence result even for non-selfadjoint sequences.
\end{enumerate}
\end{remark}

\section{Metrisability and compactness}\label{sec:metco}

In this section, we shall attach a topology to nonlocal $H$-convergence and show that bounded subsets of  $\mathcal{M}(\alpha,\beta,(A_0,A_1))$ are precisely the relatively compact ones under this topology. Furthermore, if $H_1$ is separable, we will show that bounded subsets are metrisable, so that the nonlocal $H$-closure of bounded subsets of $\mathcal{M}(\alpha,\beta,(A_0,A_1))$ are both compact and sequentially compact.  Again let $(A_0,A_1)$ be exact and closed. 

We recall a well-known result for the weak operator topology. Since this result is the basis for our metrisability and compactness statement for nonlocal $H$-convergence, we sketch the short proof.
\begin{theorem}\label{thm:comp} Let $H_0,H_1$ be Hilbert spaces. Then 
\[
   B_{L(H_0,H_1)}\coloneqq \{ T\in L(H_0,H_1); \|T\|\leq 1\}
\]
is compact under the weak operator topology of $L(H_0,H_1)$. If, in addition, both $H_0$ and $H_1$ are separable, then the weak operator topology of $L(H_0,H_1)$ on $B_{L(H_0,H_1)}$
is metrisable.
\end{theorem}
\begin{proof}
  Denoting the unit ball of $H_1$ endowed with the weak topology by $B^{\textnormal{w}}_{H_1}$, we obtain that
  \[
      K\coloneqq \prod_{\phi\in H_0} \|\phi\| B^{\textnormal{w}}_{H_1}
  \]
  is compact under the product topology by Tikhonov's theorem and the compactness of $B^{\textnormal{w}}_{H_1}$. It is elementary to show that $B_{L(H_0,H_1)}\subseteq K$ is closed, when $B_{L(H_0,H_1)}$ is endowed with the weak operator topology. If $H_1$ is separable, then $B^{\textnormal{w}}_{H_1}$ is metrisable. If $H_0$ is separable as well, it is then standard to prove that
  \[
   K\times K \ni     (T,S)\mapsto \sum_{n\in\mathbb{N}} 2^{-n} \min\{d(T(\phi_n),S(\phi_n)),1\}
  \]
  metrises the topology on $K$, where $d$ metrises the topology on $B^{\textnormal{w}}_{H_1}$ and $(\phi_n)_{n\in \mathbb{N}}$ is an orthonormal basis for $H_0$.
\end{proof}
We denote by $\tau_{H}$ the initial topology on $\mathcal{M}(\alpha,\beta,(A_0,A_1))$ such that
\begin{align*}
    a & \mapsto a_{00}^{-1} \in L_{\textnormal{w}}(\rge(A_0)) \\
       a & \mapsto a_{10} a_{00}^{-1} \in L_{\textnormal{w}}(\rge(A_0),\rge(A_1^*)) \\
              a & \mapsto a_{00}^{-1}a_{01}  \in L_{\textnormal{w}}(\rge(A_1^*),\rge(A_0)) \\
                            a & \mapsto a_{11}- a_{10} a_{00}^{-1}a_{01}  \in L_{\textnormal{w}}(\rge(A_1^*))
\end{align*}
are continuous, where for Hilbert spaces $K_0$ and $K_1$, $L_{\textnormal{w}}(K_0,K_1)$ denotes the set of bounded linear operators endowed with the weak operator topology.

\begin{remark}
 We note that $\tau_H$ is readily seen to be weaker than both the norm and the strong operator topology on $\mathcal{M}(\alpha,\beta,(A_0,A_1))$. Examples \ref{ex:1st} and \ref{ex:2nd} show that the weak operator topology on $\mathcal{M}(\alpha,\beta,(A_0,A_1))$ and $\tau_H$ cannot be compared in general.
\end{remark}

The following is a reformulation of Theorem~\ref{thm:chH}.

\begin{theorem}\label{thm:chH4}Let $(a_n)_n$ in $\mathcal{M}(\alpha,\beta,(A_0,A_1))$, $a \in L(H_1)$ invertible. Then the following conditions are equivalent:
\begin{enumerate}
 \item[(i)] $(a_n)_n$ nonlocally $H$-converges to $a$;
 \item[(ii)] $a_n \stackrel{\tau_H}{\to}a\in \mathcal{M}(\alpha,\beta,(A_0,A_1))$. 
 \end{enumerate}
\end{theorem}

Theorem~\ref{thm:chH4} shows that nonlocal $H$-convergence (of sequences) is actually induced by the topology $\tau_{H}$. Next, we show that $\tau_{H}$ is a Hausdorff topology. Together with Theorem~\ref{thm:chH4}, this yields another proof of Corollary~\ref{cor:Hu}, the uniqueness of the nonlocal $H$-limit.

\begin{proposition}\label{prop:T2} $\tau_H$ is a Hausdorff topology.
\end{proposition}
\begin{proof}
 Let $a,b\in \mathcal{M}(\alpha,\beta,(A_0,A_1))$ with $a\neq b$. It follows that (at least) one of the equalities
 \begin{align*}
 &  a_{00}^{-1} = b_{00}^{-1}\\
 & a_{10} a_{00}^{-1} =  b_{10} b_{00}^{-1}  \\
 & a_{00}^{-1}a_{01}  = b_{00}^{-1}b_{01}\\
 &  a_{11}- a_{10} a_{00}^{-1}a_{01} =  b_{11}- b_{10} b_{00}^{-1}b_{01}
 \end{align*}
 cannot be true. Since the weak operator topology is a Hausdorff topology, we find a suitable continuous semi-norm $p$ such that one of the following 4 statements is true
 \begin{align*}
 &p\big(  a_{00}^{-1}\big) \neq p\big(b_{00}^{-1}\big)\\
 & p\big(a_{10} a_{00}^{-1}\big) \neq  p\big(b_{10} b_{00}^{-1} \big) \\
 & p\big(a_{00}^{-1}a_{01} \big) \neq p\big(b_{00}^{-1}b_{01}\big)\\
 &  p\big(a_{11}- a_{10} a_{00}^{-1}a_{01}\big) \neq  p\big(b_{11}- b_{10} b_{00}^{-1}b_{01}\big).
 \end{align*}
 This implies the assertion.
\end{proof}

The following result is the announced compactness statement.
\begin{theorem}\label{thm:Hcom} The set
\[
  \mathcal{M}_1(\alpha,\beta,(A_0,A_1))\coloneqq \{ a\in \mathcal{M}(\alpha,\beta,(A_0,A_1)); \|a_{00}^{-1}a_{01}\|,\|a_{10}a_{00}^{-1}\|\leq \beta\}
\]
is compact under $\tau_{H}$.
\end{theorem}
\begin{proof}
  Let $(a_\iota)_{\iota}$ be a net in $\mathcal{M}_1(\alpha,\beta,(A_0,A_1))$. By Theorem~\ref{thm:comp}, we may choose a subnet $(a_{\phi(\iota')})_{\iota'}$ such that $a_{\phi(\iota'),00}^{-1}\to b_{00}^{-1}$ for some $b_{00}\in L(\rge(A_0))$ with $\Re b_{00}\geq \alpha$ and $\Re \big(b_{00}^{-1}\big)\geq 1/\beta$. By the boundedness of $(a_\iota)_\iota$, we infer that $(a_{\iota,01})_\iota$ and $(a_{\iota,10})_\iota$ are bounded. Again using Theorem~\ref{thm:comp}, we find a subnet such that
  \[
     a_{\phi'(\iota''),10} a_{\phi'(\iota''),00}^{-1}b_{00}\to b_{10} \text{ and } b_{00}a_{\phi'(\iota''),00}^{-1}a_{\phi'(\iota''),01} \to b_{01}
  \]
  for some $b_{10}\in L(\rge(A_1^*),\rge(A_0))$ and $b_{01}\in L(\rge(A_0),\rge(A_1^*))$. Finally, we find a subnet such that 
  \[
     a_{\phi''(\iota'''),11}-  a_{\phi''(\iota'''),10} a_{\phi''(\iota'''),00}^{-1} a_{\phi''(\iota'''),01} + b_{10} b_{00}^{-1} b_{01}\to b_{11}
  \]
  for some $b_{11}\in L(\rge(A_1^*))$. It is easy to see that $\Re b_{00}\geq \alpha$ and $\Re \big(b_{00}^{-1}\big)\geq 1/\beta$ (see also Lemma \ref{lem:conpd}(d)). Similarly, using Lemma~\ref{lem:boinv} for $(a_{\iota}^{-1})_{11}^{-1}= a_{\iota,11}-  a_{\iota,10} a_{\iota,00}^{-1} a_{\iota,01}$, it follows that $\Re (b_{11}-b_{10}b_{00}^{-1}b_{01})\geq \alpha$ and $\Re\big( (b_{11}-b_{10}b_{00}^{-1}b_{01})^{-1}\big)\geq 1/\beta$ (see also Lemma \ref{lem:conpd}). Next, using Lemma~\ref{lem:boinv}(a), we obtain
  \[
     \begin{pmatrix} b_{00} & b_{01}\\ b_{10} & b_{11}
\end{pmatrix} = \begin{pmatrix} 1 & 0\\ b_{10}b_{00}^{-1} & 1
\end{pmatrix}\begin{pmatrix} b_{00} & 0 \\ 0 & b_{11}-b_{10}b_{00}^{-1}b_{01}
\end{pmatrix}\begin{pmatrix} 1 & b_{00}^{-1} b_{01}\\ 0 & 1
\end{pmatrix}.
  \]
  Thus, we deduce that $b$ is continuously invertible. Hence, $b\in \mathcal{M}(\alpha,\beta,(A_0,A_1))$. Moreover, it is now easy to see that $a_{\phi''(\iota''')}\stackrel{\tau_H}\to b$. It remains to show that $\|b_{00}^{-1}b_{01}\|\leq \beta,\|b_{10}b_{00}^{-1}\|\leq \beta$. For this, we use lower semi-continuity of the operator norm under the weak operator topology. Thus, writing $\iota$ instead of $\phi''(\iota''')$ for simplicity, we obtain
  \begin{align*}
   \| b_{00}^{-1}b_{01}\|&=\| \lim_{\iota} a_{\iota,00}^{-1}a_{\iota,01}\|\leq \liminf_{\iota}\|  a_{\iota,00}^{-1}a_{\iota,01}\|\leq \beta
   \\    \| b_{01}b_{00}^{-1}\|&=\| \lim_{\iota} a_{\iota,01}a_{\iota,00}^{-1}\|\leq \liminf_{\iota}\| a_{\iota,01} a_{\iota,00}^{-1}\|\leq \beta.   
  \end{align*}Hence, $b\in \mathcal{M}_1(\alpha,\beta,(A_0,A_1))$.
\end{proof}
\begin{remark}
 Let $\mathcal{B}\subseteq \mathcal{M}(\alpha,\beta,(A_0,A_1))$ be bounded in $L(H_1)$. Then we find $\alpha',\beta'\in \mathbb{R}$ such that $\mathcal{B}\subseteq \mathcal{M}_1(\alpha',\beta',(A_0,A_1))$. Indeed, $\alpha'=\alpha$ and $\beta'=\sup\{\|b_{01}b_{00}^{-1}\|\lor \|b_{00}^{-1}b_{01}\|; b\in\mathcal{B}\}\lor \beta$ are possible choices. As a consequence, we obtain with Theorem~\ref{thm:Hcom} that $\mathcal{B}$ is relatively compact under $\tau_{H}$.
\end{remark}

Let us revisit Examples \ref{ex:dpbc} and \ref{ex:mcv}. 

\begin{example}\label{ex:mcv2} We shall use the notation and operators introduced in Example~\ref{ex:mcv}. We have already seen that there exists $\alpha,\beta>0$ such that $1-k_n*\in \mathcal{M}(\alpha,\beta,(\interior{\grad},\interior{\curl}))$ for all $n\in\mathbb{N}$. Moreover, we have assumed that $(1-k_n*)_n$ is a bounded sequence in $L(L^2(\Omega)^3)$. Thus, by Theorem~\ref{thm:met}, we find a strictly increasing sequence of positive integers $\kappa\colon \mathbb{N}\to \mathbb{N}$ such that $(1-k_{\kappa(n)}*)_n$ is nonlocally $H$-convergent to some $a\in \mathcal{M}(\alpha,\beta,(\interior{\grad},\interior{\curl}))$. We emphasise that the used compactness statement does not lead to the statement that $a=(1-k*)$ for some convolution-type kernel $k\in L^\infty(\Omega\times \Omega)$. More refined arguments (or assumptions) are needed to actually deduce that $a$ has the desired form.
\end{example}

\begin{example}\label{ex:dpbc2} Use the assumptions and operators introduced in Example~\ref{ex:dpbc}. We have seen that for a sequence $(b_n)_n$ in $L(V)$ satisfying $\Re b_n\geq \alpha$ and $\Re \big(b_n^{-1}\big)\geq 1/\beta$ for all $n\in \mathbb{N}$ and an invertible $b\in L(V)$ that 
\[    a(b_n) \to a(b)\text{ $H$-nonlocally w.r.t.~$(\interior{\grad},\interior{\curl})$}  \iff b_n \to b\text{ in the weak operator topology}
\]and
\begin{multline*}
    a(b_n) \to a(b)\text{ $H$-nonlocally w.r.t.~$({\grad},{\curl})$} \\ \iff b_n^{-1} \to b^{-1}\text{ in the weak operator topology.}
\end{multline*}
Thus, in these special cases, with an application of Theorem~\ref{thm:met},  we obtain special cases of Theorem~\ref{thm:comp} for $H_0=H_1=V$ in the separable case.
\end{example}

\begin{remark} (a) We have $  \mathcal{M}(\alpha,\beta,(A_0,A_1)) \nsubseteq \{ a \in L(H_1); \Re a\geq \alpha', \Re \big(a^{-1}\big)\geq 1/\beta'\}$ for any $\alpha', \beta'>0$. Indeed, $a= U\left(\begin{smallmatrix} 1 &1 \\ 2  & (9-\epsilon)/4 \end{smallmatrix}\right)U^*$ for any $\epsilon\in (0,1)$ with $U$ as in \eqref{eq:U} belongs to $\mathcal{M}(\alpha,\beta,(A_0,A_1)) $ for some $\alpha,\beta>0$ but fails to satisfy $\Re a\geq 0$.

(b) Let $b=b^*\in \mathcal{M}(\alpha,\beta,(A_0,A_1))$. Then $\Re b=b\geq \alpha'$ for some $\alpha'>0$ (and consequently $\Re \big(b^{-1}\big)\geq 1/\beta'$ for some $\beta'>0$). Indeed, we have for all $\psi\coloneqq \begin{pmatrix} 1 & 0\\ - b_{00}^{-1} b_{01} & 1
\end{pmatrix}\phi$:
   \[
    \Big \langle \begin{pmatrix} b_{11} & b_{10}\\ b_{01} & b_{00}
\end{pmatrix} \psi,\psi\Big\rangle = \Big\langle \begin{pmatrix} b_{11}-b_{10}b_{00}^{-1}b_{01} & 0 \\ 0 & b_{00}
\end{pmatrix}\phi, \phi \Big\rangle\geq \alpha \langle \phi,\phi\rangle\geq \left(\alpha/(1+\|b_{00}^{-1} b_{01}\|)^2\right)\|\psi\|^2.
  \]
\end{remark}

\begin{theorem}\label{thm:met} Assume $H_1$ to be separable. Then $(\mathcal{M}_1(\alpha,\beta,(A_0,A_1)),\tau_H)$ is metrisable and sequentially compact.
\end{theorem}
\begin{proof}
 Since $\rge(A_0)\subseteq H_1$ and $\rge(A_1^*)\subseteq H_1$ both these subsets are separable. We abbreviate $\mathcal{M}\coloneqq \mathcal{M}_1(\alpha,\beta,(A_0,A_1))$. We put
 \begin{align*}
   &\Phi_{00} \colon \mathcal{M} \ni a\mapsto a_{00}^{-1} \in \beta B_{L(\rge(A_0))}\\
   &\Phi_{01} \colon \mathcal{M} \ni a\mapsto a_{00}^{-1}a_{01} \in \beta B_{L(\rge(A_1^*),\rge(A_0))}\\
   &\Phi_{10} \colon \mathcal{M} \ni a\mapsto a_{10}a_{00}^{-1} \in \beta B_{L(\rge(A_0),\rge(A_1^*))}\\
   &\Phi_{11} \colon \mathcal{M} \ni a\mapsto a_{11}- a_{10}a_{00}^{-1}a_{01} \in \beta B_{L(\rge(A_1^*))}.
 \end{align*}
 By Theorem~\ref{thm:comp} there exists metrics $d_{00}$, $d_{01}$, $d_{10}$, and $d_{11}$ inducing the weak operator topology on $\beta B_{L(\rge(A_0))}$, $\beta B_{L(\rge(A_1^*),\rge(A_0))}$, $\beta B_{L(\rge(A_0),\rge(A_1^*))}$, and $\beta B_{L(\rge(A_1^*))}$. We define
 \begin{align*}
   d_H \colon \mathcal{M}\times \mathcal{M} & \to [0,\infty)\\
                (a,b)&\mapsto \sum_{j,k \in\{0,1\}} d_{jk}(\Phi_{jk}(a),\Phi_{jk}(b)).
 \end{align*}
 As in the proof of Proposition~\ref{prop:T2}, we verify that $(\mathcal{M},d_H)$ is a metric space. Moreover, by definition, the identity mapping
 \[
     (\mathcal{M},\tau_H) \hookrightarrow (\mathcal{M},d_H) 
 \]
 is continuous and onto. Since $(\mathcal{M}, d_H)$ is a Hausdorff space and $(\mathcal{M},\tau_H)$ is compact by Theorem~\ref{thm:Hcom}, we infer that $  (\mathcal{M},\tau_H) \hookrightarrow (\mathcal{M},d_H) $ is a homeomorphism. Hence, $  (\mathcal{M},\tau_H) $ is metrisable. Sequential compactness is now immediate since compact metric spaces are sequentially compact.
\end{proof}

We draw an important consequence of the compactness result, which establishes the connection from local to nonlocal $H$-convergence. We recall Example \ref{ex:diffo}(a2) and Example \ref{ex:dfnf} in order to deduce $(\interior{\grad},\interior{\curl})$ is compact and exact, if the underlying domain $\Omega$ is a bounded weak Lipschitz domain with $\mathbb{R}^3\setminus \Omega$ connected. In fact, due to our abstract reasoning, the assumption of $(\interior{\grad},\interior{\curl})$ being compact and exact is the assumption, we actually need in the next statement.

\begin{theorem}\label{thm:lHnlH} Let $\Omega\subseteq \mathbb{R}^3$ be a bounded weak Lipschitz domain with connected complement.  Let $(a_n)_n$ in $M(\alpha,\beta,\Omega)$, $a\in M(\alpha,\beta,\Omega)$. Then the following conditions are equivalent:
\begin{enumerate}
 \item[(i)] $(a_n)_n$ locally $H$-converges to $a$, that is, for all $f\in H^{-1}(\Omega)$ and corresponding solutions $(u_n)_n$ in $H_0^1(\Omega)$ of
 \[
    \langle a_n \grad u_n ,\grad \phi \rangle =  f(\phi) \quad (\phi\in H_0^1(\Omega))
 \]
 we have $u_n \rightharpoonup u\in H_0^1(\Omega)$, $a_n \grad u_n \rightharpoonup a\grad u$, where $u$ in $H_0^1(\Omega)$ satisfies
 \[
    \langle a \grad u ,\grad \phi \rangle =  f(\phi) \quad (\phi\in H_0^1(\Omega)).
 \]
 \item[(ii)] $(a_n)_n$ nonlocally $H$-converges to $a$ with respect to $(\interior{\grad},\interior{\curl})$.
\end{enumerate}
\end{theorem}
\begin{proof}
  The implication `(ii)$\Rightarrow$(i)' has been settled in Remark~\ref{rem:GH}(b) together with Theorem~\ref{thm:chH} (see also Theorem~\ref{thm:chHdet}). We shall assume (i). By Theorem~\ref{thm:Hcom}, we may choose a subsequence $(a_{\kappa(n)})_n$ of $(a_n)_n$, which nonlocally $H$-converges to some $b$. From the implication `(ii)$\Rightarrow$(i)' it follows that $(a_{\kappa(n)})_n$ locally $H$-converges to $b$.  Since local $H$-convergence is induced by a topology, see \cite[p. 82]{Tartar2009}, we deduce that $(a_{\kappa(n)})_n$ locally $H$-converges to $a$. By uniqueness of the local $H$-limit (see again \cite[p. 82]{Tartar2009}), we obtain $a=b$. A subsequence principle concludes the proof.
\end{proof}

\begin{remark}\label{rem:neu} Given $\Omega\subseteq \mathbb{R}^3$ a simply connected bounded weak Lipschitz domain  in order that $(\grad,\curl)$ is compact and exact; see Examples \ref{ex:diffo} and \ref{ex:dfnf}. Let $(a_n)_n$ and $a$ belong to $M(\alpha,\beta,\Omega)$. By \cite[Lemma 10.3]{Tartar2009} local $H$-convergence is independent of the attached boundary conditions. Thus, in particular, with an analogous proof to the one in Theorem~\ref{thm:lHnlH}, it is possible to show that $(a_n)_n$ locally $H$-converges to $a$, if and only if $(a_n)_n$ nonlocally $H$-converges to $a$ with respect to $(\grad,\curl)$. 
\end{remark}

\begin{remark} Another way of stating Theorem~\ref{thm:lHnlH} is the following. Let $\tau_{\textrm{loc}H}$ be the (metrisable) topology induced on $M(\alpha,\beta,\Omega)$ by local $H$-convergence. Then
\[
 (  M(\alpha,\beta,\Omega),\tau_H) \hookrightarrow  (  M(\alpha,\beta,\Omega),\tau_{\textrm{loc}H})
\] is a homeomorphism. Note that \cite[Theorem 6.5]{Tartar2009} states that $(  M(\alpha,\beta,\Omega),\tau_{\textrm{loc}H})$ is sequentially compact. Hence, so is $ (  M(\alpha,\beta,\Omega),\tau_H)$.
\end{remark}

An immediate corollary is a homogenisation result for elliptic equations involving the $\curl$-operator. We also refer to the explicit descriptions of the domain of the $\curl$-operator derived in Example~\ref{ex:np1}(b).

\begin{corollary}\label{cor:lHnlH} Let $\Omega\subseteq \mathbb{R}^3$  be a bounded weak Lipschitz domain with connected complement.  Let $(a_n)_n$ in $M(\alpha,\beta,\Omega)$, $a\in M(\alpha,\beta,\Omega)$. Assume that $(a_n)_n$ locally $H$-converges to $a$.

Then, for all $g \in H_{\textnormal{sol}}(\curl,\Omega)^*$ and solutions $(v_n)_n$ in $ H_{\textnormal{sol}}(\curl,\Omega)$ of
 \[
    \langle a_n^{-1} \curl v_n ,\curl \psi \rangle =  g(\psi) \quad (\psi\in H_{\textnormal{sol}}(\curl,\Omega)),
 \]
 we have $v_n \rightharpoonup v\in H_{\textnormal{sol}}(\curl,\Omega)$, $a_n^{-1} \curl v_n \rightharpoonup a^{-1}\curl v\in L^2(\Omega)^3$, where $v\in H_{\textnormal{sol}}(\curl,\Omega)$ satisfies
 \[
    \langle a^{-1} \curl v ,\curl \psi \rangle = g(\psi) \quad (\psi\in H_{\textnormal{sol}}(\curl,\Omega)).
 \]
\end{corollary}

\begin{remark} (a) In the light of Remark \ref{rem:GH}, we note that Corollary \ref{cor:lHnlH} particularly settles the convergence of $a_{n,11}-a_{n,10}a_{n,00}^{-1}a_{n,01}\to a_{11}-a_{10}a_{00}^{-1}a_{01}$ as $n\to\infty$ in the weak operator topology.

(b) As a consequence of Remark \ref{rem:neu}, we deduce that a similar results hold, where we replace $\curl$ by $\interior{\curl}$.
\end{remark}

\section{A div-curl type characterisation}\label{sec:dct}

Throughout this section, we shall again assume that $(A_0,A_1)$ is closed and exact.

In this section, we want to prove another characterisation of nonlocal $H$-convergence. In fact, this is the characterisation one uses in applications and can thus be viewed as the main abstract result, when characterising nonlocal $H$-convergence. We need variants of the operators $\mathcal{A}_0^{\diamond}$ and $(\mathcal{A}_1^*)^{\diamond}$ that are defined on the whole of $H_1$. We put for all $\phi\in H_1$
\[
   \mathcal{A}_{0,\textnormal{k}}^{\diamond} (\phi) = \mathcal{A}_{0}^\diamond(\pi_0 \phi)\text{ and }   (\mathcal{A}_{1}^*)_{\textnormal{k}}^{\diamond} (\phi) = (\mathcal{A}_{1}^*)^\diamond(\pi_1 \phi),
\]where $\pi_0$ and $\pi_1$ are the orthogonal projections on $\rge(A_0)=\kar(A_0^*)^\bot$ and $\rge(A_1^*)=\kar(A_1)^\bot$. Note that this definition is consistent with $A_0^*$ and $A_1$ in the sense that we have
\[
     \mathcal{A}_{0,\textnormal{k}}^{\diamond} = A_0^*\text{ on } \dom (A_0^*)
\]
and
\[
    (\mathcal{A}_{1}^*)_{\textnormal{k}}^{\diamond} = A_1  \text{ on } \dom (A_1).
\]
\begin{example}\label{ex:cala2} Recall the setting of Example \ref{ex:cala}(a). We have realised that $\mathcal{A}_0\colon H_0^1(\Omega)\to \rge(\interior{\grad}), u\mapsto \grad u$. Then it is not hard to see that 
\[
\mathcal{A}_0^\diamond \colon \rge(\interior{\grad}) \to H^{-1}(\Omega), q\mapsto \dive q.
\]
On the other hand $\dive q=0$ for all $q\in \rge(\interior{\grad})^\bot=\kar(\dive)\subseteq L^2(\Omega)^3$, we deduce that
\[
   \mathcal{A}_{0,\textnormal{k}}^\diamond \colon L^2(\Omega)^3 \to H^{-1}(\Omega), q\mapsto \dive q.
\]
\end{example}

\begin{theorem}\label{thm:chH3} Let $(a_n)_n$ in $\mathcal{M}(\alpha,\beta,(A_0,A_1))$ be bounded, $a\in L(H_1)$, $H_1$ separable.
Then the following statements are equivalent:
\begin{enumerate}
  \item[(i)] $a\in \mathcal{M}(\alpha,\beta,(A_0,A_1))$, and $(a_n)_n$ nonlocally $H$-converges to $a$;
  \item[(ii)]  for all $(q_n)_n$ in $H_1$ weakly convergent to some $q$ in $H_1$ and for all  $\kappa\colon \mathbb{N}\to\mathbb{N}$ strictly monotone we have: Given the two conditions
    \begin{enumerate}
      \item[(a)] $(\mathcal{A}_{0,\textnormal{k}}^\diamond (a_{\kappa(n)} q_n))_n$ is relatively compact in $\dom(\mathcal{A}_0)^*$,
      \item[(b)] $((\mathcal{A}_1^*\big)_{\textnormal{k}}^\diamond (q_n))_n$ is relatively compact in $\dom(\mathcal{A}_1^*)^*$,
    \end{enumerate}
    then $a_{\kappa(n)}q_n\rightharpoonup aq$ as $n\to\infty$.
\end{enumerate}
\end{theorem}

\begin{remark}
 In the proof of Theorem~\ref{thm:chH3} the separability of $H_1$ is used only in the implication `(ii)$\Rightarrow$(i)', where we employ sequential compactness of $\mathcal{M}(\alpha,\beta,(A_0,A_1))$ under the topology induced by nonlocal $H$-convergence. We included the separability assumption for convenience. For the seemingly relatively rare occasions, where non-separable Hilbert spaces are considered, we note that the corresponding  reformulation of Theorem \ref{thm:chH3} invokes (sub)nets rather than (sub)sequences. 
\end{remark}

\begin{remark} (a) For the particular case of periodic multiplication operators in $L^2(\Omega)^3$ with $a_n=a_n^*$ so that $(a_n)_n$ locally $H$-converges to $a_{\textnormal{hom}}$, where $a_{\textnormal{hom}}$ is the usual homogenised constant coefficient matrix, the implication `(i)$\Rightarrow$(ii)' is contained in \cite[Lemma 4.5]{Jikov1994}.

(b) In case of local $H$-convergence a variant of Theorem~\ref{thm:chH3} has been stated in \cite[p. 10]{Tartar1997}.
\end{remark}

An application of Theorem~\ref{thm:lHnlH} yields another characterisation of local $H$-convergence. To the best of the author's knowledge this characterisation has not been pointed out in the literature, yet. In any case, the only important point is that $(\interior{\grad},\interior{\curl})$ is closed and exact; see Examples \ref{ex:diffo} and \ref{ex:dfnf}.
\begin{theorem}\label{thm:DCLch} Let $\Omega\subseteq \mathbb{R}^3$ be an open, bounded weak Lipschitz domain with connected complement.  Let $(a_n)_n$ in $M(\alpha,\beta,\Omega)$, $a\in M(\alpha,\beta,\Omega)$. Then the following statements are equivalent:
\begin{enumerate}
  \item[(i)] $(a_n)_n$ locally $H$-converges to $a$;
  \item[(ii)] for all $(q_n)_n$ in $L^2(\Omega)^3$ weakly convergent to some $q$ in $L^2(\Omega)^3$ and $\kappa\colon \mathbb{N}\to\mathbb{N}$ strictly monotone we have: Given the conditions
    \begin{enumerate}
      \item[(a)] $(\dive (a_{\kappa(n)} q_n))_n$ is relatively compact in $H^{-1}(\Omega)$,
      \item[(b)] $(\interior{\curl} (q_n))_n$ is relatively compact in $H_{\textnormal{sol}}(\curl,\Omega)^*$,
    \end{enumerate}
    then $a_{\kappa(n)}q_n\rightharpoonup aq$ as $n\to\infty$.
\end{enumerate}
\end{theorem}

\begin{remark} We note that in Theorem~\ref{thm:DCLch} (with $\Omega$ that admit a continuous extension operator $H^2(\Omega) \to H^2(\mathbb{R}^3)$; by Calderon's extension theorem strong Lipschitz boundary is enough), it is possible to replace $H_{\textnormal{sol}}(\curl,\Omega)^*$ by $H^{-1}(\Omega)$. We refer to \cite[(the proof of) Proposition 3.10]{W17_DCL} for the details.  
\end{remark}

The next example revisits the Examples \ref{ex:mcv} and \ref{ex:mcv2}, which is used in the already mentioned McKean--Vlasov model (\cite{Carillo18}) and in the so-called nonlocal response theory, see \cite[Chapter 10]{Keller2011} as well as \cite{Gorlach2016,Ciattoni2015,Mendez2017}. $\Omega$ is assumed to be bounded and such that $(\interior {\grad},\interior{\curl})$ is closed and exact, which by Examples \ref{ex:diffo} and \ref{ex:dfnf} for instance corresponds to $\Omega$ being a bounded Lipschitz domain with connected complement.

We furthermore note that Theorem~\ref{thm:DCLch} provides the desired characterisation for nonlocal $H$-convergence, which avoids explicitly computing the inverses of the operators considered. In fact, this solves the problem we have encountered at the end of Example \ref{ex:mcv}. Moreover, assuming more regularity of the integral kernels, we are also in the position to answer a part of the question raised in Example \ref{ex:mcv} and specified at the end of Example \ref{ex:mcv2}.

\begin{example}\label{ex:mcv3} Let $\Omega\subseteq\mathbb{R}^3$ be bounded and such that $(\interior {\grad},\interior{\curl})$ is exact and closed.

Let $(a_n)_n$ be nonlocally $H$-converges to $a$ with respect to $(\interior {\grad},\interior{\curl})$. Let $k_n *\phi \coloneqq(x\mapsto  \int_{\Omega} k_n(x-y) \phi(y) dy)$ for some bounded sequence $(k_n)_n$ in $W^{1,\infty}(\mathbb{R}^3)$. Assume that $(k_n)_n$ converges in the weak*-topology to some $k\in L^\infty(\mathbb{R}^3)$. Assume further that there exists $c>0$ such that
  \[
       a_n + k_n* =  (a_n + k_n*)^*\geq c
  \]
  Note that then we find $\alpha,\beta>0$ such that $a_n+k_n* \in \mathcal{M}(\alpha,\beta,(\interior {\grad},\interior{\curl}))$ for all $n\in\mathbb{N}$.
  
  Then $(a_n+k_n*)_n$ nonlocally $H$-converges to $a+k*$.
  
  In order to establish the claim, we will apply the div-curl type characterisation from Theorem~\ref{thm:chH3} (Theorem~\ref{thm:DCLch}). So, let $(q_n)_n$ be a weakly convergent sequence in $L^2(\Omega)^3$ with limit $q$. Further, let $\kappa\colon \mathbb{N}\to \mathbb{N}$ be strictly monotone and assume that
  \begin{enumerate}
   \item[(a)] $(\dive (a_{\kappa(n)}+k_{\kappa(n)}*)q_n)_n$ is relatively compact in $H^{-1}(\Omega)$,
   \item[(b)] $(\interior{\curl} q_n)_n$ is relatively compact in $H_\textnormal{sol}(\curl,\Omega)^*$.
  \end{enumerate}
  
  Note that since $\partial_j k_n \in L^\infty(\mathbb{R}^3)$ for all $j\in\{1,2,3\}$, we obtain that 
  \[
  \dive (k_{\kappa(n)}*q_n) =  \sum_{j=1}^3 \partial_j k_{\kappa(n)}*q_{n,j} \in L^2(\Omega),
  \]
  uniformly in $n$. By the compactness of the embedding $H_0^1(\Omega)\hookrightarrow L^2(\Omega)$, we deduce that
  $ ( \dive k_{\kappa(n)}*q_n )_n$ is relatively compact in $H^{-1}(\Omega)$. Thus, condition (a), yields that $(\dive (a_{\kappa(n)} q_n))_n$ is relatively compact in $H^{-1}(\Omega)$. Thus, by nonlocal $H$-convergence of $(a_n)_n$ to $a$ and Theorem~\ref{thm:chH3}, we infer that $a_{\kappa(n)}q_n \rightharpoonup aq$. Thus, we are left with proving that
  \[
     k_{\kappa(n)}*q_n \rightharpoonup k * q.
  \]
  For this, let $\phi\in L^2(\Omega)^3$ and consider
  \[
     \langle k_{\kappa(n)}*q_n, \phi\rangle  = \langle q_n, k_{\kappa(n)}*\phi \rangle.
  \]
  Next, we see that $(\dive  k_{\kappa(n)}*\phi )_n$ is bounded in $L^2(\Omega)$ and so relatively compact in $H^{-1}(\Omega)$ by the boundedness of $\Omega$. Moreover, we compute for $\psi\in L^2(\Omega)^3$
  \begin{align*}
    \langle k_{\kappa(n)}*\phi,\psi \rangle & = \int_\Omega \langle k_{\kappa(n)}*\phi(x),\psi(x)\rangle dx
    \\ & = \int_\Omega \Big\langle \int_\Omega k_{\kappa(n)}(x-y)\phi(y),\psi(x)\Big\rangle dx.
  \end{align*}
  Since $\Omega $ is bounded, we infer that $\phi\in L^1(\Omega)^3$. Moreover, it is easy to see that $k_n$ converging weakly* to $k$ implies that $k_n(x-\cdot)$ converging weakly* to $k(x-\cdot)$. Thus, we infer by dominated convergence
  \[
    \langle k_{\kappa(n)}*\phi,\psi \rangle \to \langle k*\phi,\psi \rangle.
  \]
   By condition (b) and Theorem~\ref{thm:DCL} below, we thus infer 
  \[
   \langle k_{\kappa(n)}*q_n, \phi\rangle  = \langle q_n, k_{\kappa(n)}*\phi \rangle \to \langle q, k*\phi \rangle = \langle k*q, \phi \rangle,
  \]
  which shows the assertion.
\end{example}

The proof of Theorem \ref{thm:chH3} needs some prerequisites. The first one is a global div-curl type result, see \cite[Theorem 2.4]{W17_DCL}; see also \cite{Pauly2017} for several applications and \cite[Theorem 3.1]{Chen2017} for a Banach space setting. We shall furthermore refer to \cite{Misur2017} and the references therein for a guide to the literature for other results and approaches to the div-curl lemma.

\begin{theorem}[{{\cite[Theorem 2.4]{W17_DCL}}}]\label{thm:DCL} Let $(q_n)_n, (r_n)_n$ be weakly convergent in $H_1$. Assume that
\[
    (\mathcal{A}_{0,\textnormal{k}}^\diamond q_n)_n\ \text{and} \ \big((\mathcal{A}_1^*)_{\textnormal{k}}^\diamond r_n\big)_n
\] 
are relatively compact in $\dom(\mathcal{A}_0)^*$ and $\dom(\mathcal{A}_1^*)^*$, respectively.

Then
\[
    \lim_{n\to\infty} \langle q_n,r_n \rangle_{H_1} =\Big\langle \textnormal{w-}\lim_{n\to\infty}q_n, \textnormal{w-}\lim_{n\to\infty}r_n\Big\rangle_{H_1}.
\] 
\end{theorem}

For easy reference, we will use $\pi_0$ and $\pi_1$ for the orthogonal projections in $H_1$ projecting on $\rge(A_0)$ and $\rge(A_1^*)$, respectively.

\begin{lemma}\label{lem:sys} Let $a\in \mathcal{M}(\alpha,\beta,(A_0,A_1))$. Let $v,w\in H_1$. Then the following conditions are equivalent:
\begin{enumerate}
  \item[(i)] $w=av$;
  \item[(ii)] $\pi_0 w = \pi_0 av$ and $\pi_1 v = \pi_1 a^{-1}w$.
\end{enumerate}
\end{lemma}
\begin{proof}
Note that (i) trivially implies (ii). For the other implication, we use the block matrix representation worked out in Lemma \ref{lem:boinv}. Condition (ii) is equivalent to
\begin{equation}\label{eq:sys1}\begin{pmatrix} a_{00} & a_{01}\\ 0 & 0
\end{pmatrix} \begin{pmatrix} \pi_0 v \\ \pi_1 v\end{pmatrix}  = \begin{pmatrix} \pi_0w \\ 0\end{pmatrix}
\end{equation} and
\begin{multline*}
\begin{pmatrix} 0 \\ \pi_1 v\end{pmatrix} =   \begin{pmatrix} 0 & 0 \\ 0 & 1   
   \end{pmatrix}   \begin{pmatrix}
     \big(a^{-1}\big)_{00} & \big(a^{-1}\big)_{01}\\
     \big(a^{-1}\big)_{10} & \big(a^{-1}\big)_{11}
   \end{pmatrix}\begin{pmatrix} \pi_0 w \\ \pi_1 w\end{pmatrix} \\ 
   = \begin{pmatrix} 0 & 0 \\ -\big(a_{11}-a_{10}a_{00}^{-1}a_{01}\big)^{-1}a_{10}a_{00}^{-1} & \big(a_{11}-a_{10}a_{00}^{-1}a_{01}\big)^{-1}
\end{pmatrix}\begin{pmatrix} \pi_0 w \\ \pi_1 w\end{pmatrix}
\end{multline*}
This equation implies
\[
  \big(a_{11}-a_{10}a_{00}^{-1}a_{01}\big) \pi_1 v = \pi_1 w -a_{10}a_{00}^{-1} \pi_0 w.
\]
Next, from \eqref{eq:sys1}, we obtain $\pi_0 w = a_{00} \pi_0 v+ a_{01}\pi_1 v$. Hence,
\begin{align*}
   \big(a_{11}-a_{10}a_{00}^{-1}a_{01}\big) \pi_1 v  &= \pi_1 w -a_{10}a_{00}^{-1} \big(a_{00} \pi_0 v+ a_{01}\pi_1 v\big)
 \\ & = \pi_1 w -a_{10}\pi_0 v-a_{10}a_{00}^{-1} a_{01}\pi_1 v.
\end{align*}
Thus,
\[
  \pi_1 w = a_{11}\pi_1v+a_{10}\pi_0 v.
\]
This equation together with \eqref{eq:sys1} implies (i).
\end{proof}

\begin{lemma}\label{lem:mat} Let $a\in L(H_1)$ and $b\in \mathcal{M}(\alpha,\beta,(A_0,A_1))$. Then the following conditions are equivalent
\begin{enumerate}
 \item[(i)] $a=b$;
 \item[(ii)] $ b^{-1}a \pi_0 = \pi_0\text{ and } ab^{-1}\pi_1 = \pi_1.$
\end{enumerate}
\end{lemma}
\begin{proof}
The implication (i)$\Rightarrow$(ii) is evidently true. Thus, we assume (ii) to hold. We aim for showing $a=b$. For this, we note that $b^{-1}a\pi_0 =\pi_0$ implies $a\pi_0 =b\pi_0$. Thus, using the block matrix representation from Section~\ref{sec:block}, we infer 
\[
   \begin{pmatrix} a_{00} & a_{01} \\ a_{10} & a_{11} \end{pmatrix}  \begin{pmatrix} 1 & 0\\ 0 & 0 \end{pmatrix} =    \begin{pmatrix} b_{00} & b_{01} \\ b_{10} & b_{11} \end{pmatrix}  \begin{pmatrix} 1 & 0\\ 0 & 0 \end{pmatrix},
\]
which implies 
\begin{equation}\label{eq:pi0a}
    a_{00} = b_{00} \text{ and } a_{10} = b_{10}.
\end{equation}
Next, from $ab^{-1}\pi_1 =\pi_1$, we obtain
\begin{align*}
   \begin{pmatrix} 0 & 0 \\ 0 & 1 \end{pmatrix} & =  \begin{pmatrix} a_{00} & a_{01} \\ a_{10} & a_{11} \end{pmatrix}
   \begin{pmatrix} (b^{-1})_{00} & (b^{-1})_{01} \\ (b^{-1})_{10} & (b^{-1})_{11} \end{pmatrix} \begin{pmatrix} 0 & 0 \\ 0 & 1 \end{pmatrix} 
   \\ & =  \begin{pmatrix} a_{00} & a_{01} \\ a_{10} & a_{11} \end{pmatrix}
   \begin{pmatrix} 0 & (b^{-1})_{01} \\ 0 & (b^{-1})_{11} \end{pmatrix} 
   \\ & = \begin{pmatrix} 0 & a_{00}(b^{-1})_{01}+ a_{01}(b^{-1})_{11} \\ 0 &a_{10}(b^{-1})_{01}+ a_{11}(b^{-1})_{11} \end{pmatrix} 
\end{align*}
Thus, using \eqref{eq:pi0a}, we infer
\begin{align*}
   - b_{00} (b^{-1})_{01} &= a_{01}(b^{-1})_{11} \\
     1 - b_{10} (b^{-1})_{01} &= a_{11}(b^{-1})_{11}.
\end{align*}
Multiplying both equations by $(b^{-1})_{11}^{-1}$ from the right and using the expressions stated in Lemma \ref{lem:boinv}(b), we obtain
\begin{equation}\label{eq:pi1a}
   a_{01} = - b_{00}(b^{-1})_{01}(b^{-1})_{11}^{-1} = - b_{00}\big(-b_{00}^{-1}b_{01}(b^{-1})_{11}\big)(b^{-1})_{11}^{-1}= b_{01}
\end{equation}
and, similarly,
\begin{equation}\label{eq:pi1b}
  a_{11} = (b^{-1})_{11}^{-1} - b_{10} (b^{-1})_{01}(b^{-1})_{11}^{-1} = b_{11}-b_{10}b_{00}^{-1}b_{01} + b_{10}b_{00}^{-1}b_{01} = b_{11}.
\end{equation}
Thus, the equations \eqref{eq:pi0a} together with \eqref{eq:pi1a} and \eqref{eq:pi1b} imply $a=b$ and, hence, the assertion.
\end{proof}
We like to point out that in the implication `(ii)$\Rightarrow$(i)' of Lemma \ref{lem:mat}, the invertibility of $a$ is implied rather than assumed.

We may now present the proof of Theorem \ref{thm:chH3}. We note that the implication `(i)$\Rightarrow$(ii)' should be seen as an abstract implementation of Tartar's method of oscillating test functions.

\begin{proof}[Proof of Theorem \ref{thm:chH3}] We shall assume that $a\in \mathcal{M}(\alpha,\beta,(A_0,A_1))$ and that $(a_n)_n$ nonlocally $H$-converges to $a$ and let $(q_n)_n$ and $q$ be as in (ii). By Theorem~\ref{thm:chH}, we shall assume without loss of generality that $\kappa(n)=n$ since any subsequence of $(a_n)_n$ also nonlocally $H$-converges to $a$. By Corollary \ref{cor:chH1}, $(a_n^*)_n$ nonlocally $H$-converges to $a^*$. Let $v\in \dom(\mathcal{A}_0)$ and define $v_n$ to be the solution of 
\[
    \langle a_n^*A_0 v_n, A_0 \phi\rangle = f(\phi) \quad(\phi\in \dom(\mathcal{A}_0)),
\]
where $f\in \dom(\mathcal{A}_0)^*$ is given by
\[
    f(\phi) = \langle a^*\mathcal{A}_0 v,\mathcal{A}_0\phi\rangle\quad (\phi\in \dom(\mathcal{A}_0)).
\]
Since $(a_n^*)_n$ nonlocally $H$-converges to $a^*$, we obtain that $(v_n)_n$ weakly converges to some $w\in\dom(\mathcal{A}_0)$ satisfying
\[
    \langle a^*A_0 w, A_0 \phi\rangle = \langle a^*\mathcal{A}_0 v,\mathcal{A}_0\phi\rangle\quad (\phi\in \dom(\mathcal{A}_0)),
\]
which, by Theorem \ref{thm:varwp}, leads to
\[
   w = \mathcal{A}_0^{-1}(a^*)_{00}^{-1}(\mathcal{A}_0^\diamond)^{-1}f=\mathcal{A}_0^{-1}(a^*)_{00}^{-1}(\mathcal{A}_0^\diamond)^{-1}(\mathcal{A}_0^\diamond)(a^*)_{00}\mathcal{A}_0v = v.
\]
Moreover, by nonlocal $H$-convergence, we deduce $a_n^*A_0 v_n \rightharpoonup a^*A_0v$ in $H_1$ as $n\to\infty$. We note, in particular, that $\mathcal{A}_{0,\textnormal{k}}^\diamond (a_n^*A_0 v_n)=f$ and $(\mathcal{A}_1^*)_{\textnormal{k}}^\diamond A_0 v_n =0$, by the complex property. For the latter note that $\kar((\mathcal{A}_1^*)_{\textnormal{k}}^\diamond)=\kar(A_1)$. Without loss of generality, we may assume that $(a_nq_n)_n$ weakly converges to some $r\in H_1$. For $n\in\mathbb{N}$ we have
\begin{equation}\label{eq:dcln}
   \langle a_n q_n, \mathcal{A}_0 v_n\rangle =    \langle q_n, a_n^*\mathcal{A}_0 v_n\rangle.
\end{equation}
Using Theorem \ref{thm:DCL} together with the assumptions (a) and (b) imposed on $q$, we infer from equation \eqref{eq:dcln} by letting $n\to\infty$
\[
   \langle r, \mathcal{A}_0 v \rangle = \langle q, a^*\mathcal{A}_0 v\rangle = \langle aq,\mathcal{A}_0 v\rangle.
\]
Since $v\in \dom(\mathcal{A}_0)$ can be chosen arbitrarily, we obtain
\begin{equation}\label{eq:pi0}
    \pi_0 r = \pi_0 aq,
\end{equation}
where $\pi_0$ is the orthogonal projection on $\rge(A_0)$.

Next, let $s\in \dom(\mathcal{A}_1^*)$. Let $(s_n)_n$ be the sequence in $\dom(\mathcal{A}_1)$ satisfying
\[
   \langle (a_n^{-1})^* A_1^* s_n, A_1^* \psi \rangle =    \langle (a^{-1})^* A_1^* s, A_1^* \psi \rangle
\]
By the nonlocal $H$-convergence of $(a_n^*)_n$ to $a^*$ it follows (invoking Theorem~\ref{thm:varwp} again) that
\[
   s_n \rightharpoonup s\in \dom(\mathcal{A}_0), \text{ and }(a_n^{-1})^* A_1^* s_n\rightharpoonup (a^{-1})^*A_1^*s.
\]
Moreover, we have that
\[
 (  \mathcal{A}_1^*)_{\textnormal{k}}^\diamond (a_n^{-1})^* A_1^* s_n =  (  \mathcal{A}_1^*)_{\textnormal{k}}^\diamond(a^{-1})^* A_1^* s
\]
as well as 
\[
  \mathcal{A}_{0,\textnormal{k}}^\diamond A_1^* s_n = 0.
\]
Next, for $n\in\mathbb{N}$, we have
\[
   \langle q_n, A_1^*s_n\rangle = \langle a_nq_n, (a_n^{-1})^*A_1^*s_n\rangle.
\]
By Theorem~\ref{thm:DCL} together with the assumptions on $q_n$, we may let $n\to\infty$ and obtain
\[
      \langle q, A_1^*s\rangle = \langle r, (a^{-1})^*A_1^*s\rangle.
\]
As $s\in \dom(\mathcal{A}_1^*)$ was arbitrary, this yields 
\begin{equation}\label{eq:pi1}
   \pi_1 q = \pi_1 a^{-1}r,
\end{equation}
where $\pi_1$ is the orthogonal projection onto $\rge(A_1^*)$.
Applying Lemma~\ref{lem:sys} to $w=r$ and $v=q$, we obtain $aq=r$.

We shall now assume that (ii) holds. By Theorem~\ref{thm:Hcom}, we may choose a $\kappa\colon\mathbb{N}\to\mathbb{N}$ strictly monotone such that of $(a_{\kappa(n)})_n$ nonlocally $H$-converges to some $b\in \mathcal{M}(\alpha,\beta,(A_0,A_1))$. Next, let $f\in \dom(\mathcal{A}_0)^*$ and $g\in \dom(\mathcal{A}_1^*)^*$ and let $(u_n)_n$ as well as $(v_n)_n$ satisfy 
\[
   \langle a_{\kappa(n)} A_0 u_n, A_0 \phi\rangle = f(\phi),\quad    \langle a_{\kappa(n)}^{-1} A_1^* v_n, A_1^* \psi\rangle = g(\psi),
\]for all $\phi\in\dom(\mathcal{A}_0)$ and $\psi\in \dom(\mathcal{A}_1^*)$. By nonlocal $H$-convergence, we obtain
\begin{align*}
   u_n\rightharpoonup u\in \dom(\mathcal{A}_0), &\ a_{\kappa(n)} A_0 u_n \rightharpoonup b A_0 u\\
   v_n\rightharpoonup v\in \dom(\mathcal{A}_1^*), &\ a_{\kappa(n)}^{-1} A_0 v_n \rightharpoonup b^{-1} A_0 v,
\end{align*}
where $u$ and $v$ satisfy
\[
   \langle b A_0 u, A_0 \phi\rangle =  f(\phi),\quad    \langle b^{-1} A_1^* v, A_1^* \psi\rangle =  g(\psi),
\]for all $\phi\in\dom(\mathcal{A}_0)$ and $\psi\in \dom(\mathcal{A}_1^*)$. We observe that 
\begin{align*}
 &  \mathcal{A}_{0,\textnormal{k}}^\diamond (a_{\kappa(n)} A_0 u_n)=f, \\
 & (\mathcal{A}_1^*)_{\textnormal{k}}^\diamond A_0 u_n = 0, \\
 &  \mathcal{A}_{0,\textnormal{k}}^\diamond A_1^* v_n  = 0,\\
 &  (\mathcal{A}_1^*)_{\textnormal{k}}^\diamond a_{\kappa(n)}^{-1} A_1^* v_n =g.
\end{align*}
Hence, by the assumption applied to $q_n = A_0 u_n$ or $q_n = a_{\kappa(n)}^{-1} A_1^* v_n$, we obtain
\[
   a_{\kappa(n)} A_0 u_n \rightharpoonup a A_0 u,\text{ and } a_{\kappa(n)}a_{\kappa(n)}^{-1} A_1^* v_n\rightharpoonup a b^{-1}A_1^*v.
\]
Thus, $b^{-1}a A_0 u = A_0 u$ and $A_1^*v = ab^{-1} A_1^* v$. As $f$ and $g$ are arbitrary, as in the proof of `(i)$\Rightarrow$(ii)' we infer that $u\in \dom(\mathcal{A}_0)$ and $v\in \dom(\mathcal{A}_1^*)$ are arbitrary, as well. Hence,
\[
   b^{-1}a \pi_0 = \pi_0\text{ and } ab^{-1}\pi_1 = \pi_1.
\]By Lemma~\ref{lem:mat}, we obtain $a=b$. The subsequence principle concludes the proof.
\end{proof}

\section{An application to Maxwell's equations}\label{sec:max}

In this section, we shall consider a homogenisation problem for Maxwell's equations. In contrast to many other discussions of homogenisation problems for the Maxwell system, we shall treat the full 3-dimensional time-dependent problem. Moreover, the setting is arranged in a way that we may allow for the homogenisation of highly oscillatory mixed type problems, where several regions of the underlying material are considered to have no dielectricity at all. That is to say, at certain regions of the underlying domain, one may or may not use the eddy current approximation. This goes well beyond the available results in the literature. 

Equations having highly oscillatory change of type have also been analysed in \cite{W16_SH,FW17_1D,CW17_1D}. In these references, however, the attention is restricted to $1+1$-dimensional model examples.

For other treatments of the homogenisation of the full time-dependent 3D-Maxwell's equations we refer to \cite{Wellander2001} and \cite{Barbatis2003}. In these references, the coefficients are assumed to be periodic. We shall furthermore refer to \cite{Sjoeberg2005,CW17_FH}, where the periodicity of the problem is exploited with the help of the Floquet--Bloch or Gelfand transformation. In particular, we refer to the seminal work \cite{Suslina2008} and the references therein.

In an open set $\Omega\subseteq \mathbb{R}^3$, Maxwell's equations are formulated as follows. Find $E,H\colon \mathbb{R}\times\Omega \to \mathbb{R}^3$ for a given $J\colon \mathbb{R}\times \Omega\to \mathbb{R}^3$ such that
\begin{align*}
   & \partial_t \epsilon E   + \sigma E - \curl H = J\\
   & \partial_t \mu B +\interior{\curl} E = 0,
\end{align*}
where for simplicity, we assume zero initial conditions. Moreover, $\epsilon,\mu,\sigma\in L(L^2(\Omega)^3)$ (dielectricity, permeability, conductivity) are given bounded linear operators with $\epsilon,\mu$ being selfadjoint.

In the Hilbert space framework, we shall apply next, we will favour the following block-operator-matrix form
\begin{equation}\label{eq:max}
   \Big(\partial_t \begin{pmatrix}
     \epsilon & 0 \\ 0 & \mu
   \end{pmatrix} + \begin{pmatrix}
     \sigma & 0 \\ 0 & 0
   \end{pmatrix} + \begin{pmatrix}
     0 & -\curl \\ \interior{\curl} & 0
   \end{pmatrix}\Big)\begin{pmatrix} E \\ H \end{pmatrix}= \begin{pmatrix} J \\ 0 \end{pmatrix}.
\end{equation}
Before turning to a homogenisation result for Maxwell's equations (see in particular Example~\ref{ex:max3} below), we shall shortly recall the well-posedness result, which will be used in the following. 
For a Hilbert space $H$ and $\nu>0$ we define 
\[
L_\nu^2(\mathbb{R};H)\coloneqq \{ f\in L_{\textnormal{loc}}^2(\mathbb{R};H); \int_{\mathbb{R}} \| f(t)\|_H^2 \exp(-2\nu t) dt<\infty\}.
\]
We recall from \cite[Corollary 2.5]{KPSTW14_OD} that the Fourier--Laplace transformation
\[
   \mathcal{L}_\nu \phi (\xi)\coloneqq \frac{1}{\sqrt{2\pi}} \int_\mathbb{R} \phi(t)\exp(-it\xi-\nu t)dt\quad (\phi\in C_c(\mathbb{R};H))
\]
can be extended unitarily as an operator from $L_\nu^2(\mathbb{R};H)$ onto $L^2(\mathbb{R};H)$. Moreover, we have that the weak derivative $\partial_t$ realised as an operator with maximal domain in $L_\nu^2(\mathbb{R};H)$ enjoys the spectral representation
\[
  \partial_t = \mathcal{L}_\nu^* (im+\nu)\mathcal{L}_\nu,
\]
where $im+\nu$ is the multiplication operator of multiplying by $x\mapsto ix+\nu$ with maximal domain. We denote for $\mu\geq 0$
\[
   \mathcal{H}^\infty(\mathbb{C}_{\Re>\mu}; L(H))\coloneqq \{ M\colon \mathbb{C}_{\Re>\mu}\to L(H); M \text{ analytic and bounded}\}. 
\] For the well-posedness of Maxwell's equations we shall employ the following theorem.

\begin{theorem}[{{\cite[Solution Theory]{PicPhy}}}]\label{thm:st} Let $c>0$, $\mu\geq 0$, $M \in \mathcal{H}^\infty( \mathbb{C}_{\Re>\mu};L(H))$, $\nu>\mu$. Assume that
\[
    \Re \lambda M(\lambda) \geq c\quad(\lambda\in \mathbb{C}_{\Re>\mu}).
\]
Let $A$ be a skew-self-adjoint operator in $H$.
Then the operator
\[
  \mathcal{B}\coloneqq \overline{\partial_t M(\partial_t)+A}\coloneqq \mathcal{L}_\nu^* \big((im+\nu)M(im+\nu)+A\big)\mathcal{L}_\nu,
\]
where $(im+\nu)M(im+\nu)+A\big)$ is the (abstract) multiplication operator of multiplying by $x\mapsto (ix+\nu)M(ix+\nu)+A$, is continuously invertible in $L_\nu^2(\mathbb{R};H)$; we have $\|\mathcal{B}^{-1}\|\leq 1/c$.
\end{theorem}

Recall from the spectral representation for $\partial_t$ that for $\nu>0$, the operator $\partial_t$ is continuously invertible.  

\begin{remark}
Theorem \ref{thm:st} applies to \eqref{eq:max} with the setting
\[
    H=L^2(\Omega)^3\oplus L^2(\Omega)^3,\ M(\partial_t) =   \begin{pmatrix}
     \epsilon & 0 \\ 0 & \mu
   \end{pmatrix} + \partial_t^{-1}\begin{pmatrix}
     \sigma & 0 \\ 0 & 0
   \end{pmatrix},\ A= \begin{pmatrix}
     0 & -\curl \\ \interior{\curl} & 0
   \end{pmatrix}.
\]
Note that the positive definiteness requirement translates into
\[
    \Re \lambda M(\lambda)  = \Re \lambda \left(  \begin{pmatrix}
     \epsilon & 0 \\ 0 & \mu
   \end{pmatrix} + \frac1\lambda\begin{pmatrix}
     \sigma & 0 \\ 0 & 0
   \end{pmatrix}\right) =  \Re \left(\lambda   \begin{pmatrix}
     \epsilon & 0 \\ 0 & \mu
   \end{pmatrix} + \begin{pmatrix}
     \sigma & 0 \\ 0 & 0
   \end{pmatrix}\right)\geq c,
\] for some $c>0$ and all $\lambda$ with $\Re\lambda$ large enough. Thus,
\[
 \Re \lambda M(\lambda) \geq c \iff \big(\lambda \epsilon+\Re \sigma\geq c\big)\, \land\, \big(\mu\geq c\big).
\]
\end{remark}

For the main homogenisation theorem we shall apply to Maxwell's equations, we will need the following construction principle for complexes.

\begin{proposition}\label{prop:blockcom} Let $B_0\colon \dom(B_0)\subseteq K_0\to K_1$, $B_1\colon \dom(B_1)\subseteq K_1\to K_2$, $B_2\colon \dom(B_2)\subseteq K_2\to K_3$ be densely defined and closed linear operators acting in the Hilbert spaces $K_0$, $K_1$, $K_2$, and $K_3$. Assume that $(B_0,B_1)$, $(B_1,B_2)$ are compact and exact. Define 
$(A_0,A_1)\coloneqq\left(\left(\begin{smallmatrix}0 & B_2^* \\ B_0 & 0\end{smallmatrix}\right),\left(\begin{smallmatrix}0 & B_1 \\ -B_1^* & 0\end{smallmatrix}\right)\right)$ with $\dom(A_0)=\dom(B_0)\oplus \dom(B_2^*)$ and $\dom(A_1)=\dom(B_1^*)\oplus \dom(B_1)$ with $H_0= K_0\oplus K_3$, $H_1=H_2=K_2\oplus K_1$.

Then $(A_0,A_1)$ is compact and exact. 
\end{proposition}
\begin{proof} We frequently use Proposition \ref{prop:comelm} in the following.
Since $(B_1,B_2)$ is compact, $(B_1,B_2)$ is closed. As $(B_1,B_2)$ is also exact, we obtain that $(B_2^*,B_1^*)$ is exact, as well. Hence,
\[
   \rge(A_0)=\rge(B_2^*)\oplus \rge(B_0) = \kar(B_1^*)\oplus \kar(B_1) = \kar(A_1),
\]
which shows that $(A_0,A_1)$ is exact. We are left with showing that $(A_0,A_1)$ is compact. For this, we realise that
$
    \dom(A_0^*)= \dom(B_2)\oplus \dom(B_0^*).
$
Hence,
\[
   \dom(A_0^*)\cap \dom(A_1) = \big(\dom(B_2)\cap \dom(B_1^*)\big)\oplus \big(\dom(B_0^*)\cap \dom(B_1)\big)
\]
Since $(B_1,B_2)$ is compact, so is $(B_2^*,B_1^*)$. Hence, $\big(\dom(B_2)\cap \dom(B_1^*)\big)\hookrightarrow K_2$ compactly. The compactness of $(B_0,B_1)$, thus, implies that $ \dom(A_0^*)\cap \dom(A_1)\hookrightarrow H_1=K_2\oplus K_1$ compactly, that is, $(A_0,A_1)$ is compact.
\end{proof}

\begin{example}\label{ex:max0} Let $\Omega\subseteq \mathbb{R}^3$ be an open bounded simply connected weak Lipschitz domain with connected complement. Then the typical situation for Maxwell's equations for applying Proposition \ref{prop:blockcom} is as follows: $B_0 = \grad$, $B_1=\curl$, and $B_2=\dive$ with $\dom(\grad)=H^1(\Omega)$, $\dom(\curl)=H(\curl,\Omega)$ and $\dom(\dive)=H(\dive,\Omega)$, $K_0 = L^2(\Omega)$, $K_1=L^2(\Omega)^3$, $K_2=L^2(\Omega)^3$, and $K_3=L^2(\Omega)$. The assumptions on $\Omega$ render $(\grad,\curl)$ and $(\curl,\dive)$ exact and compact. Indeed, the compactness of the complexes follows from Example \ref{ex:diffo} (a2) and (b2). Thus, Proposition~\ref{prop:comelm}(e) implies closedness of the complexes. Next, by Example \ref{ex:dfnf}, $(\grad,\curl)$ is exact as ${\Omega}$ is simply connected. Moreover, $(\curl,\dive)$ is exact, if and only if $(\interior{\grad},\interior{\curl})$ is exact, by Proposition \ref{prop:comelm} and the closedness of $(\curl,\dive)$. By Example~\ref{ex:dfnf}, $(\curl,\dive)$ is, thus, exact since $\mathbb{R}^3\setminus\Omega$ is connected.
\end{example}

\begin{theorem}[Homogenisation theorem]\label{thm:hommax} Let $B_0\colon \dom(B_0)\subseteq K_0\to K_1$, $B_1\colon \dom(B_1)\subseteq K_1\to K_2$, $B_2\colon \dom(B_2)\subseteq K_2\to K_3$ be densely defined and closed linear operators acting in the Hilbert spaces $K_0$, $K_1$, $K_2$, and $K_3$. Assume that $(B_0,B_1)$, $(B_1,B_2)$ are compact and exact. Define $H\coloneqq K_2\oplus K_1$. Let $(M_n)_n$ in $\mathcal{H}^\infty(\mathbb{C}_{\Re>\mu};L(H))$ be bounded and let $M\in \mathcal{H}^\infty(\mathbb{C}_{\Re>\mu};L(H))$ for some $\mu\geq 0$.. Assume
\[
    \Re \lambda M_n(\lambda)\geq c \quad (\lambda \in \mathbb{C}_{\Re>\mu})
\]
as well as for all $\lambda \in \mathbb{R}_{>\mu}$
\[
    M_n (\lambda)\to M(\lambda)
\]
$H$-nonlocally with respect to $\left(\left(\begin{smallmatrix}0 & B_2^* \\ B_0 & 0\end{smallmatrix}\right),\left(\begin{smallmatrix}0 & B_1 \\ -B_1^* & 0\end{smallmatrix}\right)\right)$ as $n\to\infty$.

Then
\[
    \overline{\left(\partial_t M_n(\partial_t)+\left(\begin{smallmatrix}0 & B_1 \\ -B_1^* & 0\end{smallmatrix}\right)\right)}^{-1} \to     \overline{\left(\partial_t M(\partial_t)+\left(\begin{smallmatrix}0 & B_1 \\ -B_1^* & 0\end{smallmatrix}\right)\right)}^{-1}
\]
in the weak operator topology of $L(L_\nu^2(\mathbb{R};H))$ for all $\nu>\mu$.
\end{theorem}
\begin{proof}
 $(A_0,A_1)\coloneqq\left(\left(\begin{smallmatrix}0 & B_2^* \\ B_0 & 0\end{smallmatrix}\right),\left(\begin{smallmatrix}0 & B_1 \\ -B_1^* & 0\end{smallmatrix}\right)\right)$ is compact and exact, by Proposition \ref{prop:blockcom}.  
 Let $\lambda \in \mathbb{R}_{>\mu}$. We write $M_{n,ij}(\lambda) \in L(\rge(A_j),\rge(A_i))$ according to the decomposition induced by $\rge(A_0)\oplus \rge(A_1)$ for all $i,j\in\{0,1\}$. Let $F\in H$. We define
 \[
     U_n\coloneqq \left(\lambda M_n(\lambda) + \left(\begin{smallmatrix}0 & B_1 \\ -B_1^* & 0\end{smallmatrix}\right)\right)^{-1}F.
 \]
 Writing $F_j, U_{j,n}$ for the components in $\rge(A_j)$ for $j\in\{0,1\}$, we obtain the following equivalent formulation for the  equation defining $U_n$:
 \begin{equation}\label{eq:maxhom}
    \lambda \begin{pmatrix} M_{n,00}(\lambda) & M_{n,01}(\lambda) \\ M_{n,10}(\lambda)   & M_{n,11}(\lambda) 
     \end{pmatrix}\begin{pmatrix}  U_{0,n}\\ U_{1,n} \end{pmatrix} + \begin{pmatrix} \mathcal{B} & 0 \\ 0 & 0 \end{pmatrix}\begin{pmatrix} U_{0,n} \\ U_{1,n} \end{pmatrix} = \begin{pmatrix} F_0 \\ F_1 \end{pmatrix},
 \end{equation}
 where $\mathcal{B}$ denotes the operator acting as $\left(\begin{smallmatrix}0 & B_1 \\ -B_1^* & 0\end{smallmatrix}\right)$ which is domain-wise restricted to the orthogonal complement of the null space of $\left(\begin{smallmatrix}0 & B_1 \\ -B_1^* & 0\end{smallmatrix}\right)$ and co-domain-wise restricted to the range of $\left(\begin{smallmatrix}0 & B_1 \\ -B_1^* & 0\end{smallmatrix}\right)$. A straightforward computation shows that equation \eqref{eq:maxhom} equivalently reads as
 \begin{multline*}
      \begin{pmatrix} \lambda M_{n,00}(\lambda)-\lambda M_{n,01}(\lambda)M_{n,11}(\lambda)^{-1}M_{n,10}(\lambda) & 0 \\  M_{n,11}(\lambda)^{-1}M_{n,10}(\lambda)   & 1 
     \end{pmatrix}\begin{pmatrix}  U_{0,n}\\ U_{1,n} \end{pmatrix} + \begin{pmatrix} \mathcal{B} & 0 \\ 0 & 0 \end{pmatrix}\begin{pmatrix} U_{0,n} \\ U_{1,n} \end{pmatrix}\\ = \begin{pmatrix} F_0 - M_{n,01}(\lambda)M_{n,11}(\lambda)^{-1}F_1 \\ \tfrac{1}{\lambda} M_{n,11}(\lambda)^{-1}F_1 \end{pmatrix}
 \end{multline*}
 or 
  \begin{multline}\label{eq:maxhom2}
 \begin{pmatrix}  U_{0,n}\\ U_{1,n} \end{pmatrix} =\\    \begin{pmatrix} \left(\lambda\left(M_{n,00}(\lambda)- M_{n,01}(\lambda)M_{n,11}(\lambda)^{-1}M_{n,10}(\lambda)\right) + \mathcal{B}\right)^{-1} \left(F_0  - M_{n,01}(\lambda)M_{n,11}(\lambda)^{-1}F_1 \right) \\  - M_{n,11}(\lambda)^{-1}M_{n,10}(\lambda) U_{0,n} + \tfrac{1}{\lambda} M_{n,11}(\lambda)^{-1}F_1\end{pmatrix}. \end{multline}
 By Theorem \ref{thm:chH}, we have that 
 \begin{align*}
    \lambda\left(M_{n,00}(\lambda)- M_{n,01}(\lambda)M_{n,11}(\lambda)^{-1}M_{n,10}(\lambda)\right) &\to 
   \lambda\left(M_{00}(\lambda)- M_{01}(\lambda)M_{11}(\lambda)^{-1}M_{10}(\lambda)\right)\\
    M_{n,01}(\lambda)M_{n,11}(\lambda)^{-1}& \to M_{01}(\lambda)M_{11}(\lambda)^{-1}\\
    M_{n,11}(\lambda)^{-1}M_{n,10}(\lambda) &\to M_{11}(\lambda)^{-1}M_{10}(\lambda) \\
    M_{n,11}(\lambda)^{-1} &\to M_{11}(\lambda)^{-1}
   \end{align*}
   as $n\to \infty$ with convergence in the respective weak operator topologies.
   We note that, by the identity theorem the convergence of the just mentioned operator sequences does actually hold for all $\lambda\in \mathbb{C}$ provided $\Re\lambda$ is large enough.
   
   Next, by the compactness of the complex $(A_0,A_1)$, the operator $\mathcal{B}$ has compact resolvent. By Lemma \ref{lem:oc} below applied to $B=\mathcal{B}$, $T_n =\lambda\left(M_{n,00}(\lambda)- M_{n,01}(\lambda)M_{n,11}(\lambda)^{-1}M_{n,10}(\lambda)\right)$ and $\phi_n = \left(F_0  - M_{n,01}(\lambda)M_{n,11}(\lambda)^{-1}F_1 \right)$, we deduce that $(U_{0,n})_n$ converges in norm to some $U_0$. Hence, $(U_{1,n})_n$ weakly converges to some $U_1$. Letting $n\to \infty$ in \eqref{eq:maxhom2} thus leads to
   \[
 \begin{pmatrix}  U_{0}\\ U_{1} \end{pmatrix} =\\    \begin{pmatrix} \left(\lambda\left(M_{00}(\lambda)- M_{01}(\lambda)M_{11}(\lambda)^{-1}M_{10}(\lambda)\right) + \mathcal{B}\right)^{-1} \left(F_0  - M_{01}(\lambda)M_{11}(\lambda)^{-1}F_1 \right) \\  - M_{11}(\lambda)^{-1}M_{10}(\lambda) U_{0} + \tfrac{1}{\lambda} M_{11}(\lambda)^{-1}F_1\end{pmatrix}.
   \]Rearranging terms, we obtain with $U=(U_0,U_1)$
   \[
    \left(\lambda M (\lambda) + \left(\begin{smallmatrix}0 & B_1 \\ -B_1^* & 0\end{smallmatrix}\right)\right)U = F.
   \]
  This settles the proof.
\end{proof}

We complete the latter proof by stating and proving Lemma \ref{lem:oc}.

\begin{lemma}\label{lem:oc} Let $B\colon \dom(B)\subseteq H\to H$ be skew-self-adjoint in the Hilbert space $H$ and assume that $\dom(B)\hookrightarrow H$ is compact. Assume furthermore that $(T_n)_n$ is a sequence in $L(H)$ such that 
$\Re T_n\geq c$ for all $n\in \mathbb{N}$. If $T_n\to T$ in the weak operator topology for some $T\in L(H)$ as $n\to \infty$, then 
\[
   (    T_n+B)^{-1}\phi_n \to    (    T+B)^{-1}\phi
\] in $H$
for all $(\phi_n)_n$ weakly convergent to some $\phi\in H$.
\end{lemma}
\begin{proof}
 Let $\phi_n, \phi$ be as in the statement. We define
 \[
    u_n\coloneqq  (    T_n+B)^{-1}\phi_n.
 \]
We obtain that $(u_n)_n$ is bounded in $\dom(B)$; see also \cite[Lemma 2.12]{EGW17_D2N} for the precise argument. Possibly choosing a subsequence (not relabelled) of $(u_n)_n$, we may assume that $u_n\rightharpoonup u$ in $\dom(B)$ for some $u$. In particular, we obtain that $u_n\to u$ in $H$. Hence, in the equality $\phi_n = T_n u_n + Bu_n$, we let $n\to\infty$ and obtain
\[
   \phi = Tu + Bu.
\]
The continuous invertibility of $T+B$ identifies $u$ and thus the whole sequence converges weakly in $\dom(B)$ and strongly in $H$, which is the assertion.
\end{proof}

\begin{remark}
  A contradiction argument yields that the convergence implied in Lemma \ref{lem:oc} together with the compactness assumption is sufficient for operator norm convergence of $(T_n+B)^{-1} \to (T+B)^{-1}$ as $n\to\infty$.
\end{remark}

\begin{remark}
The proof of Theorem \ref{thm:hommax} is a variant of the rationale employed in the proof of \cite[Theorem 5.5]{W16_HPDE}. However, we note that the conditions in \cite{W16_HPDE} are more restrictive than the ones here. Indeed, in \cite{W16_HPDE} only a compactness statement for `$G$-convergence' was obtained. Moreover, in order to prove well-posedness of the limit equation, the class of sequences $M_n$ was more restrictive in the sense that a change of type was not permitted.
\end{remark}

Since we have discussed nonlocal $H$-convergence with respect to operator complexes like $(\grad,\curl)$ and $(\interior{\grad},\interior{\curl})$ only, it might be of interest to put the convergence assumed in Theorem~\ref{thm:hommax} into perspective of nonlocal $H$-convergence of simpler complexes. This is done in the following result and the subsequent example.

\begin{proposition}\label{prop:diagcase} Let the assumptions and definitions of Proposition~\ref{prop:blockcom} be in effect, $\alpha,\beta>0$. Let $(\epsilon_n)_n$ in $\mathcal{M}(\alpha,\beta,(B_0,B_1))$ and $(\mu_n)_n$ in $\mathcal{M}(\alpha,\beta,(B_2^*,B_1^*))$ be bounded sequences, $\epsilon \in L(K_2)$, $\mu \in L(K_1)$ be continuously invertible. Then for all $n\in\mathbb{N}$
\[
   \begin{pmatrix}
     \epsilon_n & 0 \\ 0 & \mu_n
   \end{pmatrix} \in \mathcal{M}(\alpha,\beta,(A_0,A_1)).
\]
Moreover, the following conditions are equivalent.
\begin{enumerate}
 \item[(i)] $(\epsilon_n)_n\to \epsilon$ $H$-nonlocally w.r.t.~$(B_0,B_1)$ and $(\mu_n)_n \to \mu$ $H$-nonlocally w.r.t.~$(B_2^*,B_1^*)$. 
 \item[(ii)] $\big( \left( \begin{smallmatrix}
     \epsilon_n & 0 \\ 0 & \mu_n
   \end{smallmatrix}\right) \big)_n\to \left(  \begin{smallmatrix}
     \epsilon & 0 \\ 0 & \mu
   \end{smallmatrix}\right) $ $H$-nonlocally w.r.t.~$(A_0,A_1)$.
\end{enumerate}
\end{proposition}

\begin{example}\label{ex:max1} With the setting given in Example \ref{ex:max0}, we obtain for operator sequences  $(\epsilon_n)_n$ and $(\mu_n)_n$ in $L(L^2(\Omega)^3)$ satisfying suitable positive definiteness constraints and for $\epsilon,\mu\in L(L^2(\Omega)^3)$ that 
$\big(   \left( \begin{smallmatrix}
     \epsilon_n & 0 \\ 0 & \mu_n
   \end{smallmatrix}\right) \big)_n\to \left(  \begin{smallmatrix}
     \epsilon & 0 \\ 0 & \mu
   \end{smallmatrix}\right) $ $H$-nonlocally w.r.t.~$(A_0,A_1)$ if and only if $(\epsilon_n)_n\to \epsilon$ $H$-nonlocally w.r.t.~$(\grad,\curl)$ and $(\mu_n)_n\to \mu$ $H$-nonlocally w.r.t.~$(\interior{\grad},\interior{\curl})$.
\end{example}

\begin{proof}[Proof of Proposition~\ref{prop:diagcase}] By Proposition~\ref{prop:blockcom} and \eqref{eq:od}, we have the decomposition  $H_1 = \rge(A_0)\oplus \rge(A_1^*)$. Moreover, we obtain from the exactness of $(B_0,B_1)$ and $(B_2^*,B_1^*)$ the decomposition
\[
   H_1 = K_2\oplus K_1 = \big(\rge(B_2^*)\oplus \rge(B_1)\big)\oplus \big(\rge(B_0)\oplus \rge(B_1^*)\big).
\]
Next, from $\rge(A_0)=\rge(B_2^*)\oplus \rge(B_0)$ and $\rge(A_1^*)=\rge(B_1)\oplus \rge(B_1^*)$, we deduce using the operators
\begin{align*}
  U & \coloneqq \begin{pmatrix} \iota_{\textnormal{r},A_0} & \iota_{\textnormal{r},A_1^*}\end{pmatrix} \\
  U_1 & \coloneqq \begin{pmatrix} \iota_{\textnormal{r},B_0} & \iota_{\textnormal{r},B_1^*}\end{pmatrix} \\
  U_0 &  \coloneqq \begin{pmatrix} \iota_{\textnormal{r},B_2^*} & \iota_{\textnormal{r},B_1}\end{pmatrix},
\end{align*}
(see also \eqref{eq:U}) that for any $\epsilon \in L(K_2)$ and $\mu\in L(K_1)$, we have
\[
  U^* \begin{pmatrix} \epsilon & 0 \\ 0 & \mu \end{pmatrix} U = \begin{pmatrix} \begin{pmatrix} \epsilon_{00} & 0 \\ 0 & \mu_{00}  \end{pmatrix} & \begin{pmatrix} \epsilon_{01} & 0 \\ 0 & \mu_{01} \end{pmatrix} \\ \begin{pmatrix} \epsilon_{10} & 0 \\ 0 & \mu_{10} \end{pmatrix} & \begin{pmatrix} \epsilon_{11} & 0 \\ 0 & \mu_{11} \end{pmatrix} \end{pmatrix},
\]
where
\begin{align*}
   U_0^* \epsilon U_0 &= \begin{pmatrix} \epsilon_{00} & \epsilon_{01} \\ \epsilon_{10} & \epsilon_{11} \end{pmatrix}\text{ and }\\
      U_1^* \mu U_1 &= \begin{pmatrix} \mu_{00} & \mu_{01} \\ \mu_{10} & \mu_{11} \end{pmatrix}.
\end{align*}
These representations of $\varepsilon$ and $\mu$ applied to $\epsilon_n$ and $\mu_n$ instead yield the first assertion of the present proposition; the equivalence then follows from Theorem~\ref{thm:chH} in a straightforward manner.
\end{proof}

\begin{remark}
\label{rem:hommax} The main result in \cite{Barbatis2003} is contained in Theorem \ref{thm:hommax}. In fact, it suffices to take the setting as outlined in Example \ref{ex:max0}; we also refer to Example~\ref{ex:max1}. We also recall that local $H$-convergence implies nonlocal $H$-convergence with respect to both $(\interior{\grad},\interior{\curl})$ and $({\grad},{\curl})$ (see also Theorem~\ref{thm:lHnlH} and Remark \ref{rem:neu}). Moreover, we note that the coefficients treated in \cite{Barbatis2003} are arranged in a way that their Fourier--Laplace transformed images locally $H$-converge. We also refer to the subsequent example.
\end{remark}

A more concrete example with change of type, that is, where the underlying problem is such that the Maxwell's equations rapidly oscillate between the parabolic eddy current problem and the hyperbolic full Maxwell's equations, is considered next. 

\begin{example}\label{ex:max3} Let $\Omega\subseteq \mathbb{R}^3$ be such that both $(\interior{\grad},\interior{\curl})$ and $(\grad,\curl)$ are compact and exact sequences; e.g.~$\Omega$ being an open, simply connected and bounded weak Lipschitz domain with connected $\mathbb{R}^3\setminus{\Omega}$. 

(a) Let $\epsilon, \mu, \sigma \in L^\infty(\mathbb{R}^3)^{3\times 3}$ be $[0,1)^3$-periodic with $\epsilon,\mu$ attaining values in the (possibly complex) self-adjoint matrices. Define $\eps_n(x)\coloneqq \eps(nx)$ for a.a.~$x\in \mathbb{R}^3$ and similarly for $\mu_n$, and $\sigma_n$. Assume that there exists $\eta>0$ such that for all $\lambda\in \mathbb{C}_{\Re>\eta}$ we have
\begin{equation}\label{eq:pdmax}
    \Re \left(\lambda \eps+\sigma\right), \mu\geq c\quad(n\in\mathbb{N})
\end{equation}
for some $c>0$ almost everywhere in the sense of positive definiteness. We emphasise that if $\eps(x)\geq c_1$ for a.a.\ $x\in \Omega_1$ and $\Re \sigma(x)\geq c_1$ for a.a.\ $x\in \Omega_2\coloneqq\mathbb{R}^3\setminus \Omega_1$ for some $c_1>0$ that $\epsilon$ on $\Omega_2$ and $\sigma$ on $\Omega_1$ are allowed to vanish, while the positive definiteness condition in \eqref{eq:pdmax} can still be warranted. This introduces a highly oscillatory change of type. Then by Theorem \ref{thm:hommax} (see also Remark~\ref{rem:hommax} and Example~\ref{ex:max1})
\begin{multline*}
   \overline{ \left(\partial_t \begin{pmatrix} \epsilon_n & 0 \\ 0 & \mu_n \end{pmatrix} + \begin{pmatrix} \sigma_n & 0 \\ 0 & 0 \end{pmatrix} + \begin{pmatrix} 0 & -\curl \\ \interior{\curl} & 0 \end{pmatrix}\right)}^{-1} \\ \to     \overline{\left(\partial_t \begin{pmatrix} \Big\langle \epsilon+(\cdot)\sigma\Big\rangle_{\textrm{hom}}(\partial_t) & 0 \\ 0 & \Big\langle \mu \Big\rangle_{\textrm{hom}} \end{pmatrix} + \begin{pmatrix} 0 & -\curl \\ \interior{\curl} & 0 \end{pmatrix}\right)}^{-1}
\end{multline*} in the weak operator topology of $L\big(L_\nu^2(\mathbb{R};L^2(\Omega)^6)\big)$ as $n\to\infty$,
where $\Big\langle \mu \Big\rangle_{\textrm{hom}}$ is the standard homogenised matrix associated with $\mu$ and 
\[
   \Big\langle \epsilon+(\cdot)\sigma\Big\rangle_{\textrm{hom}}(\partial_t) \coloneqq \Big(\lambda \mapsto \Big\langle \epsilon+\lambda^{-1}\sigma\Big\rangle_{\textrm{hom}}\Big)(\partial_t).
\]
This is a memory term that occurs during the homogenisation process. We note here that such an effect has been observed already in \cite[p. 144]{Jikov1994}, but also in \cite[Theorem 3.2]{Wellander2001}.

(b) Let $\epsilon,\sigma$ as in (a) and $\mu_n\coloneqq a_n+k_n*$ as in Example~\ref{ex:mcv3} (we shall also re-use the notation $a$ and $k*$ for the operators  mentioned in that example). Then the results in Example~\ref{ex:mcv3}, Example~\ref{ex:max1}, and Theorem~\ref{thm:hommax} yield
\begin{multline*}
   \overline{ \left(\partial_t \begin{pmatrix} \epsilon_n & 0 \\ 0 & a_n+k_n* \end{pmatrix} + \begin{pmatrix} \sigma_n & 0 \\ 0 & 0 \end{pmatrix} + \begin{pmatrix} 0 & -\curl \\ \interior{\curl} & 0 \end{pmatrix}\right)}^{-1} \\ \to     \overline{\left(\partial_t \begin{pmatrix} \Big\langle \epsilon+(\cdot)\sigma\Big\rangle_{\textrm{hom}}(\partial_t) & 0 \\ 0 & a+k* \end{pmatrix} + \begin{pmatrix} 0 & -\curl \\ \interior{\curl} & 0 \end{pmatrix}\right)}^{-1}
\end{multline*}
in the weak operator topology of $L\big(L_\nu^2(\mathbb{R};L^2(\Omega)^6)\big)$ as $n\to\infty$.
We emphasise that the convolution $k*$ is computed with respect to the \emph{spatial} variables. Thus, the limit model is both nonlocal in space and time.
\end{example}

\section{More examples}\label{sec:mex}

In this section, we shall provide two more applications. In fact, since our results has been developed for the abstract setting of 
closed complexes in Hilbert spaces and suitable operators as coefficients, this section may also be read as the versatility of the notion of complexes in the analysis of partial differential equations.

\subsection{Homogenisation problems for fourth order elliptic equations}

In this section, we shall revisit the homogenisation problem for thin plates (see e.g.~\cite{Pastukhova2016})
In that reference, the author studied operator norm error estimates for the homogenisation problem associated to the differential expression
\[
    \sum_{i,j,s,h\in \{1,2,3\}} \partial_i \partial_j a_{ijsh} \partial_s\partial_h,
\]
where the coefficients $a_{ijsh}$ are highly oscillatory. It is easy to see that the latter differential expression can be reformulated as 
\[
    \dive \Dive a \Grad \grad,
\]
where $\Grad$ is the Jacobian matrix and $\Dive$ the row-wise divergence and $a$ acts as a mapping from $2$-tensors to $2$-tensors. The variational formulation is then given by
\[
    \langle a \Grad \grad u , \Grad\grad \phi \rangle = f(\phi)
\]
for $\phi$ belonging to a suitable test-function space. If we assume $\Grad\grad$ to be endowed with full homogeneous boundary conditions (i.e.~the $L^2$-closure of $\Grad\grad$ restricted to test functions compactly supported in $\Omega$), it is possible to derive the second variational problem to be discussed for nonlocal homogenisation problems, which we will do in the following. In fact, it will turn out that the closed and exact complex involving $\Grad\grad$ bases on $(\interior{\grad},\interior{\curl})$. For more aspects of this (and an extension of this complex) we refer to the Pauly--Zulehner complex (\cite{Pauly2016}; see also \cite{Quenneville-Bair2015}). 

We introduce the following differential operators:

\begin{definition} Let $\Omega\subseteq \mathbb{R}^3$ be open and bounded. We define
\begin{align*}
 \interior{\He} & \colon H_0^2(\Omega) \subseteq L^2(\Omega)  \to L^2_{\sym}(\Omega)^{3\times 3}
 \\ \phi & \mapsto (\partial_{i}\partial_j \phi)_{i,j\in \{1,2,3\}}
 \\ \interior{\Curl}_{\sym} & \colon \dom(\interior{\curl})^3\cap L^2_{\sym}(\Omega)^{3\times 3} \subseteq L^2_{\sym}(\Omega)^{3\times 3}  \to L^2_{\trf}(\Omega)^{3\times 3}
 \\ (\phi_{ij})_{i,j\in\{1,2,3\}} &\mapsto \begin{pmatrix} \interior{\curl} (\phi_{1j})_{j\in\{1,2,3\}} \\ \interior{\curl} (\phi_{2j})_{j\in\{1,2,3\}} \\ \interior{\curl} (\phi_{3j})_{j\in\{1,2,3\}} \end{pmatrix}, 
 \end{align*}
 where $L^2_{\sym}(\Omega)^{3\times 3}$ is the set of symmetric $3$-by-$3$ matrices with entries from $L^2(\Omega)$ and $L^2_{\trf}(\Omega)^{3\times 3}$ is the set of $3$-by-$3$ matrices with vanishing matrix trace and entries from $L^2(\Omega)$. 
 \end{definition}
 
 For convenience of the reader, we show exactness of $(\interior{\He},\interior{\Curl}_{\sym})$ with a proof independent of \cite{Pauly2016}. Recall that $(\interior{\grad},\interior{\curl})$ is closed and exact, for instance, if $\Omega$ is an open, bounded weak Lipschitz domain with connected complement.
 
 \begin{theorem}\label{thm:gghe} Let $\Omega\subseteq \mathbb{R}^3$ open and bounded with $(\interior{\grad},\interior{\curl})$  exact and closed. Then $(\interior{\He},\interior{\Curl}_{\sym})$ is an exact and closed complex. 
 \end{theorem}
 
 \begin{lemma}\label{lem:heh2} Let $\Omega\subseteq \mathbb{R}^3$ open and bounded. Then the graph norm of $\interior{\He}$ and the $H^2$-norm are equivalent on $H^2_0(\Omega)$.  
 \end{lemma}
 \begin{proof}
   Let $\phi \in C_c^\infty(\Omega)$. Then we have for all $i,j\in\{1,2,3\}$
   \[
       \| \phi\|_{L^2(\Omega)}\leq c\| \partial_j \phi \|_{L^2(\Omega)}\leq c^2 \| \partial_i\partial_j \phi \|_{L^2(\Omega)}\leq c^2 \| \interior{\He}\, \phi \|_{L^2(\Omega)}\leq c^2\|\phi\|_{H^2(\Omega)}
   \]
  where the last and second last inequalities are trivial and the first and the second one follow from Poincar\'e's inequality for some suitable constant $c>0$. Thus, the graph norm of $\interior{\He}$ and the $H^2$-norm are equivalent on $C_c^\infty(\Omega)$. Since $H_0^2(\Omega)=\overline{C_c^\infty(\Omega)}^{H^2}$ we obtain the assertion.
 \end{proof}
 
 Note that Lemma \ref{lem:heh2} particularly implies that $\interior{\He}$ is closed and that $C_c^\infty(\Omega)$ is an operator core for $\interior{\He}$.
  
 \begin{proof}[Proof of Theorem \ref{thm:gghe}] Let us start with the complex property. So, let $\phi\in C_c^\infty(\Omega)$. Then
 \[
    \interior{\Curl}_{\sym}\interior{\He}\phi = \begin{pmatrix} \interior{\curl} (\interior{\He}\phi_{1j})_{j\in\{1,2,3\}} \\ \interior{\curl} (\interior{\He}\phi_{2j})_{j\in\{1,2,3\}} \\ \interior{\curl} (\interior{\He}\phi_{3j})_{j\in\{1,2,3\}} \end{pmatrix} =
    \begin{pmatrix} \interior{\curl} (\interior{\grad}\partial_1\phi) \\ \interior{\curl} (\interior{\grad}\partial_2\phi) \\ \interior{\curl} (\interior{\grad}\partial_3\phi) \end{pmatrix} = 0.
 \] By Lemma \ref{lem:heh2}, $C_c^\infty(\Omega)$ is an operator core for $\interior{\He}$. Thus, $\rge(\interior{\He})\subseteq \kar(\interior{\Curl}_{\sym})$.
 
 Next, we show closedness of the complex. For this we note that $H_0^2(\Omega)$ embeds compactly into $L^2(\Omega)$. Hence, by Lemma \ref{lem:heh2}, we deduce that $\rge(\interior{\He})$ is closed. 
 
 Since $\interior{\curl}$ has closed range -- by the closed graph theorem --  we find $c>0$ such that for all $\phi \in \dom(\interior{\curl})\cap \kar(\interior{\curl})^\bot$ we have
 \[
    \| \phi\|^2_{L^2}\leq c \|\interior{\curl}\phi\|_{L^2}^2.
 \] It is easy to see that $\kar(\interior{\Curl}_{\sym})=\kar(\interior{\curl})^3\cap L_{\sym}^2(\Omega)^{3\times 3}$. Hence, if $\Phi\in \dom(\interior{\Curl}_{\sym})\cap \kar( \interior{\Curl}_{\sym})^{\bot}$ we obtain
 \[
   \| \Phi\|_{ L_{\sym}^2}^2 =\sum_{i\in\{1,2,3\}} \| (\Phi_{ij} )_{j\in\{1,2,3\}}\|^2 \leq \sum_{i\in\{1,2,3\}} c\| \interior{\curl}\,(\Phi_{ij} )_{j\in\{1,2,3\}}\|^2 = c \|\interior{\Curl}_{\sym}\Phi\|^2_{L^2}.
 \]
 This yields closedness of the range of $\interior{\Curl}_{\sym}$.
 
 We are left with showing the exactness of the complex. More precisely, it remains to prove 
 \[
   \rge(\interior{\He})\supseteq \kar(\interior{\Curl}_{\sym}).
 \]
 For this let $\Phi\in \kar(\interior{\Curl}_{\sym})$. Then there exists $\phi_i \in \dom(\interior{\grad})$ such that $\interior{\grad}\, \phi_i = (\Phi_{ij})_{j\in\{1,2,3\}}$ for all $i\in\{1,2,3\}$ since $(\interior{\grad},\interior{\curl})$ is exact. Since $\Phi$ is symmetric, we deduce that $\partial_j\phi_i = \partial_i \phi_j$. Thus, $(\phi_{i})_{i\in\{1,2,3\}} \in \kar(\curl)\cap \dom(\interior{\grad})^3\subseteq \kar(\interior{\curl})$. Hence, by the exactness of $(\interior{\grad},\interior{\curl})$ we find $\psi\in \dom(\interior{\grad})$ such that $\interior{\grad}\psi = (\phi_{i})_{i\in\{1,2,3\}} $. It follows that $\psi\in \dom(\interior{\He})$. Moreover, we obtain
 \[
    \interior{\He}\, \psi =  \big((\interior{\grad}\, \partial_i \psi)^T\big)_{i\in \{1,2,3\}} = 
    \big((\interior{\grad}\, \phi_i)^T\big)_{i\in \{1,2,3\}} = \Phi.
 \]
 This shows the assertion.
 \end{proof}
 
 \begin{remark}
 It is an easy exercise to show that $\interior{\Curl}_{\sym}^*= \sym {\Curl}_{\trf}$, where $\sym M=\tfrac12(M+M^T)$ and $\Curl_{\trf}$ is the distributional row-wise curl operator with no boundary conditions acting on trace-free matrices. It is remarkable that the equations for the description of nonlocal $H$-convergence of $(a_n)_n$ to some $a$ with respect to the complex $(\interior{\He},\interior{\Curl}_{\sym})$ are of different order. Indeed, one equation is
 \[
      \langle a_n \Grad \grad u_n , \Grad\grad \phi \rangle =  f(\phi)
       \]
       for suitable test functions $\phi$ and given right-hand side $f$. This corresponds to the 4th order equation mentioned above. The second variational problem reads
\[
       \langle a_n^{-1} \sym {\Curl}_{\trf} v_n , \sym {\Curl}_{\trf} \psi \rangle = g(\psi),
\]
which leads to a 2nd order problem, only.
\end{remark}

\subsection{An $H$-compactness result for Riemannian manifolds} 

The general setting developed in the previous sections allows for $H$-compactness results also on manifolds. We shall refer to the fairly recent result in \cite{Hoppe2017}, where the corresponding local problem has been discussed. 

We refer to \cite{Weck1974} or \cite{Picard1984} for the precise functional analytic setting to be sketched below.

Let $\Lambda$ be a $d$-dimensional $C^\infty$-manifold and let $\Omega\subseteq \Lambda$ be an open submanifold of $\Lambda$. For any $q\in \{0,\ldots, n\}$ this induces $L_q^2(\Omega)$, the completion of the space of compactly supported $q$-forms on $\Omega$ endowed with the scalar product
\[
    \langle \omega,\eta \rangle = \int_{\Omega} \omega \land *\eta,
\]
where $*$ denotes the Hodge duality and $\land$ the alternating product. 

Next, using the thus defined scalar product, we let $\mathrm{d}$ be the (distributional) exterior derivative on $L_q^2(\Omega)$ with values in the $L_{q+1}^2(\Omega)$. The adjoint of this operator is set to be $\interior{\delta}$. Similarly, we let $\interior{\mathrm{d}}$ be the closure of $\mathrm{d}$ restricted to $C^1$-forms with compact support in $\Omega$; with adjoint $\delta$. In order to stress the dimension of the underlying spaces, we write $\mathrm{d}_{q\to q+1}$  (similarly for other operators). 

Assume that $\big(\interior{\mathrm{d}}_{q\to q+1}, \interior{\mathrm{d}}_{q+1\to q+2}\big)$ is exact and compact. Then the translation of our compactness theorem  to the present setting reads as follows.

\begin{theorem} Let $(a_n)_n$ be a sequence in $\mathcal{M}(\alpha,\beta,(\interior{\mathrm{d}}_{q\to q+1}, \interior{\mathrm{d}}_{q+1\to q+2}))$. 

Then there is a convergent subsequence of $(a_n)_n$, which $H$-nonlocally converges with respect to $(\interior{\mathrm{d}}_{q\to q+1}, \interior{\mathrm{d}}_{q+1\to q+2})$.

The nonlocal $H$-limit is unique.
\end{theorem} 

It might be worth explicitly writing down the definition of nonlocal $H$-convergence in this particular setting.

$(a_n)_n$ nonlocally $H$-converges to some $a$ with respect to $(\interior{\mathrm{d}}_{q\to q+1}, \interior{\mathrm{d}}_{q+1\to q+2})$, if and only if for all $f \in \dom(\interior{\mathrm{d}}_{q\to q+1}|_{\kar(\interior{\mathrm{d}}_{q\to q+1})^{\bot}})^*$ and $g\in \dom(\delta_{q+2\to q+1}|_{\kar(\delta_{q+2\to q+1})^{\bot}})^*$
 and corresponding $u_n$ and $v_n$ such that
\[
    \langle a_n \interior{\mathrm{d}}u_n , \interior{\mathrm{d}}\phi \rangle = f(\phi)\text{ and }    \langle a_n^{-1}
     \delta v_n , \delta \phi \rangle = g(\psi)
\]
for $\phi \in \dom(\interior{\mathrm{d}}_{q\to q+1}|_{\kar(\interior{\mathrm{d}}_{q\to q+1})^{\bot}})$ and $\psi\in \dom(\delta_{q+2\to q+1}|_{\kar(\delta_{q+2\to q})^{\bot}})$
implies $u_n \rightharpoonup u$ and $v_n\rightharpoonup v$ as well as $ a_n \mathrm{d} u_n \rightharpoonup a \mathrm{d}u$ and $a_n^{-1} \delta  v_n \rightharpoonup a^{-1}\delta v$, where $u$ and $v$ satisfy
\[
  \langle a \mathrm{d}u , \mathrm{d}\phi \rangle = f(\phi)\text{ and }    \langle a^{-1}
    \delta v , \delta \psi \rangle = g(\psi).
\]
Of course also the div-curl type characterisation in Theorem~\ref{thm:chH3} carries over to the special case discussed here.

\section{Concluding Remarks}

The present article discussed a generalisation of the well-known concept of $H$-convergence for nonlocal operator coefficients. This generalised concept, nonlocal $H$-convergence, yields another characterisation of local $H$-convergence for local operators. In particular, this automatically implies convergence of a different partial differential equation (in an appropriate sense). With this observation in mind we realise that local $H$-convergence of multiplication operators readily implies a homogenisation result for Maxwell's equations with nonperiodic coefficients.  

Yet there is more, the rationale in this article provides a way of identifying the nonlocal coefficients completely from certain solution operators of elliptic partial differential equations. Indeed, the above arguments imply that knowledge of the solution operator associated to 
\[
   \dive a \interior{\grad} u = f, \text{ and }\interior{\curl} a^{-1} \curl v = g
\]
together with all the `fluxes'
\[
  a \interior{\grad} u\text{ and } a^{-1} \curl v
\]
for sufficiently many $f$ and $g$ are needed to uniquely identify $a$. The above arguments also show that in order to use less data than the mentioned ones, more information on $a$ has to be assumed a priori.

\section*{Acknowledgements}
The author is highly indebted to the anonymous referee for their thorough and critical review, which lead to a major improvement of the manuscript.
\bigskip

\bibliographystyle{plain}

\noindent
Marcus Waurick \\Department of Mathematics and Statistics, University of Strathclyde,\\
Livingstone Tower, 26 Richmond Street, Glasgow G1 1XH,\\
Scotland\\
Email: {\tt marcus.wau\rlap{\textcolor{white}{hugo@egon}}rick@strath.\rlap{\textcolor{white}{darmstadt}}ac.uk}

\end{document}